\documentclass{cas-sc}

\makeatletter
\def\ps@pprintTitle{%
  \let\@oddhead\@empty
  \let\@evenhead\@empty
  \def\@oddfoot{}%
  \let\@evenfoot\@oddfoot}
\makeatother

\usepackage[authoryear,longnamesfirst]{natbib}
\usepackage[shortlabels]{enumitem}
\usepackage{amsthm}
\usepackage{mathtools}
\usepackage{bbm}
\usepackage{stmaryrd}
\usepackage{tikz}
\usetikzlibrary{decorations.pathreplacing}
\usepackage{hyperref}
\usepackage[numbered]{bookmark}

\renewcommand\th{{}^\text{th}}
\newcommand\dmu{\ d\mu}
\newcommand\dl{\ d\lambda}
\newcommand\prefix{\!\restriction}
\DeclareMathOperator\dom{dom}

\DeclareMathOperator\Int{int}
\DeclareMathOperator\supp{supp}
\DeclareMathOperator\diam{diam}
\newcommand\zero{\emptyset}
\newcommand\one{\mathbbm 1}
\newcommand\cyl[1]{\left\llbracket #1 \right\rrbracket}
\newcommand\SIGMA{\mathbf{\Sigma}}
\newcommand\cs[1]{\left\langle{#1}\right\rangle}
\newcommand\cl[1]{\overline{#1}}
\newcommand\floor[1]{\left\lfloor #1 \right\rfloor}
\newcommand{\paren}[1]{\hspace{-1pt}\left( {#1} \right)}
\newcommand{\abs}[1]{\left| {#1} \right|}
\newcommand{\set}[1]{\left\{ {#1} \right\}}
\renewcommand\l{\ell}
\newcommand\eps\varepsilon
\newcommand\LA\Leftarrow
\newcommand\RA\Rightarrow
\newcommand\upto{\nearrow}
\newcommand\myimplies[1]{\stackrel{\mathclap{\!\!\text{\footnotesize\mbox{#1}}}}{\implies}}
\newcommand{\downimplies}{\mathbin{\rotatebox[origin=c]{-90}{$\implies$}}}
\newcommand\N{\mathbb N}
\newcommand\Q{\mathbb Q}
\renewcommand\P{\mathbb P}
\newcommand\R{\mathbb R}
\newcommand\B{\mathcal B}
\renewcommand\H{\mathcal H}
\newcommand\PP{\mathcal P}
\newcommand\U{\mathcal U}
\let\oldtextbf=\textbf
\renewcommand\textbf[1]{{\boldmath\oldtextbf{#1}}}
\newtheorem{thm}{Theorem}[section]
\newtheorem{lem}[thm]{Lemma}
\newtheorem{cor}[thm]{Corollary}
\newtheorem{conj}[thm]{Conjecture}
\newtheorem{qn}[thm]{Question}
\theoremstyle{definition}
\newtheorem{defn}[thm]{Definition}

\makeatletter
\newcommand{\proofstep}[1]{%
  \par\vspace{3.25ex plus 1ex minus .2ex}% Space above
  \noindent{\normalsize\itshape #1}\par% Italic text on its own line
  \vspace{1.5ex plus .2ex}% Space below
  \@afterheading% Prevents a page break right after the heading
}
\makeatother

\IfFileExists{latexml.sty}{
    \usepackage{latexml}
}{
    \newif\iflatexml
    \latexmlfalse
}

\iflatexml
  \DeclareMathSymbol{\restriction}{\mathrel}{AMSa}{"16}
  \renewcommand\prefix{\!\restriction}
  \usepackage[capitalize,noabbrev]{cleveref}
  \let\oldcref\cref
  \RenewDocumentCommand{\cref}{o m}{%
    \oldcref{#2}%
  }
  \restriction
\else
  	\usepackage{zref-clever}
  	\let\cref\zcref
  
  	\zcRefTypeSetup{thm}{name-sg=Theorem, name-pl=Theorems}
	\zcRefTypeSetup{lem}{name-sg=Lemma, name-pl=Lemmas}
	\zcRefTypeSetup{cor}{name-sg=Corollary, name-pl=Corollaries}
	\zcRefTypeSetup{conj}{name-sg=Conjecture, name-pl=Conjectures}
	\zcRefTypeSetup{qn}{name-sg=Question, name-pl=Questions}
	\zcRefTypeSetup{defn}{name-sg=Definition, name-pl=Definitions}

	\AddToHook{env/lem/begin}{\zcsetup{countertype={thm=lem}}}
	\AddToHook{env/cor/begin}{\zcsetup{countertype={thm=cor}}}
	\AddToHook{env/conj/begin}{\zcsetup{countertype={thm=conj}}}
	\AddToHook{env/qn/begin}{\zcsetup{countertype={thm=qn}}}
	\AddToHook{env/defn/begin}{\zcsetup{countertype={thm=defn}}}
	
\fi

\begin{document}
\let\WriteBookmarks\relax
\def\floatpagepagefraction{1}
\def\textpagefraction{.001}

\shorttitle{Effective recurrence for computable measure-preserving transformations}
\shortauthors{Joey Veltri}

\title[mode = title]{Effective recurrence for computable measure-preserving transformations}  

\iflatexml
	\author{Joey Veltri}
	\affiliation{Penn State University}
\else
	\author{Joey Veltri}[orcid=0000-0002-6904-9852]
	\ead{jveltri@psu.edu}
	\affiliation{organization={Penn State University},
            addressline={416 McAllister Building, 54 McAllister St}, 
            city={State College},
          citysep={},
            postcode={16801}, 
            state={PA},
            country={United States}}
\fi

\begin{abstract}
	We prove several necessary and sufficient conditions under which a point satisfies the Poincar\'e Recurrence Theorem for all computable (ergodic) measure-preserving transformations and all sets of a particular complexity. The necessary conditions are obtained by constructing specific measure-preserving transformations which violate recurrence.
	
	While some of these conditions pertain to standard notions of algorithmic randomness, others involve new notions of genericity developed in the context of a computable probability space, which we call $\Pi^0_n$-genericity and quasi-$\Pi^0_n$-genericity. We also provide conditions under which points recur at a positive frequency in sets containing them, regarding both simple and multiple recurrence.
\end{abstract}


\maketitle

\section{Introduction}

Let $(X, \B, \mu)$ be a probability space and $T: X \to X$ be a \emph{measure-preserving transformation (m.p.t.)}, i.e., $\mu(T^{-1}(A)) = \mu(A)$ for every set $A$ in the $\sigma$-algebra $\B$.

The Poincar\'e Recurrence Theorem states that for each set $A$ with positive measure, almost every point in $A$ will eventually return to $A$ under the action of $T$. If $T$ is moreover \emph{ergodic}, i.e., $T^{-1}(B) = B$ implies $\mu(B) \in \{0, 1\}$ for any measurable set $B$, then almost every point in $X$ will eventually map into $A$ under the action of $T$. This is one of the most basic results of ergodic theory.

Another is the Birkhoff Ergodic Theorem, which says that in the same setting as above, for any integrable function $f: X \to \R$, the \emph{ergodic averages}
\[\frac{1}{N}\sum_{n < N} f \circ T^n\]
converge almost everywhere as $N \to \infty$. Moreover, this limit will in fact be the constant $\int f \dmu$ when the transformation $T$ is ergodic.

Many such ``almost everywhere'' theorems have been translated to a computability-theoretic setting in which one can identify precisely the full-measure set of points with the desired property. This requires imposing certain computability assumptions on the objects in question and typically guarantees the property for all sufficiently ``random'' points, using the tools of algorithmic randomness to describe exactly what this means.

\subsection{Prior Work}

Over the last few decades, the Birkhoff Ergodic Theorem has been effectivized in various ways, which are listed below. Some parts of these have been proven in the full generality of computable probability spaces, but in order to consider them altogether, we work over $2^\N$ (Cantor space) with the Lebesgue measure $\lambda$.
\begin{enumerate}[(i)]
    \item\label{prior1} \cite{FranklinTowsner2012} showed that if $x$ is Schnorr random, then the ergodic averages converge to $\int f \dl$ at $x$ for every computable ergodic m.p.t.\ $T$ and computable function $f$. \cite{Gacs2011} showed the converse.
    \item\label{prior2} \cite{Bienvenu2012}, as well as \cite{Franklin2012} independently, showed that if $x$ is Martin-L\"of random, then the ergodic averages converge to $\int f \dl$ at $x$ for every computable ergodic m.p.t.\ $T$ and nonnegative lower-semicomputable function $f$. (\cite{Hoyrup2012} slightly generalized this result.) The converse follows from the existence of a universal Martin-L\"of test and the fact that if $x$ is not Martin-L\"of random, then neither is any of its shifts.
    \item \cite{Vyugin1998} showed that if $x$ is Martin-L\"of random, then the ergodic averages converge at $x$ for every computable m.p.t.\ $T$ and computable function $f$. \cite{FranklinTowsner2012} showed the converse.
    \item \cite{Miyabe2016} showed that if $x$ is Oberwolfach random, then the ergodic averages converge at $x$ for every computable m.p.t.\ $T$ and nonnegative lower-semicomputable function $f$. 
\end{enumerate}
From these results, we can see that in order for the ergodic averages to converge for broader and broader classes of transformations (m.p.t.'s to ergodic m.p.t.'s) or functions (computable to lower-semicomputable), a point must be more and more random (Schnorr to Martin-L\"of to Oberwolfach). Therefore, the notion of randomness needed to effectivize the Birkhoff Ergodic Theorem must be calibrated appropriately to whatever class of transformations and functions one considers.

As for the simpler Poincar\'e Recurrence Theorem, surprisingly little has been said about effective versions. One thing we do know is that if $x$ is Martin-L\"of random, then $x$ is recurrent for every computable ergodic m.p.t.\ $T$ and $\Pi^0_1$ set $A$ with positive measure. (To see why, simply take $f$ to be the characteristic function of the complement of $A$, which is lower-semicomputable. Then by \ref{prior2}, the ergodic averages will converge to $\int f \dmu = 1 - \lambda(A) < 1$, so we must have $\paren{f \circ T^n}(x) < 1$ for some $n$, i.e., $T^n(x) \in A$.) The converse also holds by \ref{prior2}.

Similarly, it follows from \ref{prior1} that if $x$ is Schnorr random, then $x$ is recurrent for every computable ergodic m.p.t.\ $T$ and clopen set $A$ with positive measure. (The characteristic function of $A$ will be computable in this case, so we can argue as above.)

This is the extent of the recurrence properties one can conclude from effective versions of the Birkhoff Ergodic Theorem. Note that ergodicity of $T$ is crucial in both cases since they rely on the limit being equal to $\int f \dl$.

\cite{Downey2016} studied effective recurrence in its own right---multiple recurrence, in fact, which refers to several different measure-preserving transformations $T_1, \dots, T_k$ mapping a point into a set after the same number of iterations.

Consider in particular $2^{\N^k}$ with the uniform measure, and take $T_i$ to be the left shift in the $i\th$ coordinate while keeping the other coordinates fixed. In this setting, they showed that Kurtz random points are multiply recurrent for clopen sets, Schnorr random points are multiply recurrent for $\Pi^0_1$ sets with computable measure, and Martin-L\"of random points are multiply recurrent for $\Pi^0_1$ sets (all with positive measure). That the transformations were all shift maps (in particular, ergodic) was a key component of the proofs.

Thus, it seems that there is a significant gap in the literature when it comes to studying effective recurrence for measure-preserving transformations which are not necessarily ergodic. That is the primary goal of this paper, although we will also have more to say in the ergodic setting.

\subsection{Main results}

We attempt as complete an investigation as possible of which conditions are necessary and sufficient to guarantee recurrence for all computable (ergodic) measure-preserving transformations and all sets of a particular complexity.

As it turns out, both randomness and genericity have a role to play when it comes to effective recurrence, due to the importance of the set containing the point. To this end, some new notions of genericity will be developed, which we call (weak) $\Pi^0_n$-genericity and (weak) quasi-$\Pi^0_n$-genericity. (See \cref{pi0n_generic_definition,quasi_pi0n_generic_definition}, along with \cref[S]{genericity_picture}.)

\cref[S]{sufficient_table} displays the sufficient conditions we obtain in order for a point to recur in all sets containing it of a particular complexity for all computable (ergodic) measure-preserving transformations. These hold in any computable probability space. We also show in \cref{CPS_weak_recurrence_frequency} that every weakly $(n + 1)$-random point recurs at a positive frequency in any $\Pi^0_n$ set containing it.

\setlength{\extrarowheight}{1ex}
\begin{table}
    \centering
    \begin{tabular}{c !{\vrule width 1.5pt} c !{\vrule width 1.5pt} c}
        Complexity of the set & Sufficient to recur for & Sufficient to recur for \\
        \textbf{containing} the point & \textbf{all} transformations & \textbf{ergodic} transformations \\[1ex]
        \midrule[1.5pt]
        $\Sigma^0_1$ & \hyperref[CPS_weak_n_Poincare]{Kurtz random} & \hyperref[CPS_weak_n_Poincare_ergodic]{Kurtz random} \\[1ex]
        \hline
        $\Pi^0_n$ with $\zero^{(n - 1)}$- & \hyperref[CPS_Schnorr_n_Poincare]{Schnorr $n$-random} & \hyperref[CPS_Schnorr_n_Poincare_ergodic]{Schnorr $n$-random} \\[1ex]
        computable measure & \textbf{or} \hyperref[CPS_pi0n_generic_Poincare]{$\Pi^0_n$-generic} & \textbf{or} \hyperref[CPS_weak_pi0n_generic_Poincare_ergodic]{weakly $\Pi^0_n$-generic} \\[1ex]
        \hline
         \multirow{4}{*}[-0.1cm]{$\Pi^0_n$} & \hyperref[CPS_weak_n_Poincare]{Weakly $(n + 1)$-random} & \hyperref[CPS_n_Poincare_ergodic]{$n$-random} \\[1ex]
         & \textbf{or (}\hyperref[CPS_Schnorr_n_quasi_pi0n_generic_Poincare]{Schnorr $n$-random} & \textbf{or (}\hyperref[CPS_Schnorr_n_weak_quasi_pi0n_generic_Poincare_ergodic]{Schnorr $n$-random} \\[1ex]
    & \hyperref[CPS_Schnorr_n_quasi_pi0n_generic_Poincare]{\emph{and} quasi-$\Pi^0_n$-generic}\textbf{)} & \hyperref[CPS_Schnorr_n_weak_quasi_pi0n_generic_Poincare_ergodic]{\emph{and} weakly quasi-$\Pi^0_n$-generic}\textbf{)} \\[1ex]
    & \textbf{or} \hyperref[CPS_pi0n_generic_Poincare]{$\Pi^0_n$-generic} & \textbf{or} \hyperref[CPS_weak_pi0n_generic_Poincare_ergodic]{weakly $\Pi^0_n$-generic}\\[1ex]
    \end{tabular}
    \caption{Sufficient conditions for recurrence in any computable probability space}
    \label{sufficient_table}
\end{table}

\cref[S]{necessary_table} gives the necessary conditions we obtain for the same recurrence property when working in Cantor space with the Lebesgue measure. Proving these requires some careful constructions of computable measure-preserving transformations that violate recurrence.

\begin{table}
    \centering
    \begin{tabular}{c !{\vrule width 1.5pt} c}
        Complexity of the set \textbf{containing} the point & Necessary to recur for \textbf{all} transformations \\[1ex]
        \midrule[1.5pt]
        $\Sigma^0_1$ & \hyperref[Kurtz_Poincare_characterization]{Kurtz random} \\[1ex]
        \hline
        $\Pi^0_1$ with computable measure & \hyperref[Schnorr_converse]{Schnorr random \textbf{or} weakly 1-generic}\\[1ex]
    \end{tabular}
    \caption{Necessary conditions for recurrence in $\paren{2^\N, \lambda}$}
    \label{necessary_table}
\end{table}

\cref[S]{characterization_table} shows that when we do not require sets to contain the point, we can characterize when a point recurs for all sets of a particular complexity and all computable ergodic measure-preserving transformations when working in Cantor space with the Lebesgue measure. This characterization of $n$-randomness is already known (see \cref{n_Poincare_ergodic_characterization}), but the others are new.

\begin{table}
    \centering
    \begin{tabular}{c !{\vrule width 1.5pt} c}
        Complexity of the set & Necessary and sufficient to recur \\
        (\textbf{not necessarily containing} the point) & for \textbf{ergodic} transformations \\[1ex]
        \midrule[1.5pt]
        $\Sigma^0_1$ & \hyperref[Kurtz_Poincare_characterization]{Kurtz random} \\[1ex]
        \hline
        $\Pi^0_n$ with $\zero^{(n - 1)}$-computable measure & \hyperref[Schnorr_n_Poincare_ergodic_characterization]{Schnorr $n$-random} \\[1ex]
        \hline
        $\Pi^0_n$ & \hyperref[n_Poincare_ergodic_characterization]{$n$-random}\\[1ex]
    \end{tabular}
    \caption{Characterizations of ergodic recurrence in $\paren{2^\N, \lambda}$}
    \label{characterization_table}
\end{table}

\subsection{Organization}

The paper will be organized as follows.

\cref[S]{section_Kurtz} will contain effective recurrence theorems for weakly $n$-random points, a framework for defining computable measure-preserving transformations on Cantor space, and a converse to these theorems in the case $n = 1$.

\cref[S]{section_Schnorr} will contain effective recurrence theorems for Schnorr $n$-random points, similar results for $\Pi^0_n$-generic points, and a partial converse to these theorems in the case $n = 1$.

\cref[S]{section_other} will contain effective recurrence theorems for $n$-random points, similar results for quasi-$\Pi^0_n$-generic points, generalizations of some previously proven effective versions of the Birkhoff Ergodic Theorem for higher notions of randomness, and a pair of classical and effective results on the frequency of simple and multiple recurrence for weakly $n$-random points.

Within these sections, we refer to many definitions and auxiliary results about computable metric spaces, computable probability spaces, algorithmic randomness, and computable measure-preserving systems. An exposition of these is relegated to \cref[S]{section_background}, where we include proofs of the less standard results needed for our purposes.

Throughout the paper, we will also make note of several other questions that have come to light in the process of obtaining these results, which the author hopes to answer in the future.

\subsection{Notation}

Unless otherwise noted, all sums, unions, and intersections will range over the natural numbers, starting from 0. We also use $\neg A$ to denote the complement of a set $A$, $\PP_{<\N}(S)$ to denote the collection of finite subsets of a set $S$, $\one_A$ to denote the characteristic function of a set $A$, and $\alpha_k \upto \alpha$ as $k \to \infty$ to mean that $\paren{\alpha_k}_k$ is an increasing sequence of numbers/sets/functions that converges to $\alpha$.

When it comes to Cantor space, we use the following terminology and notation:
\begin{itemize}
    \item A \textbf{(binary) string} is an element of $2^{<\N}$, i.e., a finite sequence of 0's and 1's. We use lowercase Greek letters like $\rho$, $\sigma$, $\tau$, etc.\ for strings.
    
    \item The length of a string $\sigma$ is denoted by $|\sigma|$. $2^n$ is the set of all strings with length $n$, and the empty string is identified with the empty set $\zero$.

    \item A \textbf{(binary) sequence} is an element of $2^\N$, i.e., infinite sequences of 0's and 1's. We use lowercase Latin letters like $x$, $y$, $z$, etc.\ for sequences.
    
    \item We write $\sigma \subseteq \tau$ if $\sigma$ is an initial segment of $\tau$ and $\sigma \subset \tau$ if $\sigma$ is a proper initial segment of $\tau$. We also write $\sigma \subset x$ if $\sigma$ is a proper initial segment of a sequence $x$.

    \item The length-$n$ initial segment of a string $\sigma$ or sequence $x$ is denoted by $\sigma \prefix n$ or $x \prefix n$, respectively.

    \item The concatenation of two strings $\sigma$ and $\tau$ or a string $\sigma$ with a sequence $x$ is denoted by $\sigma\tau$ or $\sigma x$, respectively. We also write $0^k$ and $1^k$ denote the $k$-fold concatenation of the bit 0 or 1, respectively.
    
    \item \textbf{Cantor space} is the computable metric space $2^\N$ whose basic open sets are the \textbf{cylinders} $\cyl\sigma := \{x \in 2^\N: \sigma \subset x\}$.

    \item When $W \subseteq 2^{<\N}$, we use $\cyl W$ to denote the open set $\bigcup_{\sigma \in W} \cyl\sigma$.

    \item The \textbf{Lebesgue measure} on Cantor space is the computable Borel probability measure $\lambda$ defined by $\lambda\cyl\sigma = 2^{-|\sigma|}$ for every string $\sigma$.
\end{itemize}

\section[Effective recurrence in Σ0n sets]{Effective recurrence in $\Sigma^0_n$ sets}\label{section_Kurtz}

In this section we seek to describe precisely for which points the Poincar\'e Recurrence Theorem (\cref{poincare_recurrence_theorem}) holds when the set is $\Sigma^0_n$.

\subsection{Recurrence for weakly $n$-random points}

Our first observation is that weakly $n$-random points recur in $\Sigma^0_n$ sets.

\begin{thm}\label{CPS_weak_n_Poincare}
    Let $(X, \mu)$ be a computable probability space and $n \geq 1$. If $x \in X$ is weakly $n$-random, then for any $\Sigma^0_n$ set $A \ni x$ with $\mu(A) > 0$ and any computable measure-preserving transformation $T: \ \subseteq \! X \to X$ with $x \in \dom T$, we have $T^k(x) \in A$ for some $k > 0$.
\end{thm}
\begin{proof}
    Suppose there are a $\Sigma^0_n$ set $A \ni x$ with $\mu(A) > 0$ and a computable measure-preserving transformation $T: \ \subseteq\! X \to X$ with $x \in \dom T =: D$ such that $T^k(x) \notin A$ for all $k > 0$. We have two cases:
    \begin{itemize}
        \item $n = 1$. By \cref{almost_decidable_basis}, let $B \subseteq A$ be an almost decidable set containing $x$ with $\mu(B) > 0$. Take $\Sigma^0_1$ sets $U \subseteq B$ and $V \subseteq \neg B$ such that $U \cup V$ has full measure. Since $\neg U$ is $\Pi^0_1$, we know by \cref{iterated_preimage_of_sigma0n} that there exist uniformly $\Pi^0_1$ sets $P_k$ such that $T^{-k}(\neg U) = D \cap P_k$. The hypotheses then imply
        \[x \in B \cap \bigcap_{k > 0} T^{-k}(\neg B) \subseteq \neg V \cap \bigcap_{k > 0} T^{-k}(\neg U) \subseteq \neg V \cap \bigcap_{k > 0} P_k =: P.\]
        By \cref{almost_decidable_has_computable_measure}, $\neg V$ has the same measure as $B$, and $\neg U$ has the same measure as $\neg B$. Since $T$ is measure-preserving, $T^{-k}(\neg U)$ also has the same measure as $T^{-k}(\neg B)$ for all $k$. Thus, the first inclusion above does not change the measure, and neither does the second since $D$ has full measure. By the Poincar\'e Recurrence Theorem, it follows that the $\Pi^0_1$ set $P$ is null. Thus, $x \in P$ is not Kurtz random.

        \item $n \geq 2$. Let $B \subseteq A$ be a $\Pi^0_{n - 1}$ set containing $x$ with $\mu(B) > 0$. Since $\neg B$ is $\Sigma^0_{n - 1}$ and hence $\Pi^0_n$, we know by \cref{iterated_preimage_of_sigma0n} that there exist uniformly $\Pi^0_n$ sets $P_k$ such that $T^{-k}(\neg B) = D \cap P_k$. The hypotheses then imply
        \[x \in B \cap \bigcap_{k > 0} T^{-k}(\neg B) \subseteq B \cap \bigcap_{k > 0} P_k =: P.\]
        The inclusion above does not change the measure since $D$ has full measure. By the Poincar\'e Recurrence Theorem, it follows that the $\Pi^0_n$ set $P$ is null. Thus, $x \in P$ is not weakly $n$-random. \qedhere
    \end{itemize}
\end{proof}

Observe by \cref{preservation_of_randomness} that a weakly $n$-random point will return to the set $A$ infinitely often, not just once. Also, we do not technically need to assume that $\mu(A) > 0$ here because no weakly $n$-random point would ever lie in a null $\Sigma^0_n$ set. (Such a set would also be $\Sigma^0_{n + 1}$, so it would have a null $\Pi^0_n$ subset containing the point.) Nonetheless, we include it here and elsewhere in order to maintain a parallel structure to those theorems where we do not require the set to contain the point.

In addition, note that the proof above did not make full use of $T$'s computability, only that preimages of $\Pi^0_n$ sets were uniformly $\Pi^0_n$ in the domain (or equivalently, that preimages of $\Sigma^0_n$ sets were uniformly $\Sigma^0_n$ in the domain). Thus, we have the following question, which may depend on the computable probability space $(X, \mu)$.

\begin{qn}
    For $n > 1$, do there exist non-computable measure-preserving transformations $T: \ \subseteq \! X \to X$ under which the preimages of $\Sigma^0_n$ sets are uniformly $\Sigma^0_n$ in the domain?
\end{qn}

By not intersecting with the set containing the point, a similar proof to the one given above would show the following.

\begin{thm}\label{CPS_weak_n_Poincare_ergodic}
    Let $(X, \mu)$ be a computable probability space and $n \geq 1$. If $x \in X$ is weakly $n$-random, then for any $\Sigma^0_n$ set $A$ with $\mu(A) > 0$ and any computable ergodic measure-preserving transformation $T: \ \subseteq \! X \to X$ with $x \in \dom T$, we have $T^k(x) \in A$ for some $k > 0$.
\end{thm}

Note that in the ergodic setting, the set $A$ need not contain the point $x$. We will see later on in \cref[S]{section_Schnorr,section_other} just how crucial a subtlety this is because considering only sets containing the point will in general result in a strictly larger class of recurrent points.

The rest of this section will be devoted to showing that when considering Cantor space with the Lebesgue measure, both of these theorems are optimal for $n = 1$. We will return to the case $n \geq 2$ in \cref[S]{section_other}.

\subsection{Constructing computable measure-preserving transformations}

Proving the converses to \cref{CPS_weak_n_Poincare,CPS_weak_n_Poincare_ergodic} for $n = 1$ will involve building computable measure-preserving transformations with certain properties, a framework for which we now describe in terms of monotone maps. These provide a clean method of moving from strings to sequences when defining a transformation on Cantor space.

\begin{defn}
    Let $\paren{\l_k}_k$ be a strictly increasing sequence in $\N$. A function $t: \bigcup_k 2^{\l_k} \to 2^{<\N}$ is \textbf{monotone} if $\sigma \subseteq \tau$ implies $t(\sigma) \subseteq t(\tau)$ for all $\sigma, \tau \in \bigcup_k 2^{\l_k}$.
\end{defn}

\begin{defn}
    The \textbf{limit} of a monotone map $t: \bigcup_k 2^{\l_k} \to 2^{<\N}$ is the map $T: \ \subseteq \! 2^\N \to 2^\N$ given by
    \[T(x) = \lim_{k \to \infty} t\paren{x \prefix \l_k}\quad\text{whenever}\quad\abs{t\paren{x \prefix \l_k}} \to \infty \text{ as } k \to \infty.\]
\end{defn}

\begin{lem}\label{limit_of_computable_monotone_map}
    Let $\paren{\l_k}_k$ be a strictly increasing computable sequence in $\N$ ,and $t: \bigcup_k 2^{\l_k} \to 2^{<\N}$ be a computable monotone map. Then its limit $T: \ \subseteq \! 2^\N \to 2^\N$ is computable with $\Pi^0_2$ domain.
\end{lem}
\begin{proof}
    Given $\rho \in 2^{<\N}$, define 
    \[W_\rho = \set{\sigma \in \bigcup_k 2^{\l_k}: \rho \subseteq t(\sigma)},\]
    which is uniformly c.e.\ in $\rho$ and satisfies $T^{-1}\cyl\rho = \cyl{W_\rho} \cap \dom T$, so $T$ is computable. Moreover, defining uniformly c.e.\ sets
    \[V_j = \set{\sigma \in \bigcup_k 2^{\l_k}: |t(\sigma)| \geq j},\]
    we see that $\dom T = \bigcap_j \cyl{V_j}$ is $\Pi^0_2$.
\end{proof}

\begin{lem}\label{limit_is_measure_preserving}
    Let $\paren{\l_k}_k$ be a strictly increasing sequence in $\N$, $t: \bigcup_n 2^{\l_k} \to 2^{<\N}$ be a monotone map, and
    \[S_k = \{\sigma \in 2^{\l_k}: |t(\sigma)| = \l_k\}.\]
    Assume the following:
    \begin{itemize}
        \item The sequence $\paren{\cyl{S_k}}_k$ is increasing.
        \item $t$ is injective on $S_k$ for all $k$.
        \item $t(\sigma) = \zero$ for all $\sigma \in 2^{\l_k} \setminus S_k$.
        \item $\lambda\cyl{S_k} \to 1$ as $n \to \infty$.
    \end{itemize}
    Then the limit $T: \ \subseteq \! 2^\N \to 2^\N$ of $t$ is injective and measure-preserving with domain $\bigcup_k \cyl{S_k}$.
\end{lem}
\begin{proof}
    First we show that $\dom T = \bigcup_k \cyl{S_k}$:

    \begin{itemize}
        \item $\subseteq$. If $x \notin \cyl{S_k}$ for all $k$, then $t(x \prefix \l_k) = \zero$ for all $k$, which implies that $x \notin \dom T$.

        \item $\supseteq$. If $x \in \cyl{S_k}$ for some $k$, then because $\paren{\cyl{S_k}}_k$ is increasing, for all sufficiently large $k$ we have $x \in \cyl{S_k}$ and hence $|t(x \prefix \l_k)| = \l_k \to \infty$, i.e., $x \in \dom T$.
    \end{itemize}

    Next, if $x, y \in \dom T$ are distinct, then taking $k$ large enough so that $x, y \in \cyl{S_k}$, we have $t\paren{x \prefix \l_k} \neq t\paren{y \prefix \l_k}$ and hence $T(x) \neq T(y)$. Thus, $T$ is injective.

    Finally, we will show that $T$ is measure-preserving. Consider any string $\rho$. If $\rho = \zero$, then $T^{-1}\cyl\rho = \dom T$ and hence trivially $\lambda\paren{T^{-1}\cyl\rho} = 1 = \lambda\cyl\rho$.
    
    Now assume $\rho \neq \zero$. Define
    \[W_{\rho, k} = \set{\sigma \in 2^{\l_k}: \rho \subseteq t(\sigma)}.\]
    This is the same as $\set{\sigma \in S_k: \rho \subseteq t(\sigma)}$ since $t(\sigma) = \zero \not\supseteq \rho$ for all $\sigma \in 2^{\l_k} \setminus S_k$. Thus, $\cyl{W_{\rho, k}} \subseteq \dom T$ for all $k$, and in fact,
    \[T^{-1}\cyl\rho = \bigcup_k \cyl{W_{\rho, k}}.\]
    Given that $\paren{\cyl{W_{\rho, k}}}_k$ is increasing by monotonicity of $t$, it follows that
    \[\lambda\paren{T^{-1}\cyl\rho} = \lim_{k \to \infty} \lambda\cyl{W_{\rho, k}}.\]
    Now for each $k$, since $t$ injectively maps $S_k$ into $2^{\l_k}$ and $\rho \subseteq \tau$ for every $\tau \in t\paren{W_{\rho, k}}$, we have
    \[\lambda\cyl{W_{\rho, k}} = 2^{-\l_k}\abs{W_{\rho, k}} = 2^{-\l_k}\abs{t\paren{W_{\rho, k}}} = \lambda\cyl{t\paren{W_{\rho, k}}} \leq \lambda\cyl\rho.\]
     Thus, $\lambda\paren{T^{-1}\cyl\rho} \leq \lambda\cyl\rho$.
    
    To get the other inequality, note that
    \[S_k \setminus W_{\rho, k} = \bigcup_{\tau \in 2^{|\rho|} \setminus \{\rho\}} W_{\tau, k}.\]
    Therefore, we have
    \[\lambda\cyl{S_k} - \lambda\cyl{W_{\rho, k}} = \lambda\cyl{S_k \setminus W_{\rho, k}} \leq \sum_{\tau \in 2^{|\rho|} \setminus \{\rho\}} \lambda\cyl{W_{\tau, k}} \leq \sum_{\tau \in 2^{|\rho|} \setminus \{\rho\}} \lambda\cyl\tau = 1 - \lambda\cyl\rho,\]
    and since $\lambda\cyl{S_k} \to 1$ as $k \to \infty$, we get
    \[\lambda\paren{T^{-1}\cyl\rho} = \lim_{k \to \infty} \lambda\cyl{W_{\rho, k}} \geq \lim_{k \to \infty}\lambda\cyl{S_k} - 1 + \lambda\cyl\rho = \lambda\cyl\rho.\]
    Consequently, $\lambda\paren{T^{-1}\cyl\rho} = \lambda\cyl\rho$ for all strings $\rho$. As the cylinders $\cyl\rho$ generate the Borel $\sigma$-algebra on $2^\N$, it follows that $T$ is measure-preserving.
\end{proof}

\subsection{Characterizing Kurtz randomness}

Working in Cantor space with the Lebesgue measure, we now characterize Kurtz randomness in terms of recurrence in clopen sets. This serves as a converse to \cref{CPS_weak_n_Poincare,CPS_weak_n_Poincare_ergodic} in the case $n = 1$.

\begin{thm}
    Consider Cantor space with the Lebesgue measure $\lambda$. Then the following are equivalent for a point $x \in 2^\N$:
    \begin{enumerate}[(a)]
        \item $x$ is Kurtz random.
        \item For every clopen set $A \ni x$ with $\lambda(A) > 0$ and every computable measure-preserving transformation $T: \ \subseteq \! 2^\N \to 2^\N$, we have $T^k(x) \in A$ for some $k > 0$.
        \item For every clopen set $A$ with $\lambda(A) > 0$ and every computable ergodic measure-preserving transformation $T: \ \subseteq \! 2^\N \to 2^\N$, we have $T^k(x) \in A$ for some $k > 0$.
        \item For every clopen set $A \ni x$ with $\lambda(A) > 0$ and every computable ergodic measure-preserving transformation $T: \ \subseteq \! 2^\N \to 2^\N$, we have $T^k(x) \in A$ for some $k > 0$.
    \end{enumerate}
\end{thm}
\begin{proof}
    Implications (a) $\RA$ (b) and (a) $\RA$ (c) come from \cref{CPS_weak_n_Poincare,CPS_weak_n_Poincare_ergodic} when $n = 1$. Implications (b) $\RA$ (d) and (c) $\RA$ (d) are also clear since they merely restrict to ergodic transformations and to sets containing the point, respectively.
    
    Thus, it suffices to prove that (d) $\RA$ (a), which we accomplish in the following theorem. (Note that we cannot simply take the shift map because there exist computable---and hence non-Kurtz-random---sequences which contain every possible string within them. Thus, their orbits under the shift would be dense, intersecting every nonempty clopen set.)
\end{proof}

\begin{thm}\label{Kurtz_Poincare_characterization}
    Consider Cantor space with the Lebesgue measure $\lambda$. Let $x \in 2^\N$ not be Kurtz random. Then there are a clopen set $A \ni x$ with $\lambda(A) > 0$ and an injective computable ergodic measure-preserving transformation $T: \ \subseteq\! 2^\N \to 2^\N$ with $\dom T$ a $\Sigma^0_1$ set containing $A$ such that $T^k(x) \notin A$ for all $k > 0$.
\end{thm}
\begin{proof}
    The proof is organized in several steps.

    \proofstep{Step 1: Sketching the construction}

    Before diving into technical details, we give an intuitive idea of how the construction will work. Suppose $01 \subset x$, in which case we want to construct $T$ so that $1 \subset T^k(x)$ for all $k > 0$. We achieve this via a standard technique in ergodic theory called ``cutting and stacking'' to build what are known as rank-one maps. Introductions to this sort of construction are given in section 5 of \cite{Chen} and section 6 of \cite{Friedman}.
    
    The basic idea is that at each stage, we divide Cantor space into several ``stacks'' of cylinders. Each cylinder in a given stack is mapped by prefix replacement to the cylinder above it in the stack, except for the topmost cylinder, whose image is not yet determined. Then we ``cut up'' the stacks (so that their constituent cylinders are given by longer strings) and map some of the tops of these stacks to the bases of other stacks. This creates new stacks to manage during the next stage, which are now taller and narrower than before. By repeating this process, we end up with an ergodic measure-preserving transformation.

    In our case, we begin by dividing Cantor space into three groups:
    \begin{itemize}
        \item Group I consists of the cylinder $\cyl{00}$.
        \item Group II consists of a stack with the cylinder $\cyl{01}$ at the base and $\cyl{10}$ on top. (This means that $\cyl{01}$ is mapped via prefix replacement to $\cyl{10}$.)
        \item Group III consists of the cylinder $\cyl{11}$.
    \end{itemize}
    At this point, we do not yet decide where anything else will map. This completes the setup, as illustrated in \cref[S]{Kurtz_setup}.

    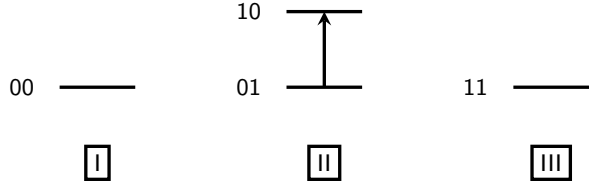
\begin{figure}
        \centering
        \begin{tikzpicture}[line width=1.25pt]
            \node[draw] at (0.5,0) {I};
            \node[draw] at (3.5,0) {II};
            \node[draw] at (6.5,0) {III};
            
            % Group I
            
            \node at (-0.5,1) {00};
            \draw (0,1) -- (1,1);
            
            % Group II
            
            \node at (2.5,1) {01};
            \draw (3,1) -- (4,1);
            
            \draw[-stealth] (3.5, 1) -- (3.5, 2);
            
            \node at (2.5,2) {10};
            \draw (3,2) -- (4,2);
            
            % Group III
            
            \node at (5.5,1) {11};
            \draw (6,1) -- (7,1);
        \end{tikzpicture}
        \caption{The initial setup of the transformation}
        \label{Kurtz_setup}
    \end{figure}
    
    Since $x$ is not Kurtz random, there exists a computable sequence $\paren{F_k}_k$ of finite sets of strings of some common length $\l_k$ such that $x \in \cyl{F_k}$ and $\lambda\cyl{F_k} = 2^{-k}$ for all $k$. Assume that $F_2 = \{01\}$. Then at stage $k \geq 2$, Cantor space is partitioned into three groups:
    \begin{itemize}
        \item Group I consists of a single stack of some large height $h_k$ with base $\cyl{0^{\l_k}}$.
        \item Group II consists of several stacks of height $k$ with bases coming from $\cyl{F_k}$.
        \item Group III consists of the single cylinder $\cyl{1^k}$.
    \end{itemize}

    In moving from stage $k$ to stage $k + 1$, we do the following, as illustrated in \cref[S]{Kurtz_next_stage}:
    \begin{enumerate}
        \item Refine the stacks in Groups I and II so that each cylinder corresponds to a string of length $\l_{k + 1}$.

        \item For each refined stack in Group I except the last, map its top to the bottom of the next stack. For the last of these stacks, map its top to the bottom of a refined stack in Group II with base coming from $\cyl{F_k} \setminus \cyl{F_{k + 1}}$. Then map the top of this stack to the bottom of the next one, and so on until we have one tall stack of height $h_{k + 1}$ consisting of everything that was originally in Group I, along with everything that was originally in Group II whose base came from $\cyl{F_k} \setminus \cyl{F_{k + 1}}$.
        
        \item Map the tops of the stacks in Group II with bases coming from $\cyl{F_{k + 1}}$ to equal-measure portions of $\cyl{1^k0}$ from Group III. Leave $\cyl{1^{k + 1}}$ from Group III alone for now.
    \end{enumerate}

    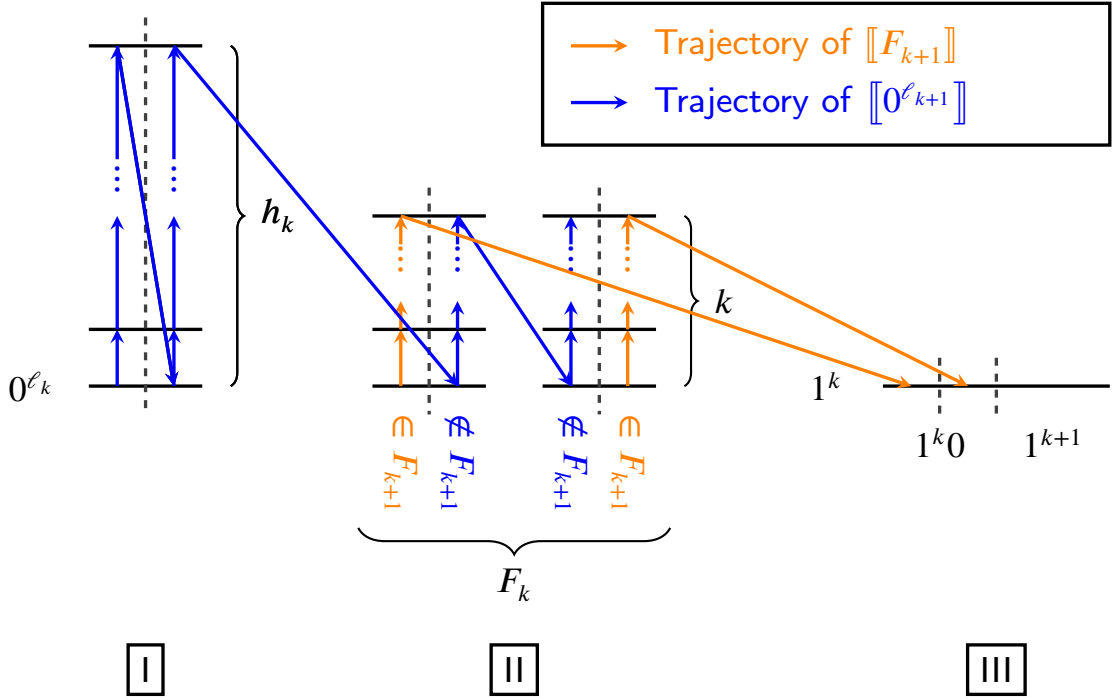
\begin{figure}
        \centering
        \begin{tikzpicture}[line width=1.25pt, scale=1.5, transform shape]
            \node[draw] at (0.5,0) {I};
            \node[draw] at (3.75,0) {II};
            \node[draw] at (8,0) {III};
            
            % Group I
            
            % Bottom
            \node at (-0.5,2.5) {$0^{\ell_k}$};
            \draw (0,2.5) -- (1,2.5);
            
            \draw[-stealth, blue] (0.25, 2.5) -- (0.25, 3);
            \draw[-stealth, blue] (0.75, 2.5) -- (0.75, 3);
            
            \draw (0,3) -- (1,3);
            
            % ...
            
            \draw[-stealth, blue] (0.25, 3) -- (0.25, 4);
            \draw[-stealth, blue] (0.75, 3) -- (0.75, 4);
            
            \node[blue] at (0.25,4.33) {$\vdots$};
            \node[blue] at (0.75,4.33) {$\vdots$};
            
            \draw[-stealth, blue] (0.25, 4.5) -- (0.25, 5.5);
            \draw[-stealth, blue] (0.75, 4.5) -- (0.75, 5.5);
            
            \draw (0,5.5) -- (1,5.5);
            
            % Other
            
            \draw[dashed, darkgray] (0.5,5.75) -- (0.5,2.25);
            
            \draw [thick, decorate, decoration={brace, amplitude=5pt}] (1.25,5.5) -- (1.25,2.5) node [midway, right=4pt] {$h_k$};
            
            % Group II
            
            % F_k
            
            \node at (3.75,0.75) {$F_k$};
            \draw [thick, decorate, decoration={brace, amplitude=10pt}] (5.125,1.25) -- (2.375,1.25);
            
            % F_k+1 and F_k\F_k+1
            
              % Stack 1
            
              \node[rotate=-90, orange] at (2.75,1.8125) {$\in F_{k + 1}$};
            
              \node[rotate=-90, blue] at (3.25,1.8125) {$\notin F_{k + 1}$};
            
              % Stack 2    
            
              \node[rotate=-90, blue] at (4.25,1.8125) {$\notin F_{k + 1}$};
            
              \node[rotate=-90, orange] at (4.75,1.8125) {$\in F_{k + 1}$};
            
            % Bottom
            
              % Stack 1    
            
              \draw (2.5,2.5) -- (3.5,2.5);
            
              \draw[-stealth, orange] (2.75, 2.5) -- (2.75, 3);
              \draw[-stealth, blue] (3.25, 2.5) -- (3.25, 3);
            
              \draw (2.5,3) -- (3.5,3);
            
              % Stack 2
            
              \draw (4,2.5) -- (5,2.5);
            
              \draw[-stealth, blue] (4.25, 2.5) -- (4.25, 3);
              \draw[-stealth, orange] (4.75, 2.5) -- (4.75, 3);
            
              \draw (4,3) -- (5,3);
            
            % ...
            
              % Stack 1
            
              \draw[-stealth, orange] (2.75, 3) -- (2.75, 3.25);
              \node[orange] at (2.75,3.625) {$\vdots$};
              \draw[-stealth, orange] (2.75, 3.75) -- (2.75, 4);
            
              \draw[-stealth, blue] (3.25, 3) -- (3.25, 3.25);
              \node[blue] at (3.25,3.625) {$\vdots$};
              \draw[-stealth, blue] (3.25, 3.75) -- (3.25, 4);
            
              \draw (2.5,4) -- (3.5,4);
            
              % Stack 2
            
              \draw[-stealth, blue] (4.25, 3) -- (4.25, 3.25);
              \node[blue] at (4.25,3.625) {$\vdots$};
              \draw[-stealth, blue] (4.25, 3.75) -- (4.25, 4);
            
              \draw[-stealth, orange] (4.75, 3) -- (4.75, 3.25);
              \node[orange] at (4.75,3.625) {$\vdots$};
              \draw[-stealth, orange] (4.75, 3.75) -- (4.75, 4);
            
              \draw (4,4) -- (5,4);
            
            % Other
            
            \draw[dashed, darkgray] (3,4.25) -- (3,2.25);
            \draw[dashed, darkgray] (4.5,4.25) -- (4.5,2.25);
            
            \draw[-stealth] (0.25, 5.5) -- (0.75, 2.5);
            
            \draw [thick, decorate, decoration={brace, amplitude=5pt}] (1.25,5.5) -- (1.25,2.5) node [midway, right=4pt] {$h_k$};
            
            \draw [thick, decorate, decoration={brace, amplitude=5pt}] (5.25,4) -- (5.25,2.5) node [midway, right=4pt] {$k$};
            
            % Group III
            
            \node at (6.5,2.5) {$1^k$};
            \draw (7,2.5) -- (9,2.5);
            
            \node at (7.5,2) {$1^k0$};
            \node at (8.5,2) {$1^{k + 1}$};

            \draw[dashed, darkgray] (7.5,2.75) -- (7.5,2.25);
            \draw[dashed, darkgray] (8,2.75) -- (8,2.25);
            
            % Stacking up I
            
            \draw[-stealth, blue] (0.25, 5.5) -- (0.75, 2.5);
            
            % Putting II on top of I
            
            \draw[-stealth, blue] (0.75, 5.5) -- (3.25, 2.5);
            \draw[-stealth, blue] (3.25, 4) -- (4.25, 2.5);
            
            % Putting III on top of II
            
            \draw[-stealth, orange] (2.75, 4) -- (7.25, 2.5);
            \draw[-stealth, orange] (4.75, 4) -- (7.75, 2.5);
            
            % Legend
            
            \draw[-stealth, blue] (4.25, 5) -- (4.75, 5) node [midway, right=10pt] {Trajectory of $\cyl{0^{\l_{k + 1}}}$};
            \draw[-stealth, orange] (4.25, 5.5) -- (4.75, 5.5) node [midway, right=10pt] {Trajectory of $\cyl{F_{k + 1}}$};
            
            \draw[draw=black] (4,4.625) rectangle ++(5,1.25);
        \end{tikzpicture}
        \caption{Refining the transformation from stage $k$ to stage $k + 1$. In this example, there are two stacks in Group II at stage $k$, and $\l_{k + 1} = \l_k + 1$. In general, there may be many stacks in Group II, and $\l_{k + 1}$ may be much larger than $\l_k$.}
        \label{Kurtz_next_stage}
    \end{figure}

    Observe that with each stage of the construction, another iterate of $x$ under the map $T$ is guaranteed to avoid $\cyl 0$. Moreover, almost all points will eventually enter Group I, the cutting-and-stacking process behind which ensures ergodicity of the transformation.

    With this picture in mind, we will now proceed to flesh out the details of the construction with full rigor.

    \proofstep{Step 2: Reducing to the case $01 \subset x$}
    
    Purely for ease of notation, we will first reduce to the case where $01 \subset x$. To see why we can do this, assume the theorem holds when $x$ begins with the bits 01. To handle the general case, we define $x' = 01x_2x_3\dots$, which is simply $x$ but with the first two bits replaced with 01. Then $x'$ is likewise not Kurtz-random, so we may take $A \ni x'$ and $T$ as in the theorem such that $T^k\paren{x'} \notin A$ for all $k > 0$.
    
    Now define $S: 2^\N \to 2^\N$ by $S(01y) = x_0x_1y$ on $\cyl{01}$, $S(x_0x_1y) = 01y$ on $\cyl{x_0x_1}$, and $S(y) = y$ elsewhere. Note that $S = S^{-1}$ is a computable measure-preserving transformation.
    
    Then $S(A)$ is a clopen set containing $S\paren{x'} = x$, and $STS: \ \subseteq \! 2^\N \to 2^\N$ is an injective computable measure-preserving transformation with $\dom STS = S(\dom T)$ a $\Sigma^0_1$ set containing $S(A)$ such that $(STS)^k(x) = ST^kS(x) = ST^k\paren{x'} \notin S(A)$ for all $n > 0$. It is also ergodic since
    \begin{align*}
        ST^{-1}S(B) = (STS)^{-1}(B) = B \quad&\implies\quad T^{-1}S(B) = S(B) \\
        &\implies\quad \lambda(S(B)) \in \{0, 1\} \quad\implies\quad \lambda(B) \in \{0, 1\}
    \end{align*}
    for any Borel set $B$, where we use that $T$ is ergodic and $S = S^{-1}$ is measure-preserving. Thus, we may indeed assume without loss of generality that $01 \subset x$.

    \proofstep{Step 3: Describing the null set containing $x$}
    
    Since $x$ is not Kurtz random, there is a null $\Pi^0_1$ set $P \ni x$, which can be written as $\bigcap_k C_k$ for some nested sequence $\paren{C_k}_k$ of uniformly clopen sets. By uniformly clopen, we mean that there is a computable sequence $\paren{F_k}_k$ in $\PP_{<\N}\paren{2^{<\N}}$ such that $C_k = \cyl{F_k}$ for all $k$. We make the following assumptions without loss of generality:
    \begin{itemize}
        \item $\lambda\paren{C_k} = 2^{-k}$ for all $k$.
        \item All strings in $F_k$ have the same length $\l_k$, where $\paren{\l_k}_k$ is a strictly increasing computable sequence.
        \item For $k \leq 2$, we have $C_k = \cyl{x \prefix k}$. That is, $C_0 = 2^\N$, $C_1 = \cyl 0$, and $C_2 = \cyl{01}$.
    \end{itemize}
    This has the following consequences:
    \begin{itemize}
        \item Any string in $F_{k + 1}$ has a proper initial segment in $F_k$.

        To see why, consider any $\sigma \in F_{k + 1}$. Since $\sigma 0^\infty \in \cyl{F_{k + 1}} \subseteq \cyl{F_k}$, the sequence $\sigma 0^\infty$ must have an initial segment $\tau \in F_k$. Moreover, given that $|\tau| = \l_k < \l_{k + 1} \leq |\sigma|$, we must have $\tau \subset \sigma$.

        \item $\abs{F_k} = 2^{\l_k - k}$.

        This holds because
        \[\abs{F_k}2^{-\l_k} = \sum_{\sigma \in F_k} 2^{-|\sigma|} = \lambda\paren{C_k} = 2^{-k}.\]
    \end{itemize}
    \proofstep{Step 4: Setting up the notation}

    Before proceeding, we establish the following notation:
    \begin{itemize}[listparindent=1.5em]
        \item $F'_{k + 1} = \set{\rho\tau \in 2^{\l_{k + 1}} \setminus F_{k + 1}: \rho \in F_k}$.

        Note that $\abs{F'_{k + 1}} = \abs{F_{k + 1}}$ since $\cyl{F_{k + 1}}$ and $\cyl{F'_{k + 1}}$ form a partition of $\cyl{F_k}$, whose measure is twice as large as $\cyl{F_{k + 1}}$.
        
        \item $\rho_n\tau_n$ is the $n\th$ string lexicographically in $F'_{k + 1}$ with $\rho_n \in F_k$, indexing from 0.
        
        \item $\theta_n$ is the $n\th$ string lexicographically of length $\l_{k + 1} - \l_k$, indexing from 0.

        \item $h_k = 2^{\l_k} - (k + 1)\abs{F_k}$. This will represent the height at stage $k$ of the Group I stack discussed in the sketch. This sequence satisfies the recurrence relation
        \[h_{k + 1} = d_kh_k + k\abs{F_{k + 1}}, \quad\text{where}\quad d_k = 2^{\l_{k + 1} - \l_k}, \tag{$*$}\label{hk_recurrence}\]
        since
        \begin{align*}
            h_{k + 1} - k\abs{F_{k + 1}} &= 2^{\l_{k + 1}} - (k + 2)\abs{F_{k + 1}} - k\abs{F_{k + 1}} \\
            &= 2^{\l_{k + 1}} - (2k + 2)\abs{F_{k + 1}} \\
            &= 2^{\l_{k + 1}} - (2k + 2)2^{\l_{k + 1} - (k + 1)} \\
            &= 2^{\l_{k + 1}} - (k + 1)2^{\l_{k + 1} - k} \\
            &= 2^{\l_{k + 1} - \l_k}\paren{2^{\l_k} - (k + 1)2^{\l_k - k}} \\
            &= 2^{\l_{k + 1} - \l_k}\paren{2^{\l_k} - (k + 1)\abs{F_k}} \\
            &= d_kh_k.
        \end{align*}
        Note that $h_0 = h_1 = 0$ while $h_k > 0$ for all $k \geq 2$. Also, $h_{k + 1} > d_kh_k$ for all $k \geq 1$.
    \end{itemize}

    \proofstep{Step 5: Stating the conditions for recursion}

    Our transformation $T: \ \subseteq \! 2^\N \to 2^\N$ will be constructed as the limit of a computable monotone function $t: \bigcup_k 2^{\l_k} \to 2^{<\N}$. The $k$-fold iterate of $t$ (wherever it is defined) is denoted by $t^k$, and $t^0$ is the identity.

    We define $t$ recursively, starting with $t(\sigma) = \zero$ for each string $\sigma$ of length $\l_0$. Now suppose for some $k$ that $t(\sigma)$ has been defined for all strings $\sigma \in \bigcup_{j \leq k} 2^{\l_j}$ and that the following conditions hold (which are all trivial when $k = 0$):
    \begin{enumerate}[label=(\Alph*), ref=\Alph*]
        \item\label{monotone} (Monotonicity) $t\paren{\sigma \prefix \l_j} \subseteq t(\sigma)$ whenever $|\sigma| = \l_k$ and $j < k$.

        \item\label{partition} (Partitioning Cantor space into stacks) Every string $\sigma$ of length $\l_k$ falls into one of the following cases:
        \begin{enumerate}[label=\Roman*., ref=\Roman*]
            \item\label{I} $\sigma = t^j\paren{0^{\l_k}}$ for some $j < h_k$, and $1^{k - 2}0 \subseteq \sigma$ if $j = h_k - 1$.
            \item\label{II} $\sigma = t^j(\rho)$ for some $\rho \in F_k$ and $j < k$, and $1^{k - 1}0 \subseteq \sigma$ if $j = k - 1$.
            \item\label{III} $\sigma = 1^k\eta$ for some $\eta$ with length $\l_k - k$.
        \end{enumerate}
        
        \item\label{top_unmapped} (Leaving the top of each stack unmapped) $t$ maps every string $\sigma$ of the following form to $\zero$:
        \begin{enumerate}[label=\roman*., ref=(\theenumi.\roman*)]
            \item\label{top_i} $\sigma = t^{h_k - 1}\paren{0^{\l_k}}$, where $h_k > 0$.
            \item\label{top_ii} $\sigma = t^{k - 1}(\rho)$ for some $\rho \in F_k$, where $k > 0$.
            \item\label{top_iii} $\sigma = 1^k\eta$ for some $\eta$ with length $\l_k - k$.
        \end{enumerate}
    \end{enumerate}

    By condition (\ref{partition}), observe that all strings of the form \ref{I}, \ref{II}, or \ref{III} are unique because
    \[2^{\l_k} \leq \#(\text{I}) + \#(\text{II}) + \#(\text{III}) \leq h_k + k\abs{F_k} + 2^{\l_k - k} = h_k + (k + 1)\abs{F_k} = 2^{\l_k}.\]
    \proofstep{Step 6: Defining $t$ on strings of length $\l_{k + 1}$}

    Consider any string $\sigma$ with length $\l_{k + 1}$. We define $t(\sigma)$ according to the form (\ref{I}, \ref{II}, or \ref{III}) of $\sigma \prefix \l_k$, in each case verifying the monotonicity condition (\ref{monotone}) for $k + 1$, i.e.,  $t\paren{\sigma \prefix \l_k} \subseteq t(\sigma)$. 

    \begin{enumerate}[label=\Roman*., ref=\Roman*]
        \item $\sigma \prefix \l_k = t^j\paren{0^{\l_k}}$ for some $j < h_k$.

        Take $\theta$ with length $\l_{k + 1} - \l_k$ such that $\sigma = t^j\paren{0^{\l_k}}\theta$. We have two subcases:

        \begin{enumerate}[listparindent=1.5em, label=\alph*., ref=\theenumi.\alph*]
            \item\label{I_not_top} $j < h_k - 1$. \emph{(This means that $\sigma$ lies in one of the refined Group I stacks but is not at the top. We simply map $\sigma$ to the string one level up.)}

            Define $t(\sigma) = t^{j + 1}\paren{0^{\l_k}}\theta$. Then $t\paren{\sigma \prefix \l_k} = t^{j + 1}\paren{0^{\l_k}} \subseteq t^{j + 1}\paren{0^{\l_k}}\theta = t(\sigma)$.

            \item\label{I_top} $j = h_k - 1$. \emph{(This means that $\sigma$ lies at the top of one of the refined Group I stacks. If it is not the last of these stacks, we map $\sigma$ to the bottom of the next stack. Otherwise, we map $\sigma$ to the bottom of the first refined Group II stack with base coming from $F'_{k + 1}$.)}
            
            Take $n$ such that $\theta = \theta_n$ and define
            \[t(\sigma) = \begin{cases}
                0^{\l_k}\theta_{n + 1} & \text{if } n < d_k - 1 \\
                \rho_0\tau_0 & \text{if } n = d_k - 1
            \end{cases}.\]
            Then we have $t(\sigma \prefix \l_k) = t^{h_k}\paren{0^{\l_k}} = \zero \subseteq t(\sigma)$ by condition \ref{top_i}.
        \end{enumerate}
        
        \item $\sigma \prefix \l_k = t^j(\rho)$ for some $\rho \in F_k$ and $j < k$.

        Take $\tau$ with length $\l_{k + 1} - \l_k$ such that $\sigma = t^j(\rho)\tau$. We have three subcases:
        \begin{enumerate}[listparindent=1.5em, label=\alph*., ref=\theenumi.\alph*]
            \item\label{II_not_top} $j < k - 1$. \emph{(This means that $\sigma$ lies in one of the refined Group II stacks but is not at the top. We simply map $\sigma$ to the string one level up.)}

            Define $t(\sigma) = t^{j + 1}(\rho)\tau$. Then $t\paren{\sigma \prefix \l_k} = t^{j + 1}(\rho) \subseteq t^{j + 1}(\rho)\tau = t(\sigma)$.

            \item\label{II_top_continue} $j = k - 1$ and $\rho\tau \in F_{k + 1}$. \emph{(This means that $\sigma$ lies at the top of one of the refined Group II stacks with base coming from $F_{k + 1}$. We map $\sigma$ to an appropriate portion of $1^k0$ from Group III.)}
        
            Take $n$ such that $\rho\tau$ is the $n\th$ string lexicographically in $F_{k + 1}$, and define $t(\sigma)$ to be the $n\th$ string lexicographically of length $\l_{k + 1}$ extending $1^k0$.
            
            This is well-defined because there are exactly $2^{\l_{k + 1} - (k + 1)} = \abs{F_{k + 1}}$ strings of length $\l_{k + 1}$ extending $1^k0$. Also, $t\paren{\sigma \prefix \l_k} = t^k(\rho) = \zero \subseteq t(\sigma)$ by condition \ref{top_ii}.

            \item\label{II_top_stop} $j = k - 1$ and $\rho\tau \notin F_{k + 1}$. \emph{(This means that $\sigma$ lies at the top of one of the refined Group II stacks with base coming from $F'_{k + 1}$. If it is not the last of these stacks, we map $\sigma$ to the bottom of the next stack. Otherwise, we leave $\sigma$ unmapped because it will be at the top of the new Group I stack.)}

        Take $n$ such that $\rho\tau = \rho_n\tau_n$ and define
        \[t(\sigma) = \begin{cases}
            \rho_{n + 1}\tau_{n + 1} & \text{if } n < \abs{F_{k + 1}} - 1 \\
            \zero & \text{if } n = \abs{F_{k + 1}} - 1
        \end{cases}.\]
        Then we have $t\paren{\sigma \prefix \l_k} = t^k(\rho) = \zero \subseteq t(\sigma)$ by condition \ref{top_ii}.
            
        \end{enumerate}

        \item\label{III_top} $\sigma \prefix \l_k = 1^k\eta$ for some $\eta$ with length $\l_k - k$. \emph{(This means that $\sigma$ comes from Group III. We leave $\sigma$ unmapped for now because it will either remain in Group III or be at the top of a stack in the new Group II.)}

        Define $t(\sigma) = \zero$. Then $t\paren{\sigma \prefix \l_k} = t\paren{1^k\eta} = \zero = t(\sigma)$ by condition \ref{top_iii}.
    \end{enumerate}

    \proofstep{Step 7: Verifying where strings end up after many iterates}

    The following facts regarding the definition of $t$ on strings of length $\l_{k + 1}$ will help us complete the induction. The first shows that the iterates of $t$ in the refined Group II stacks merely extend the iterates of $t$ in the old Group II stacks as intended.

    \begin{lem}\label{II_refined}
        $t^j(\rho\tau) = t^j(\rho)\tau$ whenever $\rho \in F_k$, $|\tau| = \l_{k + 1} - \l_k$, and $j < k$.
    \end{lem}
    \begin{proof}
        Suppose this holds for some $j < k$, which it trivially does for $j = 0$. Then if $j < k - 1$, $t^j(\rho\tau) = t^j(\rho)\tau$ falls into case \ref{II_not_top} above, and we have
        \begin{align*}
            t^{j + 1}(\rho\tau) &= t\paren{t^j(\rho)\tau} = t^{j + 1}(\rho)\tau. \qedhere
        \end{align*}
    \end{proof}

    The next fact demonstrates how the old Group I stack with base $0^{\l_k}$ was cut up and stacked, now with base $0^{\l_{k + 1}}$.

    \begin{lem}\label{I_refined}
        $t^{nh_k + j}\paren{0^{\l_{k + 1}}} = t^j\paren{0^{\l_k}}\theta_n$ whenever $n < d_k$ and $j < h_k$.
    \end{lem}
    \begin{proof}
        Suppose this holds for some $n < d_k$ and $j < h_k$, which it trivially does for $n = j = 0$ since $\theta_0 = 0^{\l_{k + 1} - \l_k}$. We must consider two cases:
        \begin{itemize}[listparindent=1.5em]
            \item $j < h_k - 1$.

            Then $t^{nh_k + j}\paren{0^{\l_{k + 1}}} = t^j\paren{0^{\l_k}}\theta_n$ falls into case \ref{I_not_top} above, and we have
            \[t^{nh_k + (j + 1)}\paren{0^{\l_{k + 1}}} = t\paren{t^j\paren{0^{\l_k}}\theta_n} = t^{j + 1}\paren{0^{\l_k}}\theta_n.\]
            Thus, the claim holds for $n$ and $j + 1$.
            \item $n < d_k - 1$ and $j = h_k - 1$.

            Then $t^{nh_k + j}\paren{0^{\l_{k + 1}}} = t^{h_k - 1}\paren{0^{\l_k}}\theta_n$ falls into case \ref{I_top} above, and we have
            \[t^{(n + 1)h_k}\paren{0^{\l_{k + 1}}} = t\paren{t^{h_k - 1}\paren{0^{\l_k}}\theta_n} = 0^{\l_k}\theta_{n + 1}.\]
            Thus, the claim holds for $n + 1$ and $0$. \qedhere
        \end{itemize}
    \end{proof}

    The next fact demonstrates how the portion of the old Group II stacks with bases coming from $F'_{k + 1}$ were cut up and stacked.

    \begin{lem}\label{II_rejected}
        $t^{nk + j}\paren{\rho_0\tau_0} = t^j\paren{\rho_n}\tau_n$ whenever $n < \abs{F_{k + 1}}$ and $j < k$.
    \end{lem}
    \begin{proof}
        Suppose this holds for some $n < \abs{F_{k + 1}}$ and $j < k$, which it trivially does for $n = j = 0$. We must consider two cases:
        \begin{itemize}[listparindent=1.5em]
            \item $j < k - 1$.

            Then $t^{d_kh_k + nk + j}\paren{0^{\l_{k + 1}}} = t^j\paren{\rho_n}\tau_n$ falls into case \ref{II_not_top} above, and we have
            \[t^{d_kh_k + nk + j + 1}\paren{0^{\l_{k - 1}}} = t\paren{t^j\paren{\rho_n}\tau_n} = t^{j + 1}\paren{\rho_n}\tau_n.\]
            Thus, the claim holds for $n$ and $j + 1$.

            \item $n < \abs{F_{k + 1}} - 1$ and $j = k - 1$.

            Then $t^{d_kh_k + nk + j}\paren{0^{\l_{k + 1}}} = t^{k - 1}\paren{\rho_n}\tau_n$ falls into case \ref{II_top_stop} above, and we have
            \[t^{d_kh_k + (n + 1)k}\paren{0^{\l_{k + 1}}} = t\paren{t^{k - 1}\paren{\rho_n}\tau_n} = \rho_{n + 1}\tau_{n + 1}.\]
            Thus, the claim holds for $n + 1$ and $0$. \qedhere
        \end{itemize}
    \end{proof}

    The last fact demonstrates how the portion of the old Group II stacks with bases coming from $F'_{k + 1}$ are now at the top of the new Group I stack.

    \begin{lem}\label{I_to_II}
        $t^{d_kh_k + nk + j}\paren{0^{\l_{k + 1}}} = t^j\paren{\rho_n}\tau_n$ whenever $n < \abs{F_{k + 1}}$ and $j < k$.
    \end{lem}
    \begin{proof}
        Let $n < \abs{F_{k + 1}}$ and $j < k$. (Thus, $k \geq 1$.) We first claim that $t^{d_kh_k}\paren{0^{\l_{k + 1}}} = \rho_0\tau_0$, for which we must consider two cases:
        \begin{itemize}
            \item $k = 1$. Then $\cyl{F'_2} = \cyl{F_1} \setminus \cyl{F_2} = \cyl{0} \setminus \cyl{01} = \cyl{00}$, so the least string lexicographically in $F'_2$ is $\rho_0\tau_0 = 0^{\l_2}$. As $h_1 = 0$, this gives the claim.

            \item $k \geq 2$. Then $h_k > 0$, so we know by \cref{I_refined} that
            \[t^{d_kh_k - 1}\paren{0^{\l_{k + 1}}} = t^{(d_k - 1)h_k + (h_k - 1)}\paren{0^{\l_{k + 1}}} = t^{h_k - 1}\paren{0^{\l_k}}\theta_{d_k - 1}.\]
            This falls into case \ref{I_top} above, which implies
            \[t^{d_kh_k}\paren{0^{\l_{k + 1}}} = t\paren{t^{h_k - 1}\paren{0^{\l_k}}\theta_{d_k - 1}} = \rho_0\tau_0.\]
        \end{itemize}
        Then since $t^{nk + j}\paren{\rho_0\tau_0} = t^j\paren{\rho_n}\tau_n$ by \cref{II_rejected}, the proof is complete.
    \end{proof}

    \proofstep{Step 8: Completing the induction}

    Now we verify each of the conditions above for $k + 1$:
    \begin{enumerate}[listparindent=1.5em, label=(\Alph*)]
        \item Let $|\sigma| = \l_{k + 1}$. Then for all $j < k$, we have
        \[t\paren{\sigma \prefix \l_j} \subseteq t\paren{\sigma \prefix \l_k} \subseteq t(\sigma),\]
        where the former $\subseteq$ holds by the inductive hypothesis, and the latter $\subseteq$ was shown by cases in the definition of $t$ for strings of length $\l_{k + 1}$.

        \item Let $\sigma$ be a string with length $\l_{k + 1}$. By the inductive hypothesis, $\sigma \prefix \l_k$ has one of the following forms:
        \begin{enumerate}[listparindent=1.5em, label=\Roman*.]
            \item $\sigma \prefix \l_k = t^j\paren{0^{\l_k}}$ for some $j < h_k$.

            Take $n$ such that $\sigma = t^j\paren{0^{\l_k}}\theta_n$. Then \cref{I_refined} tells us that $\sigma = t^{nh_k + j}\paren{0^{\l_{k + 1}}}$, which has the form \ref{I} for length $\l_{k + 1}$.
            
            Also, since $n < d_k$ and $j < h_k$, we have $nh_k + j < d_kh_k \leq h_{k + 1} - 1$. (Thus, we need not satisfy $1^{k - 1}0 \subseteq \sigma$.)

            \item $\sigma \prefix \l_k = t^j(\rho)$ for some $\rho \in F_k$ and $j < k$.

            Take $\tau$ with length $\l_{k + 1} - \l_k$ such that $\sigma = t^j(\rho)\tau$. We have two subcases:
            \begin{itemize}[listparindent=1.5em]
                \item $\rho\tau \in F_{k + 1}$.

                Then \cref{II_refined} tells us that $\sigma = t^j(\rho\tau)$, which has the form \ref{II} for length $\l_{k + 1}$.

                \item $\rho\tau \notin F_{k + 1}$.

                Then taking $n$ such that $\rho\tau = \rho_n\tau_n$, \cref{I_to_II} tells us that $\sigma = t^{d_kh_k + nk + j}\paren{0^{\l_{k + 1}}}$, which has the form \ref{I} for length $\l_{k + 1}$.

                Now suppose we have $d_kh_k + nk + j = h_{k + 1} - 1$. By (\ref{hk_recurrence}), we know that $h_{k + 1} = d_kh_k + k\abs{F_{k + 1}}$, which implies
                \[nk + j = k\abs{F_{k + 1}} - 1 = \paren{\abs{F_{k + 1}} - 1}k + (k - 1)\]
                and hence
                $n = \abs{F_{k + 1}} - 1$ and $j = k - 1$. (These are the largest that $n$ and $j$ can be.) Then the inductive hypothesis ensures that
                \[1^{k - 1}0 \subseteq t^j(\rho) \subseteq \sigma.\]
            \end{itemize}

            \item $\sigma \prefix \l_k = 1^k\eta$ for some $\eta$ with length $\l_k - k$.

            We have two subcases:
            \begin{itemize}[listparindent=1.5em]
                \item $1^k0 \subseteq \sigma$.

                Take $n$ such that $\sigma$ is the $n\th$ string lexicographically of length $\l_{k + 1}$ extending $1^k0$, and let $\rho\tau$ to be the $n\th$ string lexicographically in $F_{k + 1}$. Since $t^{k - 1}(\rho)\tau$ falls into case \ref{II_top_continue} above, it follows by \cref{II_refined} that
                \[\sigma = t\paren{t^{k - 1}(\rho)\tau} = t\paren{t^{k - 1}(\rho\tau)} = t^k(\rho\tau)\]
                has the form \ref{II} for length $\l_{k + 1}$. Moreover, we have $1^k0 \subseteq \sigma$, as required.

                \item $1^{k + 1} \subseteq \sigma$.

                Then $\sigma$ has the form \ref{III} for length $\l_{k + 1}$.
            \end{itemize}
            
        \end{enumerate}
        
        \item We consider strings $\sigma$ of the following forms.
        \begin{enumerate}[listparindent=1.5em, label=\roman*.]
             \item $\sigma = t^{h_{k + 1} - 1}\paren{0^{\l_{k + 1}}}$, where $h_{k + 1} > 0$.

             By (\ref{hk_recurrence}),
             \[h_{k + 1} - 1 = d_kh_k + k\abs{F_{k + 1}} - 1 = d_kh_k + \paren{\abs{F_{k + 1}} - 1}k + (k - 1),\]
             so letting $n = \abs{F_{k + 1}} - 1$, \cref{I_to_II} tells us that
             \[\sigma = t^{d_kh_k + nk + (k - 1)}\paren{0^{\l_{k + 1}}} = t^{k - 1}\paren{\rho_n}\tau_n.\]
             This falls into case \ref{II_top_stop} above, so $t(\sigma) = \zero$.

             \item $\sigma = t^k(\rho\tau)$ for some $\rho\tau \in F_{k + 1}$, where $\rho \in F_k$.

            If $k > 0$, then by \cref{II_refined}, $t^{k - 1}(\rho\tau) = t^{k - 1}(\rho)\tau$ falls into case \ref{II_top_continue} above, so
            \[\sigma = t^k(\rho\tau) = t\paren{t^{k - 1}(\rho)\tau} \supseteq 1^k0 \supseteq 1^k.\]
            If $k = 0$, then trivially $\sigma \supseteq 1^k$ as well. Thus, $\sigma$ falls into case \ref{III_top} above, which means $t(\sigma) = \zero$.

             \item $\sigma = 1^{k + 1}\eta$ for some $\eta$ with length $\l_{k + 1} - (k + 1)$.

             Since $\sigma \supseteq 1^{k + 1}$ falls into case \ref{III_top} above, we have $t(\sigma) = \zero$.
        \end{enumerate}
    \end{enumerate}

    This at last concludes the recursive definition of $t: \bigcup_k 2^{\l_k} \to 2^{<\N}$, which is then computable since $\paren{F_k}_k$ is.
    
    \proofstep{Step 9: Checking properties of the limit}

    For $k \geq 2$, let $S_k = \set{\sigma \in 2^{\l_k}: |t(\sigma)| = \l_k}$. Then by conditions (\ref{partition}) and (\ref{top_unmapped}), we have
    \[S_k = \underbrace{\set{t^j(0^{\l_k}): j < h_k - 1}}_{\text{Group I (without the top)}} \cup \underbrace{\set{t^j(\rho): \rho \in F_k, j < k - 1}}_{\text{Group II (without the top)}}\]
    and
    \[2^{\l_k} \setminus S_k = \underbrace{\set{t^{h_k - 1}(0^{\l_k})}}_{\text{Top of Group I}} \cup \underbrace{\set{t^{k - 1}(\rho): \rho \in F_k}}_{\text{Top of Group II}} \cup \underbrace{\set{1^k\eta: \eta \in 2^{\l_k - k}}}_{\text{Group III}}.\]
    We now verify each of the requirements of \cref{limit_is_measure_preserving}:
    \begin{itemize}
        \item $t: \bigcup_k 2^{\l_k} \to 2^{<\N}$ is monotone by condition (\ref{monotone}).
        
        \item $\paren{\cyl{S_k}}_k$ is increasing. Consider any $x \in \cyl{S_k}$ and let $\sigma = x \prefix \l_{k + 1}$. We show that $\sigma \in S_{k + 1}$ and hence $x \in \cyl{S_{k + 1}}$ by cases:
        \begin{enumerate}[I.]
            \item $\sigma \prefix \l_k = t^j\paren{0^{\l_k}}$ for some $j < h_k - 1$.
            
            Then taking $n$ such that $\sigma = t^j\paren{0^{\l_k}}\theta_n$, \cref{I_refined} implies
            \[\sigma = t^j\paren{0^{\l_k}}\theta_n = t^{nh_k + j}\paren{0^{\l_{k + 1}}} \in S_{k + 1}.\]
            \item $\sigma \prefix \l_k = t^j(\rho)$ for some $\rho \in F_k$ and $j < k - 1$.
            
            Take $\tau$ with length $\l_{k + 1} - \l_k$ such that $\sigma = t^j(\rho)\tau$. We have two subcases:
            \begin{itemize}
                \item $\rho\tau \in F_{k + 1}$.
                
                Then \cref{II_refined} implies
                \[\sigma = t^j(\rho)\tau = t^j(\rho\tau) \in S_{k + 1}.\]
                \item $\rho\tau \in F_{k + 1}'$.
                
                Then taking $n$ such that $\rho\tau = \rho_n\tau_n$, \cref{I_to_II} implies
                \[\sigma = t^j(\rho_n)\tau_n = t^{d_kh_k + nk + j}\paren{0^{\l_{k + 1}}} \in S_{k + 1}.\]
            \end{itemize}
        \end{enumerate}

        \item $t$ is injective on $S_k$ because the strings listed in condition (\ref{partition}) are all distinct.
        \item $t$ maps every string in $2^{\l_k} \setminus S_k$ to $\zero$ by condition (\ref{top_unmapped}).
        \item $\lambda\cyl{S_k} \to 1$ as $k \to \infty$ because
        \[1 - \lambda\cyl{S_k} = \lambda\cyl{2^{\l_k} \setminus S_k} = 2^{-\l_k}\paren{1 + 2 \cdot 2^{\l_k - k}} = 2^{-\l_k} + 2 \cdot 2^{-k} \to 0 \quad\text{as}\quad k \to \infty.\]
    \end{itemize}
    Thus, \cref{limit_is_measure_preserving} tells us that the limit $T: \ \subseteq \! 2^\N \to 2^\N$ is injective and measure-preserving with $\Sigma^0_1$ domain $\bigcup_k \cyl{S_k} = \neg\set{1^\infty}$, where equality holds as follows:
    \begin{itemize}
        \item $\subseteq$. Since $1^{\l_k} \notin S_k$ for all $k$ by condition (\ref{partition}), we have $1^\infty \notin \bigcup_k \cyl{S_k}$.
        \item $\supseteq$. If $x \notin \bigcup_k \cyl{S_k}$, then we must have $1^{k - 2} \subseteq x \prefix \l_k$ for all $k$ by condition (\ref{partition}).
    \end{itemize}

    $T$ also satisfies the following properties:
    \begin{itemize}
        \item For all $y \in \dom T$, we have $T(y) \in \dom T$.
        
        This is because $1^k \not\subseteq t(\sigma)$ for all $\sigma \in S_k$.
        
        \item $T$ is computable.
        
        This holds by \cref{limit_of_computable_monotone_map} since $t$ is computable.
        
        \item $T^k(x) \notin \cyl 0$ for all $k > 0$.
        
        For any $k > 0$, by substituting $k + 1$ for $k$ in condition (\ref{partition}) with $\rho = x \prefix \l_{k + 1}$ and $j = k$, we have
        \[1 \subset 1^k0 \subseteq t^k\paren{x \prefix \l_{k + 1}} \subset T^k(x).\]
    \end{itemize}

    \proofstep{Step 10: Proving ergodicity of the limit}

    Before showing that $T$ is ergodic, we prove an elementary lemma.

    \begin{lem}\label{injective_invariance}
        Let $(X, \B, \mu)$ be a probability space, $T: \ \subseteq \! X \to X$ be an injective measure-preserving transformation, and $E \in \B$ satisfy $T^{-1}(E) = E$. Let $D = \dom T$. Then for all $j \in \N$, we have $T^j(E \cap D) \subseteq E \cap D$ and $\mu\paren{E \setminus T^j(E \cap D)} = 0$. Thus, $E$ differs from $T^j(E \cap D)$ by a measure-zero set.
    \end{lem}
    \begin{proof}
        Assume the claim holds for some $j \in \N$, which it trivially does for $j = 0$. Then using the inductive hypothesis and $E$'s invariance, we have
        \[T^{j + 1}(E \cap D) \subseteq T(E \cap D) = T\paren{T^{-1}(E) \cap D} \subseteq E \cap D.\]
        Then using $T$'s preservation of measure, $E$'s invariance and $T$'s injectivity, and the inductive hypothesis, we have
        \[\mu\paren{E \setminus T^{j + 1}(E \cap D)} = \mu\paren{T^{-1}(E) \setminus T^{-1}\paren{T^{j + 1}(E \cap D)}} = \mu\paren{E \setminus T^j(E \cap D)} = 0.\]
        Thus, the claim holds for all $j \in \N$.
    \end{proof}
    
    We now give an argument closely resembling that of Theorem 6.2 in \cite{Friedman}. Let $E$ be a $T$-invariant set with positive measure. That is, $T^{-1}(E) = E$ and $\lambda(E) > 0$. We will show that $\lambda(E) = 1$.

    Let $I_k = \bigcup_{j < h_k} \cyl{t^j(0^{\l_k})}$ be the Group I stack at stage $k$. Observe that $\paren{I_k}_k$ is increasing in $k$ by \cref{I_refined} and that
    \[\lambda(I_k) = 2^{-\l_k}h_k = 2^{-\l_k}\paren{2^{\l_k} - (k + 1)2^{\l_k - k}} = 1 - (k + 1)2^{-k} \to 1\quad\text{as}\quad k \to \infty.\]
    By the Lebesgue Density Theorem, almost every point of $E$ has Lebesgue density 1, so there exists such a point $y \in E \cap \bigcup_k I_k$. This means that
    \[\lim_{\l \to \infty} \frac{\lambda\paren{E \cap \cyl{y \prefix \l}}}{\lambda\cyl{y \prefix \l}} = 1.\]
    Fix $\eps > 0$ and take $L$ such that for all $\l \geq L$,
    \[\lambda\paren{E \cap \cyl{y \prefix \l}} \geq (1 - \eps)\lambda\cyl{y \prefix \l}.\]
    Using $E$'s invariance and \cref{injective_invariance}, this implies
    \[\lambda\paren{E \cap T^j\cyl{y \prefix \l}} \geq (1 - \eps)\lambda\cyl{y \prefix \l}\]
    for all integers $j$. Consider any $k$ such that $\l_k \geq L$ and $y \in I_k$, which holds for all sufficiently large $k$. Take $i < h_k$ such that $t^i\paren{0^{\l_k}} \subset y$. Then since $T$ is injective, for all $j < h_k$, we have
    \[T^{j - i}\cyl{y \prefix \l_k} = T^{j - i}\cyl{t^i\paren{0^{\l_k}}} = \cyl{t^j\paren{0^{\l_k}}},\]
    which are all disjoint by condition (\ref{partition}). Therefore,
    \begin{align*}
        \lambda(E) &\geq \lambda\paren{\bigcup_{j < h_k} \paren{E \cap \cyl{t^j\paren{0^{\l_k}}}}} = \sum_{j < h_k} \lambda\paren{E \cap \cyl{t^j\paren{0^{\l_k}}}} \\
        &= \sum_{j < h_k} \lambda\paren{E \cap T^{j - i}\cyl{y \prefix \l_k}} \geq (1 - \eps)\sum_{j < h_k} \lambda\cyl{y \prefix \l_k} = (1 - \eps)\lambda(I_k).
    \end{align*}
    Sending $k \to \infty$ gives $\lambda(E) \geq 1 - \eps$. As $\eps > 0$ was arbitrary, we conclude that $\lambda(E) = 1$ and consequently that $T$ is ergodic.
\end{proof}

\subsection{Discussion}

With the proof of \cref{Kurtz_Poincare_characterization} at last complete, a few questions arise.

First, note that the transformation $T$ we constructed was not total but instead had a $\Sigma^0_1$ domain. This limitation seems inherent when working with an arbitrary point which is not Kurtz random, but a total transformation can clearly suffice some of the time. For example, if $x = 01^\infty$, we could take $A = \cyl 0$ and $T$ to be the shift map. Thus, we have the following question.

\begin{qn}
    For what non-Kurtz-random points $x \in 2^\N$ do there exist a clopen set $A \ni x$ with $\lambda(A) > 0$ and a \emph{total} computable measure-preserving transformation $T: 2^\N \to 2^\N$ such that $T^k(x) \notin A$ for all $k > 0$? Does this change if we require $T$ to be ergodic or injective?
\end{qn}

Next, observe that while the transformation $T$ we constructed was ergodic, we did not show that it possesses any stronger mixing properties like the shift map has. Indeed, the simple cutting-and-stacking procedure we implemented likely will not guarantee anything stronger than ergodicity. A more complex procedure resembling the staircase construction of  \cite{Adams}, on the other hand, could very well achieve a mixing transformation.

Suitably implementing this would probably be rather messy given that the Group II stacks which contain the first several iterates of the clopen set $C_k$ must be kept intact when we eventually move them over to Group I. Nonetheless, the following question likely has an affirmative answer with the right strategy.

\begin{qn}
    Given a non-Kurtz-random point $x \in 2^\N$, do there exist a clopen set $A \ni x$ with $\lambda(A) > 0$ and a computable \emph{weakly/lightly/strongly mixing} measure-preserving transformation $T: \ \subseteq \! 2^\N \to 2^\N$ such that $T^k(x) \notin A$ for all $k > 0$?
\end{qn}

Finally, our construction was very much specific to Cantor space with the Lebesgue measure. One could consider other computable probability measures on Cantor space or even computable probability spaces other than Cantor space. Note that one cannot simply rely on the isomorphism between any computable probability space and Cantor space given by Theorem 5.1.1 of \cite{Hoyrup2009} because this is only up to a null $\Sigma^0_2$ set, which could contain the very non-Kurtz-random points we care about!

\begin{qn}
    For what computable probability spaces other than $\paren{2^\N, \lambda}$ is it true that given a non-Kurtz-random point $x \in X$, there exist a $\Sigma^0_1$ set $A \ni x$ with $\lambda(A) > 0$ and a computable (ergodic) measure-preserving transformation $T: \ \subseteq \! X \to X$ such that $T^k(x) \notin A$ for all $k > 0$?
\end{qn}

\section[Effective recurrence in Π0n sets with Ø(n-1) computable measure]{Effective recurrence in $\Pi^0_n$ sets with $\zero^{(n - 1)}$-computable measure}\label{section_Schnorr}

We now shift our focus from $\Sigma^0_n$ sets to $\Pi^0_n$ sets with $\zero^{(n - 1)}$-computable measure.

\subsection{Recurrence for Schnorr $n$-random points}

The following result for $\paren{2^\N, \lambda}$ and $n = 1$ was mentioned by \cite{Downey2016}, courtesy of Jason Rute. Below we extend the argument to any computable probability space.

\begin{thm}\label{CPS_Schnorr_n_Poincare}
    Let $(X, \mu)$ be a computable probability space and $n \geq 1$. If $x \in X$ is Schnorr $n$-random, then for any $\Pi^0_n$ set $A \ni x$ with $\zero^{(n - 1)}$-computable $\mu(A) > 0$ and any computable measure-preserving transformation $T: \ \subseteq \! X \to X$ with $x \in \dom T$, we have $T^k(x) \in A$ for some $k > 0$.
\end{thm}
\begin{proof}
    Suppose there are a $\Pi^0_n$ set $A \ni x$ with $\zero^{(n - 1)}$-computable $\mu(A) > 0$ and a computable measure-preserving transformation $T: \ \subseteq\! X \to X$ with $x \in \dom T =: D$ such that $T^k(x) \notin A$ for all $k > 0$.

    Define $\paren{A_k}_{k > 0}$ according to the following two cases:
    \begin{itemize}
        \item $n = 1$. Then there exists by \cref{approximate_sigma01_or_pi01} a nested sequence $\paren{A_k}_{k > 0}$ of uniformly $\Sigma^0_1$ sets with uniformly computable measures such that $A = \bigcap_{k > 0} A_k$.
        \item $n \geq 2$. Then there exists a nested sequence $\paren{A_k}_{k > 0}$ of uniformly $\Sigma^0_{n - 1}$ sets such that $A = \bigcap_{k > 0} A_k$. By \cref{measure_of_sigma0n}, $\mu(A_k)$ is uniformly $\zero^{(n - 2)}$-lower-semicomputable in $k$ and hence uniformly $\zero^{(n - 1)}$-computable in $k$.
    \end{itemize}
    
    \cref{iterated_preimage_of_sigma0n} implies that there are uniformly $\Sigma^0_n$ sets $U_k$ such that $T^{-k}(\neg A) = D \cap U_k$. The hypotheses then imply
    \[x \in A \cap \bigcap_{k > 0} T^{-k}(\neg A) = D \cap \bigcap_{k > 0} \paren{A_k \cap U_k} \subseteq \bigcap_{k > 0} \paren{A_k \cap U_k}.\]
    Since $D$ has full measure, the inclusion above does not change the measure, which is null by the Poincar\'e Recurrence Theorem. The sets $A_k$ have uniformly $\zero^{(n - 1)}$-computable measures by construction, as do the sets $U_k$ since $\mu\paren{U_k} = \mu(\neg A)$ for all $k$ due to the fact that $D$ has full measure and $T$ is measure-preserving.
    
    Thus, \cref{intersection_of_sigma0n_with_computable_measure} implies that their intersections $A_k \cap U_k$ likewise have uniformly $\zero^{(n - 1)}$-computable measures. It follows by \cref{loose_Schnorr_test} that $x$ is not Schnorr $n$-random.
\end{proof}

By not intersecting with $A$ or involving the sets $A_k$, a similar proof to the one given above would show the following.

\begin{thm}\label{CPS_Schnorr_n_Poincare_ergodic}
    Let $(X, \mu)$ be a computable probability space and $n \geq 1$. If $x \in X$ is Schnorr $n$-random, then for any $\Pi^0_n$ set $A$ with $\zero^{(n - 1)}$-computable $\mu(A) > 0$ and any computable ergodic measure-preserving transformation $T: \ \subseteq \! X \to X$ with $x \in \dom T$, we have $T^k(x) \in A$ for some $k > 0$.
\end{thm}

By \cref{preservation_of_randomness}, for both of these theorems, we will have $T^k(x) \in A$ for infinitely many $k$ whenever $x$ is Schnorr $n$-random.

The converse to \cref{CPS_Schnorr_n_Poincare_ergodic} does hold in $\paren{2^\N, \lambda}$, which characterizes Schnorr $n$-randomness in terms of recurrence for ergodic transformations.

\begin{thm}\label{Schnorr_n_Poincare_ergodic_characterization}
    Consider Cantor space with the Lebesgue measure $\lambda$, and let $n \geq 1$. Then $x \in 2^\N$ is Schnorr $n$-random if and only if for any $\Pi^0_n$ set $A$ with $\zero^{(n - 1)}$-computable $\mu(A) > 0$ and any computable ergodic measure-preserving transformation $T: \ \subseteq \! 2^\N \to 2^\N$ with $x \in \dom T$, we have $T^k(x) \in A$ for some $k > 0$.
\end{thm}
\begin{proof}
    The $\RA$ direction is simply \cref{CPS_Schnorr_n_Poincare_ergodic}. For the $\LA$ direction, assume that $x \in 2^\N$ is not Schnorr $n$-random. By \cref{relativizing_randomness}, $x$ is not Schnorr random relative to $\zero^{(n - 1)}$, so there is a sequence of uniformly $\Sigma^0_1\paren{\zero^{(n - 1)}}$ sets $U_k$ such that $\lambda(U_k) \leq 2^{-k}$ for all $k$ and $\lambda(U_k)$ is uniformly $\zero^{(n - 1)}$-computable in $k$. Replacing $U_k$ with $U_{2k + 2}$, we may assume that $\lambda(U_k) \leq 2^{-2k - 2}$ for all $k$.

    Let $T: 2^\N \to 2^\N$ be the shift map. By \cref{shift_of_sigma01}, the sets $T^k\paren{U_k}$ are uniformly $\Sigma^0_1\paren{\zero^{(n - 1)}}$ with uniformly $\zero^{(n - 1)}$-computable measures. Define $V = \bigcup_k T^k\paren{U_k}$, which is then also $\Sigma^0_1\paren{\zero^{(n - 1)}}$. We will take $A = \neg V$, which is $\Pi^0_1\paren{\zero^{(n - 1)}}$ and hence $\Pi^0_n$ by \cref{sigma0n(A(m))_is_sigma0(m+n)(A)}. By definition of $V$, we have $T^k(x) \notin A$ for all $k$. It remains to be shown that $A$ has $\zero^{(n - 1)}$-computable measure.
    
    Since $T$ at most doubles the measure of each cylinder and hence of every open set, we have
    \[\lambda(V) \leq \sum_k \lambda\paren{T^k\paren{U_k}} \leq \sum_k 2^k\lambda\paren{U_k} \leq \sum_k 2^k \cdot 2^{-2k - 2} = \sum_k 2^{-k - 2} = \frac{1}{2} < 1.\]
    Thus, $\lambda(A) > 0$. Now let $V_k = \bigcup_{j \leq k} T^j\paren{U_j}$, which is uniformly $\Sigma^0_1\paren{\zero^{(n - 1)}}$ in $k$ and has uniformly $\zero^{(n - 1)}$-computable measure in $k$ by \cref{intersection_of_sigma0n_with_computable_measure}. Then as above, we have
    \[\lambda\paren{V \setminus V_k} \leq \lambda\paren{\bigcup_{j > k} T^j\paren{U_j}} \leq \sum_{j > k} 2^{-j - 2} = 2^{-k - 2} < 2^{-k}.\]
    That is, $\lambda(V_k)$ approximates $\lambda(V)$ within $2^{-k}$ and has uniformly $\zero^{(n - 1)}$-computable measure. This implies that $\lambda(V)$ is $\zero^{(n - 1)}$-computable, along with $\lambda(A) = 1 - \lambda(V)$.
\end{proof}

\subsection[Recurrence for Π0n-generic points]{Recurrence for $\Pi^0_n$-generic points}

Observe that the set $A$ we constructed in the proof of \cref{Schnorr_n_Poincare_ergodic_characterization} did \emph{not} contain the point $x$. One could try to fix this by defining $V = \bigcup_{k > 0} T^k\paren{U_k}$, but as we will soon see, there are some points for which such an attempt would necessary fail. Thus, a richer structure of recurrent points arises when we consider only sets $A \ni x$.

\subsubsection{Defining genericity in computable probability spaces}

Recall what it means for a point in Cantor space to be (weakly) $n$-generic.

\begin{defn}
    A point $x \in 2^\N$ is \textbf{$n$-generic} if $x$ does not lie on the boundary of any $\Sigma^0_1\paren{\zero^{(n - 1)}}$ set. Equivalently, for any $\Pi^0_1\paren{\zero^{(n - 1)}}$ set $P \ni x$, there is a clopen set $Q \subseteq P$ with $x \in Q$.
\end{defn}

\begin{defn}
    A point $x \in 2^\N$ is \textbf{weakly $n$-generic} if $x$ lies in every dense $\Sigma^0_1\paren{\zero^{(n - 1)}}$ set. Equivalently, for any $\Pi^0_1\paren{\zero^{(n - 1)}}$ set $P \ni x$, there is a nonempty clopen set $Q \subseteq P$.
\end{defn}

The following result is classical.

\begin{thm}\label{genericity_relationships}
    For all $n \geq 1$,
    \begin{center}
        weakly $(n + 1)$-generic $\myimplies{(a)}$ $n$-generic $\myimplies{(b)}$ weakly $n$-generic \\
        $\myimplies{(c)}$ (Kurtz random relative to $\zero^{(n - 1)}$ and not Schnorr random).
    \end{center}
\end{thm}
\begin{proof}
    Implication (a) is trivial, implication (b) holds by Theorem 2.24.11 in \cite{DowneyHirschfeldt}, and implication (c) holds by Proposition 8.11.9 in \cite{DowneyHirschfeldt} and the fact that no weakly $n$-generic point can lie in a null $\Pi^0_1\paren{\zero^{(n - 1)}}$ set.
\end{proof}

Using the Baire Category Theorem with the countable collection of dense $\Sigma^0_1\paren{\zero^{(n - 1)}}$ sets, we obtain the following.

\begin{cor}
    For all $n \geq 1$, the set of (weakly) $n$-generic points is comeager in $2^\N$.
\end{cor}

The reason we bring up generic points now is because they constitute a class of points separate from the Schnorr random points which recur in any $\Pi^0_1$ set containing them. Indeed, consider any computable measure-preserving transformation $T: \ \subseteq \! 2^\N \to 2^\N$ on Cantor space with the Lebesgue measure $\lambda$. Let $x \in 2^\N$ be 1-generic and $P \ni x$ be $\Pi^0_1$ with $\lambda(P) > 0$. Then by the definition of 1-genericity, there is a clopen set $Q \subseteq P$ with $x \in Q$. Since $x$ is a Kurtz random point inside the clopen set $Q$ with positive measure, \cref{CPS_weak_n_Poincare} guarantees the existence of some $k > 0$ such that $T^k(x) \in Q \subseteq P$, meaning that $x$ recurs in $P$.

While genericity has occasionally been studied in computable topological spaces other than Cantor space---see \cite{Hoyrup2017}, for example---we will need to develop some new notions of genericity in order to obtain similar recurrence results for any computable probability space and for sets more complex than $\Pi^0_1$.

\begin{defn}\label{pi0n_generic_definition}
    Let $(X, \mu)$ be a computable probability space and $A \subseteq \N$. (In each of the following definitions, we drop ``relative to $A$'' if $A = \zero$.)
    \begin{itemize}
        \item A point $x \in X$ is \textbf{$\Pi^0_1$-generic relative to $A$} if for any $\Pi^0_1(A)$ set $P \ni x$, there is a $\Sigma^0_1$ set $U \subseteq P$ with $\mu(U) > 0$ and $x \in U$.

        \item A point $x \in X$ is \textbf{weakly $\Pi^0_1$-generic relative to $A$} if for any $\Pi^0_1(A)$ set $P \ni x$, there is a $\Sigma^0_1$ set $U \subseteq P$ with $\mu(U) > 0$.

        \item For $n \geq 2$, a point $x \in X$ is \textbf{$\Pi^0_n$-generic relative to $A$} if for any $\Pi^0_n(A)$ set $P \ni x$, there is a $\Pi^0_1\paren{A^{(n - 2)}}$ set $Q \subseteq P$ with $\mu(Q) > 0$ and $x \in Q$.

        \item For $n \geq 2$, a point $x \in X$ is \textbf{weakly $\Pi^0_n$-generic relative to $A$} if for any $\Pi^0_n(A)$ set $P \ni x$, there is a $\Pi^0_1\paren{A^{(n - 2)}}$ set $Q \subseteq P$ with $\mu(Q) > 0$.
    \end{itemize}
\end{defn}

Some of the familiar relationships between the standard notions of genericity are likewise true for $\Pi^0_n$ genericity no matter what $n$ is.

\begin{lem}
    Let $(X, \mu)$ be a computable probability space. Then for all $n \geq 1$ and $A \subseteq \N$,
    \begin{center}
        $\Pi^0_n$-generic relative to $A$ $\myimplies{(a)}$ weakly $\Pi^0_n$-generic relative to $A$ \\
        $\myimplies{(b)}$ weakly $n$-random relative to $A$.
    \end{center}
\end{lem}
\begin{proof}
    Implication (a) is obvious from the definition since we merely remove the condition ``$x \in U$'' or ``$x \in Q$'' when moving from $\Pi^0_n$-genericity to weak $\Pi^0_n$-genericity.

    Implication (b) holds because no weakly $\Pi^0_n$-generic point relative to $A$ can lie in a null $\Pi^0_n(A)$ set.
\end{proof}

The following result shows that in the definition of (weak) $\Pi^0_1$-genericity, we could equivalently replace ``$\Sigma^0_1$'' with ``open'' or ``almost decidable.''

\begin{lem}\label{pi01_genericity_with_open_sets}
    Let $(X, \mu)$ be a computable probability space and $A \subseteq \N$. Then the following are equivalent:
    \begin{enumerate}[(a)]
        \item For any $\Pi^0_1(A)$ set $P \ni x$, there is a $\Sigma^0_1$ set $U \subseteq P$ with $\mu(U) > 0$ (and $x \in U$).
        \item For any $\Pi^0_1(A)$ set $P \ni x$, there is an open set $U \subseteq P$ with $\mu(U) > 0$ (and $x \in U$).
        \item For any $\Pi^0_1(A)$ set $P \ni x$, there is an almost decidable set $B \subseteq P$ with $\mu(B) > 0$ (and $x \in B$).
    \end{enumerate}
\end{lem}
\begin{proof}
    Below we prove each implication including ``$x \in U$'' or ``$x \in B$.'' The proof excluding this condition is nearly identical.

    (a) $\RA$ (b). This is obvious since any $\Sigma^0_1$ set is open.

    (b) $\RA$ (c). Assume $P \ni x$ is $\Pi^0_1(A)$ and let $U \subseteq P$ be open with $\mu(U) > 0$ and $x \in U$. Since the almost decidable sets form a basis for the topology on $X$ by \cref{almost_decidable_basis}, there is an almost decidable set $B \subseteq U$ with $\mu(B) > 0$ and $x \in B$.
    
    (c) $\RA$ (a). Assume $P \ni x$ is $\Pi^0_1(A)$ and let $B \subseteq P$ be almost decidable with $\mu(B) > 0$ and $x \in B$. Take $\Sigma^0_1$ sets $U \subseteq B$ and $V \subseteq \neg B$ such that $\mu(U \cup V) = 1$. Since $x$ is Kurtz random, we must have $x \in U \cup V$ and hence $x \in U$ given that $x \in B \subseteq \neg V$. Then $U \subseteq P$ is $\Sigma^0_1$ with $\mu(U) = \mu(B) > 0$ and $x \in U$.
\end{proof}

\subsubsection[Π01-genericity]{$\Pi^0_1$-genericity}

First we note the elementary fact that $\Pi^0_1$-genericity becomes a stronger condition when relativized to a stronger oracle. (The same is not obviously true of $\Pi^0_n$-genericity for $n \geq 2$ since the required complexity of the set $Q$ depends on the oracle. Indeed, we will later prove this false in \cref{pi0n_genericity_does_not_get_stronger_with_relativization}.)

\begin{lem}\label{pi01_genericity_gets_stronger_with_relativization}
    Let $(X, \mu)$ be a computable probability space and $A, B \subseteq \N$. Then whenever $A \geq_T B$,
    \begin{center}
        (weakly) $\Pi^0_1$-generic relative to $A$ $\implies$ (weakly) $\Pi^0_1$-generic relative to $B$.
    \end{center}
\end{lem}
\begin{proof}
    This is immediate since any $\Pi^0_1(B)$ set is also $\Pi^0_1(A)$ when $B \leq_T A$.
\end{proof}

Note that in Cantor space, $\Pi^0_1$-genericity relative to $\zero^{(n - 1)}$ is nearly identical to $n$-genericity, except that it refers to $\Sigma^0_1$ sets with positive measure rather than nonempty clopen sets. By the following result, this does not make a difference, as long as the measure $\mu$ is \emph{strictly positive}, i.e., $\mu(U) > 0$ for any nonempty open set $U$.

\begin{lem}\label{equivalence of_genericity_and_pi01_genericity}
    Consider Cantor space with a computable probability measure $\mu$. Then for all $n \geq 1$,
    \begin{enumerate}[(a)]
        \item Any (weakly) $\Pi^0_1$-generic point relative to $\zero^{(n - 1)}$ is also (weakly) $n$-generic.
        \item Assuming $\mu$ is strictly positive, any (weakly) $n$-generic point is also (weakly) $\Pi^0_1$-generic relative to $\zero^{(n - 1)}$.
        \item Assuming $\mu$ is not strictly positive, there is an $n$-generic point which is not weakly $\Pi^0_1$-generic.
    \end{enumerate}
\end{lem}
\begin{proof}
    \begin{enumerate}[(a)]
        \item Assume $x \in 2^\N$ is $\Pi^0_1$-generic relative to $\zero^{(n - 1)}$, and consider any $\Pi^0_1\paren{\zero^{(n - 1)}}$ set $P \ni x$. Then there is a $\Sigma^0_1$ set $U \subseteq P$ with $\mu(U) > 0$ and $x \in U$. This implies that there is a clopen set $Q \subseteq U \subseteq P$ with $x \in Q$. Since $P$ was an arbitrary $\Pi^0_1\paren{\zero^{(n - 1)}}$ set containing $x$, it follows that $x$ is $n$-generic.

        A nearly identical proof shows that any weakly $\Pi^0_1$-generic point relative to $\zero^{(n - 1)}$ is also weakly $n$-generic.

        \item Assume that $\mu$ is strictly positive. Then since any nonempty clopen set is $\Sigma^0_1$ with positive measure, $n$-genericity clearly implies $\Pi^0_1$-genericity relative to $\zero^{(n - 1)}$, and weak $n$-genericity clearly implies weak $\Pi^0_1$-genericity relative to $\zero^{(n - 1)}$.

        \item Assume that $\mu$ is not strictly positive, and let $U$ be a nonempty open set with $\mu(U) = 0$. Take a cylinder $\cyl\sigma \subseteq U$ and an $n$-generic point $x \supset \sigma$. Then $\cyl\sigma \ni x$ is $\Pi^0_1$ yet contains no open set with positive measure, so $x$ cannot be weakly $\Pi^0_1$-generic. \qedhere
    \end{enumerate}
\end{proof}

Now we turn to describing how $\Pi^0_1$-genericity relates to weak $\Pi^0_1$-genericity, for which we require an elementary lemma.

\begin{lem}\label{when_is_ball_in_pi01}
    Let $(X, d, S)$ be a computable metric space, $A \subseteq \N$, $\paren{r_j}_j$ be a dense uniformly computable sequence in $\R_{> 0}$, and $B_{\cs{i, j}} = \{x \in X: d(x, s_i) < r_j\}$. Then the condition $B_{\cs{i, j}} \subseteq P$ is $\Pi^0_1(A)$ in $\cs{i, j}$ and an index for a $\Pi^0_1(A)$ set $P$.
\end{lem}
\begin{proof}
    Because $\paren{r_j}_j$ is a dense uniformly computable sequence in $\R_{> 0}$, we can write the given $\Pi^0_1(A)$ set $P$ as $\bigcap_{\cs{k, \l} \in W} \neg B_{\cs{k, \l}}$ for some $A$-c.e.\ set $W \subseteq \N$. Then $B_{\cs{i, j}} \subseteq P$ if and only if $B_{\cs{i, j}}$ and $B_{\cs{k, \l}}$ are disjoint for all $\cs{k, \l} \in W$.

    Observe that $B_{\cs{i, j}} \cap B_{\cs{k, \l}} \neq \zero$ holds if and only if there exists $m$ such that $d(s_i, s_m) < r_j$ and $d(s_k, s_m) < r_\l$, which is c.e.\ in $i$, $j$, $k$, and $\l$. Therefore,
    \[B_{\cs{i, j}} \subseteq P \iff \forall\cs{k, \l} \ \Big(\cs{k, \l} \notin W \quad\text{or}\quad B_{\cs{i, j}} \cap B_{\cs{k, \l}} = \zero\Big)\]
    is $\Pi^0_1(A)$ in $\cs{i, j}$ and an index for $P$.
\end{proof}

The following result demonstrates that $\Pi^0_1$-genericity has precisely the relationship one would expect with weak $\Pi^0_1$-genericity.

\begin{thm}\label{weakly_pi01_generic_for_A'_implies_pi01_generic_for_A}
    Let $(X, \mu)$ be a computable probability space and $A \subseteq \N$. Then any weakly $\Pi^0_1$-generic point relative to $A'$ is also $\Pi^0_1$-generic relative to $A$.
\end{thm}
\begin{proof}
    Let $x \in X$ be weakly $\Pi^0_1$-generic relative to $A'$, and consider any $\Pi^0_1(A)$ set $P \ni x$. Let $B_i$ be the $i\th$ ideal ball and define
    \[S = \set{i \in \N: \mu\paren{B_i} > 0 \quad\text{and}\quad B_i \subseteq P},\]
    which is the intersection of a $\Sigma^0_1$ set and a $\Pi^0_1(A)$ set by \cref{when_is_ball_in_pi01}, so it is $A'$-c.e. Let $V = \neg P \cup \bigcup_{i \in S} B_i$, which is then $\Sigma^0_1(A')$. We claim that $V$ intersects every $\Sigma^0_1$ set with positive measure.
    
    Indeed, suppose $U$ is $\Sigma^0_1$ with $\mu(U) > 0$. If $U$ intersects $\neg P$, we are done, so assume $U \subseteq P$. Then there is a ball $B_i \subseteq U$ with positive measure. This means that $i \in S$ and hence $B_i \subseteq V$.

    Now since $x$ is weakly $\Pi^0_1$-generic relative to $A'$ and $V$ is a $\Sigma^0_1\paren{A'}$ set which intersects every $\Sigma^0_1$ set with positive measure, we must have $x \in V$. Given that $x \in P$, it follows that $x \in B_i$ for some $i \in S$. That is, there is a $\Sigma^0_1$ set $U \subseteq P$ with $\mu(U) > 0$ and $x \in U$. As $P$ was an arbitrary $\Pi^0_1(A)$ set containing $x$, we conclude that $x$ is $\Pi^0_1$-generic relative to $A$.
\end{proof}

Now we complete the picture of $\Pi^0_1$-genericity as it relates to randomness, which exactly mirrors \cref{genericity_relationships}.

\begin{thm}\label{pi01_generic_not_Schnorr_random}
    Consider Cantor space with the Lebesgue measure $\lambda$. Then for all $A \subseteq \N$, no weakly $\Pi^0_1$-generic point relative to $A$ is Schnorr random.
\end{thm}
\begin{proof}
    By \cref{pi01_genericity_gets_stronger_with_relativization}, any weakly $\Pi^0_1$-generic point relative to $A$ is weakly $\Pi^0_1$-generic. By \cref{equivalence of_genericity_and_pi01_genericity}, any weakly $\Pi^0_1$-generic point in $\paren{2^\N, \lambda}$ is weakly 1-generic. By \cref{genericity_relationships}, any weakly 1-generic point is not Schnorr random.
\end{proof}

Thus far, we have not made any claims about the existence of $\Pi^0_1$-generic points for an arbitrary computable probability space, and indeed, they do not necessarily exist. To see why, let $C$ be a fat Cantor set in $[0, 1]$ with Lebesgue measure $1/2$, and consider $[0, 1]$ with the measure $\mu$ defined by $\mu(A) = 2\lambda(A \cap C)$. Then for any point $x \in [0, 1]$, we have two cases:
\begin{itemize}
    \item $x \in C$. Then $C \ni x$ is $\Pi^0_1$ yet contains no open set with positive measure since it is nowhere dense (i.e., has empty interior), so $x$ cannot be weakly $\Pi^0_1$-generic.

    \item $x \notin C$. Then there exist rational numbers $q$ and $r$ such that $x \in [q, r] \subseteq \neg C$. Thus, $[q, r] \ni x$ is $\Pi^0_1$ yet contains no open set with positive measure, so $x$ cannot be weakly $\Pi^0_1$-generic.
\end{itemize}

The problem with this example was that the support of $\mu$ was nowhere dense. The following result shows that this is the only obstruction. In it we use the following notation:
\begin{itemize}
    \item $\Int(A)$ denotes the interior of a set $A$.
    \item $\cl A$ denotes the closure of a set $A$.
    \item $\supp \mu$ denotes the \emph{support} of a measure $\mu$, i.e., the set of all points every open neighborhood of which has positive measure. (Recall that the support is always closed, and it also has full measure whenever $X$ is Hausdorff and $\mu$ is a Radon measure, which is the case for any computable probability space.)
\end{itemize}

\begin{thm}\label{pi01_generic_comeager}
    Let $(X, \mu)$ be a computable probability space and $A \subseteq \N$. Then the set of (weakly) $\Pi^0_1$-generic points relative to $A$ is comeager in $\cl{\Int(\supp \mu)}$ and hence nonempty as long as the support of $\mu$ is somewhere dense.
\end{thm}
\begin{proof}
    Let $\U$ be the collection of all $\Sigma^0_1(A)$ sets which intersect every open set with positive measure. Consider any $U \in \U$ and any open set $V$ such that $V \cap \cl{\Int(\supp \mu)} \neq \zero$. Then $V \cap \Int(\supp \mu)$ is a nonempty open set inside the support of $\mu$, so it has positive measure and hence intersects $U$.

    Thus, every set from $\U$ is dense in $\cl{\Int(\supp \mu)}$. This implies that the countable intersection $\bigcap_{U \in \U} U$, which is precisely the set of weakly $\Pi^0_1$-generic points relative to $A$, is comeager in $\cl{\Int(\supp \mu)}$. By the Baire Category Theorem, as long as $\Int(\supp \mu) \neq \zero$, this is a nonempty set.

    Finally, since weakly $\Pi^0_1$-generic points relative to $A'$ are $\Pi^0_1$-generic relative to $A$ by \cref{weakly_pi01_generic_for_A'_implies_pi01_generic_for_A}, we can make the same claim for $\Pi^0_1$-generic points.
\end{proof}

\subsubsection[Π0n-genericity when n ≥ 2]{$\Pi^0_n$-genericity when $n \geq 2$}

$\Pi^0_n$-genericity becomes much more complicated a property when $n \geq 2$. Before proceeding, we recall what it means for a subset of a partial order to be dense.

\begin{defn}
    Let $(P, \leq)$ be a partial order. Then a subset $D \subseteq P$ is \textbf{dense} if for every $p \in P$, there exists $q \in D$ with $q \leq p$.
\end{defn}

The following lemma is instrumental in showing that $\Pi^0_n$-generic points exist. Through the proof, it becomes clear why we must consider $\Pi^0_1\paren{A^{(n - 2)}}$ subsets in the definition of $\Pi^0_n$-genericity relative to $A$.

\begin{lem}\label{pi0n_generic_dense}
    Let $(X, \mu)$ be a computable probability space, $n \geq 2$, and $A \subseteq \N$. Let $(\P, \subseteq)$ be the partial order consisting of $\Pi^0_1(A^{(n - 2)})$ sets with positive measure, ordered by the subset relation. For each $\Pi^0_n(A)$ set $P$, let
    \[D_P = \{Q \in \P: Q \subseteq P \text{ or } P \cap Q = \zero\}.\]
    Then $D_P$ is dense in $\P$ for each $\Pi^0_n(A)$ set $P$.
\end{lem}
\begin{proof}
    Consider any $\Pi^0_n(A)$ set $P$ and $Q \in \P$. We seek to show that there is some $R \in D_P$ such that $R \subseteq Q$.
    
    Write $\neg P = \bigcup_j P_j$, where the sets $P_j$ are uniformly $\Pi^0_{n - 1}(A)$. There are two cases:
    \begin{itemize}
        \item $\mu(Q \setminus P) > 0$. Since
        \[\mu\paren{\bigcup_j \paren{Q \cap P_j}} = \mu(Q \cap \neg P) > 0,\]
        there exists some $j$ such that $\mu\paren{Q \cap P_j} > 0$. Then $Q \cap P_j$ is $\Pi^0_{n - 1}(A)$ with positive measure, so by \cref{approximate_sigma0n_with_sigma01(zero(n-1))} relativized to $A$, there exists a $\Pi^0_1\paren{A^{(n - 2)}}$ set $R \subseteq Q \cap P_j$ with positive measure. Then $R \in \P$ and $R \subseteq P_j \subseteq \neg P$, so $R \in D_P$. Also, $R \subseteq Q$.

        \item $\mu(Q \setminus P) = 0$. Then $\mu\paren{Q \cap P_j} = 0$ for all $j$. Fix $q \in \Q$ with $0 < q < \mu(Q)$. We now define $\Sigma^0_1\paren{A^{(n - 2)}}$ sets $U_k$ and $U_{j, k}$ and an $A^{(n - 2)}$-computable function $s$ such that 
        \[\mu\paren{U_{s(j)} \cap U_{j, s(j)}} \leq 2^{-(j + 1)}q\]
        according to the following subcases:
        \begin{itemize}
            \item $n = 2$. Since $Q$ and the sets $P_j$ are uniformly $\Pi^0_1(A)$, by \cref{approximate_sigma01_or_pi01}, we can write $Q = \bigcap_k U_k$ and $P_j = \bigcap_k U_{j, k}$, where the sets $U_k$ and $U_{j, k}$ are uniformly $\Sigma^0_1(A)$ with uniformly $A$-computable measures. Then for each $j$,
            \[\lim_{k \to \infty} \mu\paren{U_k \cap U_{j, k}} = \mu\paren{P \cap P_j} = 0,\]
            where $\mu\paren{U_k \cap U_{j, k}}$ is uniformly $A$-computable by \cref{intersection_of_sigma0n_with_computable_measure}. Thus, there is an $A$-computable function $s$ such that $\mu\paren{U_{s(j)} \cap U_{j, s(j)}} \leq 2^{-(j + 1)}q$ for all $j$.

            \item $n \geq 3$. Since $Q$ is $\Pi^0_1\paren{A^{(n - 2)}}$, by \cref{approximate_sigma01_or_pi01}, we can write $Q = \bigcap_k U_k$, where the sets $U_k$ are uniformly $\Sigma^0_1\paren{A^{(n - 2)}}$ with uniformly $A^{(n - 2)}$-computable measures.
            
            Likewise, since the sets $P_j$ are uniformly $\Pi^0_{n - 1}(A)$, we can write $P_j = \bigcap_k V_{j, k}$, where the sets $V_{j, k}$ are uniformly $\Sigma^0_{n - 2}(A)$ (with uniformly $A^{(n - 2)}$-computable measures by \cref{sigma0n(A(m))_is_sigma0(m+n)(A),measure_of_sigma0n}). Then for each $j$,
            \[\lim_{k \to \infty} \mu\paren{U_k \cap V_{j, k}} = \mu\paren{P \cap P_j} = 0,\]
            so there is an $A^{(n - 2)}$-computable function $s$ such that $\mu\paren{U_{s(j)} \cap V_{j, s(j)}} \leq 2^{-(j + 2)}q$ for all $j$.
            
            Since the sets $V_{j, k}$ are uniformly $\Sigma^0_{n - 2}(A)$, by \cref{approximate_sigma0n_with_sigma01(zero(n-1))} relativized to $A$, we can take uniformly $\Sigma^0_1\paren{A^{(n - 3)}}$ sets $U_{j, k} \supseteq V_{j, k}$ such that $\mu\paren{U_{j, k}} \leq \mu\paren{V_{j, k}} + 2^{-(j + 2)}q$ for all $j$ and $k$. Then since
            \[U_k \cap U_{j, k} \subseteq \paren{U_k \cap V_{j, k}} \cup \paren{U_{j, k} \setminus V_{j, k}},\]
            we have
            \begin{align*}
                \mu\paren{U_{s(j)} \cap U_{j, s(j)}} &\leq \mu\paren{U_{s(j)} \cap V_{j, s(j)}} + \mu\paren{U_{j, s(j)} \setminus V_{j, s(j)}} \\
                &\leq 2^{-(j + 2)}q + 2^{-(j + 2)}q = 2^{-(j + 1)}q.
            \end{align*}
        \end{itemize}
        Now regardless of the case by which we defined $U_k$, $U_{j, k}$, and $s$, observe that since $Q \subseteq U_k$ for all $k$,
        \[\mu\paren{Q \cap \bigcup_j U_{j, s(j)}} \leq \sum_j \mu\paren{Q \cap U_{j, s(j)}} \leq \sum_j \mu\paren{U_{s(j)} \cap U_{j, s(j)}} \leq \sum_j 2^{-(j + 1)}q = q.\]
        Then $R := Q \setminus \bigcup_j U_{j, s(j)}$ is $\Pi^0_1\paren{A^{(n - 2)}}$ and satisfies
        \[\mu(R) = \mu(Q) - \mu\paren{Q \cap \bigcup_j U_{j, s(j)}} \geq \mu(Q) - q > 0,\]
        so $R \in \P$. Moreover,
        \[\neg P = \bigcup_j P_j \subseteq \bigcup_j V_{j, s(j)} \subseteq \bigcup_j U_{j, s(j)},\]
        so
        \[R \subseteq \neg \bigcup_j U_{j, s(j)} \subseteq P.\]
        This implies that $R \in D_P$. Also, $R \subseteq Q$.
    \end{itemize}
    Thus, $D_P$ is dense in $\P$.
\end{proof}

Now it easily follows that $\Pi^0_n$-generic points always exist when $n \geq 2$.

\begin{cor}\label{pi0n_generic_exists}
    Let $(X, \mu)$ be a computable probability space, $n \geq 2$, and $A \subseteq \N$. Then there is a $\Pi^0_n$-generic point relative to $A$.
\end{cor}
\begin{proof}
    Define the partial order $(\P, \subseteq)$ and the sets $D_P$ as in \cref{pi0n_generic_dense}. Let $P_e$ denote the $e\th$ $\Pi^0_n(A)$ set. Also, by \cref{approximate_sigma01_inside_with_compact}, we may take a compact $\Pi^0_1$ set $K$ with positive measure.
    
    Since $K \in \P$ and each set $D_P$ is dense in $\P$, we can recursively define a nested sequence $\paren{Q_e}_e \in \prod_e D_{P_e}$ such that $Q_0 \subseteq K$. Since $K$ is compact and each set $Q_e \subseteq K$ is $\Pi^0_1\paren{A^{(n - 2)}}$ and hence closed, there exists a point $x \in \bigcap_e Q_e$.
    
    For each $\Pi^0_n(A)$ set $P_e \ni x$, we have $x \in Q_e$, which means that $P_e$ and $Q_e$ are not disjoint. Since $Q_e \in D_{P_e}$, it follows that $Q_e \subseteq P_e$. This set is $\Pi^0_1\paren{A^{(n - 2)}}$ with $\mu(Q_e) > 0$ and $x \in Q_e$, so $x$ is $\Pi^0_n$-generic point relative to $A$.
\end{proof}

We can also make a stronger claim about the cardinality of the set of $\Pi^0_n$-generic points, as long as the measure is \emph{non-atomic}, i.e., no singleton has positive measure.

\begin{thm}\label{pi0n_generic_uncountable}
    Let $(X, \mu)$ be a computable probability space, where $\mu$ is non-atomic. Let $n \geq 2$ and $A \subseteq \N$. Then the set of $\Pi^0_n$-generic points relative to $A$ has the cardinality of the continuum.
\end{thm}
\begin{proof}
    Define the partial order $(\P, \subseteq)$ and the sets $D_P$ as in \cref{pi0n_generic_dense}. Let $P_e$ denote the $e\th$ $\Pi^0_n(A)$ set. We recursively define finite sequences $\paren{Q_e^\sigma}_{e < |\sigma|}$ for all strings $\sigma$ as follows.
    
    Assume for some $\l$ that for all strings $\sigma$ with $|\sigma| \leq \l$, we have defined $\paren{Q_e^\sigma}_{e < |\sigma|}$ such that the following hold:
    \begin{enumerate}[(a)]
        \item $Q_e^\sigma \in D_{P_e}$ for all $e < |\sigma|$.
        \item $\paren{Q_e^\sigma}_{e < |\sigma|}$ is nested.
        \item $Q_e^\sigma = Q_e^\tau$ whenever $\sigma \prefix (e + 1) = \tau \prefix (e + 1)$.
        \item Whenever $\tau \neq \sigma$ has the same length as $\sigma$, the sets $Q_{\l - 1}^\sigma$ and $Q_{\l - 1}^\tau$ are disjoint.
    \end{enumerate}
    (These conditions hold vacuously for $\l = 0$ since $\paren{Q_e^\zero}_{e < 0}$ is the empty sequence.)

    Now consider any string $\sigma$ with length $\l$. By \cref{approximate_sigma01_inside_with_compact}, there exists a compact $\Pi^0_1$ set $K$ with positive measure. Define
    \[R = \begin{cases}
         K & \text{if } \l = 0 \\
         Q_{\l - 1}^\sigma & \text{if } \l > 0
    \end{cases},\]
    which lies in $\P$. Define a finite Borel measure $\nu$ on $X$ by $\nu(C) = \mu(R \cap C)$. Since $\mu$ is non-atomic, so is $\nu$:
    \[\nu(\{x\}) = \mu(R \cap \{x\}) \leq \mu(\{x\}) = 0 \quad\text{for all}\quad x \in X.\]
    Hence, given that
    \[\nu(\supp \nu) = \nu(X) = \mu(R) > 0,\]
    there must be at least two points $x_0$ and $x_1$ in the support of $\nu$. Take ideal points $s_i$ and rationals $q_i > 0$ for $i \in \{0, 1\}$ such that, letting
    \[B_i = \set{x \in X: d\paren{x, s_i} < q_i}\quad\text{and}\quad \cl B_i = \set{x \in X: d\paren{x, s_i} \leq q_i},\]
    we have $B_i \ni x_i$ for each $i$ and $\cl B_0 \cap \cl B_1 = \zero$. Then for each $i$, we have
    \[\mu\paren{R \cap \cl B_i} = \nu\paren{\cl B_i} \geq \nu\paren{B_i} > 0\]
    and hence $R \cap \cl B_i \in \P$. Since $D_{P_\l}$ is dense in $\P$, there must exist $R_i \in D_{P_\l}$ such that $R_i \subseteq R \cap \cl B_i$. Thus, we may define $\paren{Q^{\sigma i}_e}_{e \leq \l}$ by $Q^{\sigma i}_e = Q^\sigma_e$ for $e < \l$ and $Q^{\sigma i}_\l = R_i$.
    
    This clearly satisfies all of the properties stated above for strings of length up to $\l + 1$, completing the recursive definition of the finite sequences $\paren{Q^\sigma_e}_{e < |\sigma|}$ for all strings $\sigma$.

    Now given $x \in 2^\N$, define a sequence $\paren{Q^x_e}_e$ by $Q^x_e = Q^\sigma_e$, where $\sigma = x \prefix (e + 1)$. Then this sequence satisfies $Q^x_e \in D_{P_e}$ by property (a) and is nested by properties (b) and (c).

    Because $K$ is compact and the sets $Q^x_e \subseteq K$ are all closed, every intersection $\bigcap_e Q^x_e$ is nonempty. (We could even guarantee that it is a singleton if we insisted that the balls $B_i$ above each have radius no more than $2^{-\l}$.) Hence we may define a function $f: 2^\N \to X$ such that $f(x) \in \bigcap_e Q^x_e$ for all $x \in 2^\N$. This function is injective by property (d), and just as in the proof of \cref{pi0n_generic_exists}, $f(x)$ is always $\Pi^0_n$-generic relative to $A$.
\end{proof}

By following the structure of this proof, forthcoming existence results for points which satisfy some genericity or randomness properties can likewise be strengthened to obtain continuum-many such points, as long as the measure is non-atomic.

In order to speak to the non-randomness of $\Pi^0_n$-generic points, we require a lemma.

\begin{lem}\label{when_is_cylinder_in_sigma01}
    Let $A \subseteq \N$. In Cantor space, the condition $\cyl\sigma \subseteq U$ is $A$-c.e.\ in a string $\sigma$ and an index for a $\Sigma^0_1(A)$ set $U$.
\end{lem}
\begin{proof}
    Write the given $\Sigma^0_1(A)$ set $U$ as $\bigcup_{\tau \in W} \cyl\tau$ for some $A$-c.e.\ set $W \subseteq 2^{<\N}$. Let $W_s$ denote the $A$-computable approximation to $W$ by stage $s$. Then by compactness of cylinders, we have
    \[\cyl\sigma \subseteq U \iff (\exists s) \ \cyl\sigma \subseteq \bigcup_{\tau \in W_s} \cyl\tau,\]
    which is $A$-c.e.\ in $\sigma$ and an index for $U$.
\end{proof}

Note in the following result that unlike the analogous \cref{pi01_generic_not_Schnorr_random} for $\Pi^0_1$-genericity, the oracle contributes to the level of non-randomness that $\Pi^0_n$-generic points have when $n \geq 2$.

\begin{thm}\label{pi0n_generic_not_Schnorr_random}
    Consider Cantor space with the Lebesgue measure $\lambda$. Then for all $n \geq 2$ and $A \subseteq \N$, no weakly $\Pi^0_n$-generic point relative to $A$ is Schnorr $n$-random relative to $A$.
\end{thm}
\begin{proof}
    Let $P_e$ denote the $e\th$ $\Pi^0_1\paren{A^{(n - 2)}}$ set. Then by \cref{measure_of_sigma0n}, $\lambda\paren{P_e}$ is uniformly $A^{(n - 2)}$-upper-semicomputable and hence uniformly $A^{(n - 1)}$-computable. Then
    \[S := \set{e \in \N: \mu\paren{P_e} > 0}\]
    is $A^{(n - 1)}$-c.e. Let $f: \N \to \N$ be an $A^{(n - 1)}$-computable function with range $S$.
    
    Now given that $P_e \neq \zero$ for each $e \in S$, for each length we consider, there will always exist some $\sigma$ of that length satisfying $\cyl\sigma \cap P_e \neq \zero$. Hence we can define functions $\sigma_k: \N \to 2^{<\N}$ such that $\sigma_k(j)$ is the lexicographically least string $\sigma$ of length $j + k + 1$ such that $\cyl\sigma \cap P_{f(j)} \neq \zero$.

    By \cref{when_is_cylinder_in_sigma01}, the condition $\cyl\sigma \subseteq \neg P_e$ is $A^{(n - 2)}$-c.e.\ in $\sigma$ and $e$, which implies that its negation $\cyl\sigma \cap P_e \neq \zero$ is $A^{(n - 1)}$-computable in $\sigma$ and $e$. Therefore, since $f$ is also $A^{(n - 1)}$-computable, it follows that the functions $\sigma_k$ are uniformly $A^{(n - 1)}$-computable. This implies that the sets
    \[U_k := \bigcup_j \cyl{\sigma_k(j)}\]
    are uniformly $\Sigma^0_1\paren{A^{(n - 1)}}$. Observe that for each $J$ and $k$,
    \[\lambda\paren{\bigcup_{j \geq J} \cyl{\sigma_k(j)}} \leq \sum_{j \geq J} \lambda\cyl{\sigma_k(j)} = \sum_{j \geq J} 2^{-(j + k + 1)} = 2^{-(J + k)}.\]
    Thus, $\lambda\paren{U_k} \leq 2^{-k}$ is uniformly $A^{(n - 1)}$-computable, so $\paren{U_k}_k$ is a Schnorr test relative to $A^{(n - 1)}$ and hence a Schnorr $n$-test relative to $A$. This test captures every weakly $\Pi^0_n$-generic point relative to $A$ since each of the sets $U_k$ intersects every $\Pi^0_1\paren{A^{(n - 2)}}$ set with positive measure.
\end{proof}

Thus, the weakly $\Pi^0_n$-generic points form a measure-zero set when $n \geq 2$. In fact, we can say something stronger using the notion of Hausdorff dimension.

\begin{defn}
    Let $X$ be a metric space and $A \subseteq X$. For $s \geq 0$ and $\delta > 0$, we define
    \[\H^s_\delta(A) := \inf\set{\sum_j \diam\paren{U_j}^s: \paren{U_j}_j \text{ is a countable cover of $A$ with } \diam\paren{U_j} \leq \delta \text{ for all } j}.\]
    Then the \textbf{$s$-dimensional Hausdorff outer measure} of $A$ is $H^s(A) := \lim_{\delta \to 0} \H^s_\delta(A)$, and the \textbf{Hausdorff dimension} of $A$ is $\inf\set{s \geq 0: \H^s(A) = 0}$.
\end{defn}

Measure-zero subsets of Cantor space can in general have Hausdorff dimension anywhere in $[0, 1]$. The following result shows that the set of weakly $\Pi^0_n$-generic points when $n \geq 2$ remains as small as possible in this sense.

\begin{cor}
    Consider Cantor space with the Lebesgue measure $\lambda$. Then for all $n \geq 2$ and $A \subseteq \N$, the set of weakly $\Pi^0_n$-generic points relative to $A$ has Hausdorff dimension zero.
\end{cor}
\begin{proof}
    Define the functions $\sigma_k$ as in the proof of \cref{pi0n_generic_not_Schnorr_random} and fix $s > 0$ and $\delta > 0$. Recall the metric we use on $2^\N$, which is given by $d(x, y) = 2^{-\min\set{i: \ x_i \neq y_i}}$ whenever $x \neq y$. This implies that the diameter of any cylinder $\cyl\sigma$ is simply $2^{-|\sigma|}$.
    
    Take $k$ such that $2^{-(k + 1)} \leq \delta$, in which case the diameter of $\cyl{\sigma_k(j)}$ is $2^{-(j + k + 1)} \leq \delta$ for all $j$. Therefore, these cylinders form a $\delta$-cover of the set $G$ of weakly $\Pi^0_n$-generic points relative to $A$, so
    \[\H^s_\delta(G) \leq \sum_j \diam\cyl{\sigma_k(j)}^s = \sum_j 2^{-s(j + k + 1)} = \frac{2^{-s(k + 1)}}{1 - 2^{-s}}.\]
    Sending $k \to \infty$ gives us $\H^s_\delta(G) = 0$ for all $s > 0$ and $\delta > 0$. This in turn implies that $\H^s(G) = 0$ for all $s > 0$ and hence that $G$ has Hausdorff dimension zero.
\end{proof}

A similar proof exploiting the omitted argument for why weakly 1-generic points are not Schnorr random would show that they too have Hausdorff dimension zero.

There does not appear to be a simple relationship between (weakly) $\Pi^0_n$-generic points relative to different oracles when $n \geq 2$. To see why, we first make the following observations.

\begin{lem}\label{weakly_n+1_random_dense}
    Let $(X, \mu)$ be a computable probability space and $n \geq 2$. Take any $A, B \subseteq \N$ such that $A \geq_T B'$. Let $(\P, \subseteq)$ be the partial order consisting of $\Pi^0_1\paren{A^{(n - 2)}}$ sets with positive measure, ordered by the subset relation. For each null $\Pi^0_{n + 1}(B)$ set $P$, let
    \[E_P = \set{Q \in \P: Q \cap P = \zero}.\]
    Then $E_P$ is dense in $\P$ for each null $\Pi^0_{n + 1}(B)$ set $P$.
\end{lem}
\begin{proof}
    Consider any null $\Pi^0_{n + 1}(B)$ set $P$ and $Q \in \P$. We seek to show that there is some $R \in E_P$ such that $R \subseteq Q$.

    Consider any $Q \in \P$ and write $P = \bigcap_k U_k$ for some uniformly $\Sigma^0_n(B)$ sets $U_k$. Since
    \[\mu\paren{\bigcup_k \paren{Q \setminus U_k}} = \mu(Q \setminus P) = \mu(Q) > 0,\]
    there must be some $k$ such that $\mu\paren{Q \setminus U_k} > 0$. Fix $q \in \Q$ with $0 < q < \mu\paren{Q \setminus U_k}$. Since $U_k$ is $\Sigma^0_n(B)$, by \cref{approximate_sigma0n_with_sigma01(zero(n-1))} relativized to $B$, there exists a $\Sigma^0_1\paren{B^{(n - 1)}}$ set $V_k \supseteq U_k$ such that $\mu\paren{V_k} \leq \mu\paren{U_k} + q$. Then given that
    \[q < \mu\paren{Q \setminus U_k} \leq \mu\paren{Q \setminus V_k} + \mu\paren{V_k \setminus U_k} \leq \mu\paren{Q \setminus V_k} + q,\]
    we must have $\mu\paren{Q \setminus V_k} > 0$. Since $\neg V_k$ is $\Pi^0_1\paren{B^{(n - 1)}}$ and $A \geq_T B'$, this implies that $\neg V_k$ is $\Pi^0_1\paren{A^{(n - 2)}}$, just like $Q$. Thus, $R := Q \setminus V_k \in \P$. Clearly, $R$ is disjoint from $P$, so $R \in E_P$, and $R \subseteq Q$. Thus, $E_P$ is dense in $\P$.
\end{proof}

\begin{cor}\label{pi0n_generic_for_A'_and_weakly_n+1_random_for_A_exists}
    Let $(X, \mu)$ be a computable probability space, $n \geq 2$, and $A, B \subseteq \N$. Then whenever $A \geq_T B'$, there is a $\Pi^0_n$-generic point relative to $A$ which is also weakly $(n + 1)$-random relative to $B$.
\end{cor}
\begin{proof}
    Define the partial order $(\P, \subseteq)$ and the sets $D_P$ and $E_R$ as in \cref{pi0n_generic_dense,weakly_n+1_random_dense}. Let $P_e$ denote the $e\th$ $\Pi^0_n(A)$ set and $R_e$ denote the $e\th$ null $\Pi^0_{n + 1}(B)$ set (according to some non-effective enumeration of them). Also, by \cref{approximate_sigma01_inside_with_compact}, we may take a compact $\Pi^0_1$ set $K$ with positive measure.
    
    Since $K \in \P$ and the sets $D_P$ and $E_R$ are dense in $\P$, we can recursively define a nested sequence $\paren{Q_j}_j$ such that $Q_0 \subseteq K$, $Q_{2e} \in D_{P_e}$ for all $e$, and $Q_{2e + 1} \in E_{R_e}$ for all $e$. Since $K$ is compact and each set $Q_j \subseteq K$ is $\Pi^0_1\paren{A^{(n - 2)}}$ and hence closed, there exists a point $x \in \bigcap_j Q_j$, which as in \cref{pi0n_generic_exists} is $\Pi^0_n$-generic point relative to $A$.

    Also, for each null $\Pi^0_{n + 1}(B)$ set $R_e$, we have $Q_{2e + 1} \in E_{R_e}$, which means that $Q_{2e + 1} \ni x$ is disjoint from $R_e$. That is, $x \notin R_e$, so $x$ is weakly $(n + 1)$-random relative to $B$.
\end{proof}

This immediately yields the following result, showing that neither \cref{pi01_genericity_gets_stronger_with_relativization} nor \cref{weakly_pi01_generic_for_A'_implies_pi01_generic_for_A} extends to $\Pi^0_n$-genericity when $n \geq 2$.

\begin{cor}\label{pi0n_genericity_does_not_get_stronger_with_relativization}
    Consider Cantor space with the Lebesgue measure and $n \geq 2$. Then whenever $A, B \subseteq \N$ satisfy $A \geq_T B'$, there is a $\Pi^0_n$-generic point relative to $A$ which is not weakly $\Pi^0_n$-generic relative to $B$.
\end{cor}
\begin{proof}
    By \cref{pi0n_generic_for_A'_and_weakly_n+1_random_for_A_exists}, there exists a $\Pi^0_n$-generic point $x$ relative to $A$ which is also weakly $(n + 1)$-random relative to $B$ and hence Schnorr $n$-random relative to $B$. Then by \cref{pi0n_generic_not_Schnorr_random}, $x$ cannot be weakly $\Pi^0_n$-generic relative to $B$.
\end{proof}

Before returning to effective recurrence, we comment on the ``largeness'' of $\Pi^0_n$-generic points in Cantor space with the Lebesgue measure. Since the $\Pi^0_1$-generic points form a comeager set, they are large in the sense of Baire category. When $n \geq 2$, on the other hand, the collection of $\Pi^0_n$-generic points seems rather small in most ways.

Indeed, because they are all weakly 2-random and hence Schnorr random, they are disjoint from the comeager set of weakly 1-generic points, meaning that they form a meager set. Moreover, given that they are not Schnorr $n$-random, they also form a measure-zero set, which in fact has Hausdorff dimension zero. The only way in which the collection of $\Pi^0_n$-generic points seems to be large when $n \geq 2$ is by its cardinality, which is that of the continuum.

Yet the definition of weak $\Pi^0_n$-genericity could be phrased as follows: a point is weakly $\Pi^0_n$-generic for $n \geq 2$ if and only if it lies in every $\Sigma^0_n$ set that intersects every $\Pi^0_1\paren{\zero^{(n - 2)}}$ set with positive measure. This formulation is reminiscent of an alternative notion of density with respect to particular sets of positive measure. Thus, we have the following question.

\begin{qn}
    In Cantor space with the Lebesgue measure, are there any meaningful senses other than cardinality in which the collection of (weakly) $\Pi^0_n$-generic points for $n \geq 2$ is large?
\end{qn}

\subsubsection[Effective recurrence and preservation of Π0n-genericity]{Effective recurrence and preservation of $\Pi^0_n$-genericity}

We at last provide effective recurrence theorems for (weakly) $\Pi^0_n$-generic points. Since these points are not Schnorr $n$-random in $\paren{2^\N, \lambda}$ by \cref{pi01_generic_not_Schnorr_random,pi0n_generic_not_Schnorr_random}, the following results properly expand the class of recurrent points for $\Pi^0_n$ sets established so far.

\begin{thm}\label{CPS_pi0n_generic_Poincare}
    Let $(X, \mu)$ be a computable probability space, $n \geq 1$, and $A \subseteq \N$. If $x \in X$ is $\Pi^0_n$-generic relative to $A$, then for any $\Pi^0_n(A)$ set $P \ni x$ with $\mu(P) > 0$ and any computable measure-preserving transformation $T: \ \subseteq \! X \to X$ with $x \in \dom T$, we have $T^k(x) \in P$ for some $k > 0$.
\end{thm}
\begin{proof}
    Assume $x \in X$ is $\Pi^0_n$-generic relative to $A$, and let $P \ni x$ be $\Pi^0_n(A)$ with $\mu(P) > 0$. We have two cases:
    \begin{itemize}
        \item $n = 1$. Then there is a $\Sigma^0_1$ set $U \subseteq P$ with $\mu(U) > 0$ and $x \in U$. Since $x$ is Kurtz random and $U \ni x$ is $\Sigma^0_1$ with $\mu(U) > 0$, \cref{CPS_weak_n_Poincare} implies that $T^k(x) \in U \subseteq P$ for some $k > 0$.

        \item $n \geq 2$. Then there is a $\Pi^0_1\paren{A^{(n - 2)}}$ set $Q \subseteq P$ with $\mu(Q) > 0$ and $x \in Q$. By \cref{sigma0n(A(m))_is_sigma0(m+n)(A)}, $Q$ is $\Pi^0_{n - 1}(A)$. Since $x$ is weakly $n$-random relative to $A$ and $Q \ni x$ is $\Sigma^0_n(A)$ with $\mu(Q) > 0$, \cref{CPS_weak_n_Poincare} relativized to $A$ implies that $T^k(x) \in Q \subseteq P$ for some $k > 0$. \qedhere
    \end{itemize}
\end{proof}

Using \cref{CPS_weak_n_Poincare_ergodic}, a nearly identical proof gives us the same result for weakly $\Pi^0_n$-generic points and ergodic transformations.

\begin{thm}\label{CPS_weak_pi0n_generic_Poincare_ergodic}
    Let $(X, \mu)$ be a computable probability space, $n \geq 1$, and $A \subseteq \N$. If $x \in X$ is weakly $\Pi^0_n$-generic relative to $A$, then for any $\Pi^0_n(A)$ set $P \ni x$ with $\mu(P) > 0$ and any computable ergodic measure-preserving transformation $T: \ \subseteq \! X \to X$ with $x \in \dom T$, we have $T^k(x) \in P$ for some $k > 0$.
\end{thm}

Note that neither of these theorems requires any kind of computability of the measures of the $\Pi^0_n(A)$ sets $P$, unlike \cref{CPS_Schnorr_n_Poincare,CPS_Schnorr_n_Poincare_ergodic}. In \cref[S]{section_other}, we will expand the hypotheses even further through what will be called quasi-$\Pi^0_n$-genericity.

One shortcoming of these two effective recurrence theorems is that we are only guaranteed that \emph{some} iterate of the point $x$ returns to the set $P$, not that infinitely many iterates return. In order to achieve infinitely many, we would need to know that (weak) $\Pi^0_n$-genericity is preserved by the transformation $T$. Any computable measure-preserving bijection with computable inverse would clearly suffice, but there is a weaker condition that also works.

\begin{lem}\label{preservation_of_pi0n_genericity}
    Let $(X, \mu, T)$ be a computable measure-preserving system and $A \subseteq \N$.
    \begin{enumerate}[(a)]
        \item If $T$ is an open map, then $T$ preserves (weak) $\Pi^0_1$-genericity relative to $A$.
        \item For each $n \geq 2$, if $T$ preserves $\Pi^0_1\paren{A^{(n - 2)}}$ sets, then $T$ preserves (weak) $\Pi^0_n$-genericity relative to $A$.
    \end{enumerate}
\end{lem}
\begin{proof}
    Below we only prove each claim for $\Pi^0_1$-genericity because the argument for weak $\Pi^0_1$-genericity is nearly identical.
    \begin{enumerate}[(a)]
        \item Assume that $T$ is an open map. Let $x$ be $\Pi^0_1$-generic relative to $A$ and consider any $\Pi^0_1(A)$ set $P \ni T(x)$. Then $T^{-1}(P) \ni x$ is $\Pi^0_1(A)$ and $x$ is $\Pi^0_1$-generic relative to $A$, so there is an open set $U \subseteq P$ with $\mu(U) > 0$ and $x \in U$. Since $T$ is open, $T(U) \subseteq P$ is likewise open, and we also have
        \[\mu(T(U)) = \mu\paren{T^{-1}(T(U))} \geq \mu(U) > 0,\]
        along with $T(x) \in T(U)$. Therefore, $T(x)$ is $\Pi^0_1$-generic relative to $A$.

        \item Let $n \geq 2$ and assume that $T$ preserves $\Pi^0_1\paren{A^{(n - 2)}}$ sets. Let $x$ be $\Pi^0_n$-generic relative to $A$ and consider any $\Pi^0_n(A)$ set $P \ni T(x)$. Then $T^{-1}(P) \ni x$ is $\Pi^0_n(A)$ and $x$ is $\Pi^0_n$-generic relative to $A$, so there is a $\Pi^0_1\paren{A^{(n - 2)}}$ set $Q \subseteq P$ with $\mu(Q) > 0$ and $x \in Q$. By assumption, $T(Q) \subseteq P$ is likewise $\Pi^0_1\paren{A^{(n - 2)}}$, and we also have
        \[\mu(T(Q)) = \mu\paren{T^{-1}(T(Q))} \geq \mu(Q) > 0,\]
        along with $T(x) \in T(Q)$. Therefore, $T(x)$ is $\Pi^0_1$-generic relative to $A$. \qedhere
    \end{enumerate}
\end{proof}

We can use this result to show that the shift on $\paren{2^\N, \lambda}$, for example, preserves (weak) $\Pi^0_n$-genericity.

\begin{cor}
    Consider Cantor space with any computable probability measure $\mu$ preserved by the shift map $T: 2^\N \to 2^\N$. Let $n \geq 1$ and $A \subseteq \N$. Then $T$ preserves (weak) $\Pi^0_n$-genericity relative to $A$.
\end{cor}
\begin{proof}
    By \cref{shift_of_sigma01}, $T$ preserves $\Sigma^0_1$ sets and hence is an open map. Also, by \cref{shift_of_pi01}, $T$ preserves $\Pi^0_1\paren{A^{(n - 2)}}$ sets whenever $n \geq 2$. Thus, \cref{preservation_of_pi0n_genericity} implies that $T$ preserves (weak) $\Pi^0_n$-genericity relative to $A$.
\end{proof}

\subsection{Failing recurrence}

This last subsection will be devoted to proving a partial converse to the effective recurrence theorems we have provided for Schnorr $n$-random points and (weakly) $\Pi^0_n$-generic points in the case $n = 1$.

Before stating the theorem, we first prove a lemma, where $\cl U$ denotes the closure of $U$.

\begin{lem}\label{subset_of_sigma01_with_rational_measure_and_same_closure}
    Consider Cantor space with the Lebesgue measure $\lambda$. Let $V$ be a $\Sigma^0_1$ set and $q$ be a rational number such that $0 < q < \lambda(V)$. Then there is a $\Sigma^0_1$ set $U \subseteq V$ with such that $\lambda(U) = q$ and $\cl U = \cl V$.
\end{lem}
\begin{proof}
    Write $V = \cyl W$ for some prefix-free c.e.\ set $W \subseteq 2^{<\N}$. Let $W_s$ be the computable approximation to $W$ by stage $s$, and assume without loss of generality that $W_0 = \zero$. Take $s_0$ largest such that $\lambda\cyl{W_{s_0}} < q$, which implies $\lambda\cyl{W_{s_0 + 1}} \geq q$.
    
    We define an increasing computable sequence $\paren{F_n}_n$ in $\PP_{<\N}\paren{2^{<\N}}$ as follows, starting with $F_0 = W_{s_0}$. Assume for some $n$ that $F_n$ has been defined such that the following conditions hold (which are all clearly true for $n = 0$):
    \begin{enumerate}[(a)]
        \item\label{induction_inside_V} $\cyl{F_n} \subseteq \cyl{W_{s_0 + n}}$.
        \item\label{induction_covers_most_of_V} For all strings $\sigma$ with length $n$ such that $\cyl\sigma \subseteq \cyl{W_{s_0 + n}}$, we have $\cyl{F_n} \cap \cyl\sigma \neq \zero$.
        \item\label{induction_measure_q} $\paren{1 - 2^{-n}}q \leq \lambda\cyl{F_n} < q$.
    \end{enumerate}
    We define a finite set $R \subseteq 2^{<\N}$ according to the following cases.
    \begin{itemize}
        \item $\lambda\cyl{F_n} \geq \paren{1 - 2^{-(n + 1)}}q$. Then take $R = \zero$.
        
        \item $\lambda\cyl{F_n} < \paren{1 - 2^{-(n + 1)}}q$. Then we have
        \[\lambda\paren{\cyl{W_{s_0 + n + 1}} \setminus \cyl{F_n}} \geq \lambda\cyl{W_{s_0 + n + 1}} - \lambda\cyl{F_n} > q - \paren{1 - 2^{-(n + 1)}}q = 2^{-(n + 1)}q.\]
        Thus, we may take a finite set $R \subseteq 2^{<\N}$ such that
        \[\cyl R \subseteq \cyl{W_{s_0 + n + 1}} \setminus \cyl{F_n}\quad\text{and}\quad 2^{-(n + 1)}q \leq \cyl R < q - \lambda\cyl{F_n}.\]
        Then
        \[\lambda\cyl{F_n \cup R} = \lambda\cyl{F_n} + \lambda\cyl R < q,\]
        and by condition \ref{induction_measure_q},
        \[\lambda\cyl{F_n \cup R} = \lambda\cyl{F_n} + \lambda\cyl R \geq \paren{1 - 2^{-n}}q + 2^{-(n + 1)}q = \paren{1 - 2^{-(n + 1)}}q.\]
    \end{itemize}
    Either way, we now have
    \[\cyl{F_n \cup R} \subseteq \cyl{W_{s_0 + n + 1}}\quad\text{and}\quad \paren{1 - 2^{-(n + 1)}}q \leq \lambda\cyl{F_n \cup R} < q.\]
    Let
    \[S = \{\sigma \in 2^{n + 1}: \cyl\sigma \subseteq \cyl{W_{s_0 + n + 1}} \text{ and } \cyl{F_n \cup R} \cap \cyl\sigma = \zero\}.\]
    Take $k$ least such that
    \[|S|2^{-(k + n + 1)} < q - \lambda\cyl{F_n \cup R}\]
    and define
    \[F_{n + 1} = F_n \cup R \cup \{\sigma 0^k: \sigma \in S\}.\]
    Then conditions \ref{induction_inside_V} and \ref{induction_covers_most_of_V} clearly hold for $n + 1$, along with condition \ref{induction_measure_q} since
    \[\lambda\cyl{F_{n + 1}} = \lambda\cyl{F_n \cup R} + |S|2^{-(k + n + 1)} < q\]
    and
    \[\lambda\cyl{F_{n + 1}} \geq \lambda\cyl{F_n \cup R} \geq \paren{1 - 2^{-(n + 1)}}q.\]
    This completes the definition of $\paren{F_n}_n$. Now we define
    \[U := \bigcup_n \cyl{F_n},\]
    which is $\Sigma^0_1$ and contained inside $V$ by condition \ref{induction_inside_V}. We also know by condition \ref{induction_measure_q} that
    \[\lambda(U) = \lim_{n \to \infty} \lambda\cyl{F_n} = q.\]
    Finally, consider any $x \in \cl V$ and any $\tau \subset x$. Then $V \cap \cyl\tau \neq \zero$, which implies that $\cyl{W_s} \cap \cyl\tau \neq \zero$ for some $s$. Take a string $\sigma$ with $n := |\sigma| \geq s - s_0$ such that $\cyl\sigma \subseteq \cyl{W_s} \cap \cyl\tau$.
    
    Then since $\sigma$ has length $n$ and satisfies $\cyl\sigma \subseteq \cyl{W_s} \subseteq \cyl{W_{s_0 + n}}$, condition \ref{induction_covers_most_of_V} implies that $\cyl{F_n} \cap \cyl\sigma \neq \zero$. Since $\cyl{F_n} \subseteq U$ and $\cyl\sigma \subseteq \cyl\tau$, we then also have $U \cap \cyl\tau \neq \zero$.

    Thus, $U$ intersects any neighborhood of $x$, which means that $x \in \cl U$ and hence $\cl V \subseteq \cl U$. The reverse conclusion trivially holds since $U \subseteq V$, so indeed $\cl U = \cl V$.
\end{proof}

Now we prove a sufficient condition for recurrence to fail with respect to $\Pi^0_1$ sets with computable measure. Its hypothesis (neither Schnorr random nor weakly 1-generic) properly expands that of \cref{Kurtz_Poincare_characterization} (not Kurtz random) because there is a Kurtz random point which is neither Schnorr random nor weakly 1-generic.

One way to see this is by noting that any Schnorr random point is also UD-random, a notion of randomness introduced by \cite{Avigad2013}. Then since there is a Kurtz random point which is neither UD-random nor weakly 1-generic by Theorem 2.1 of \cite{Calvert2015}, this point is also not Schnorr random.

\begin{thm}\label{Schnorr_converse}
    Consider Cantor space with the Lebesgue measure $\lambda$. Let $x \in 2^\N$ be neither Schnorr random nor weakly 1-generic. Then there are a $\Pi^0_1$ set $A \ni x$ with computable $\lambda(A) > 0$ and a total computable measure-preserving transformation $T: 2^\N \to 2^\N$ such that $T^k(x) \notin A$ for all $k > 0$.
\end{thm}
\begin{proof}
    The proof is organized in several steps.

    \proofstep{Step 1: Defining the Schnorr test $\paren{U_n}_n$ and the dense $\Sigma^0_1$ set $V$ such that $x \in \bigcap_n U_n \setminus V$}

    Since $x$ is not weakly 1-generic, there exists a dense $\Sigma^0_1$ set $V \not\ni x$, which necessarily has positive measure. By \cref{subset_of_sigma01_with_rational_measure_and_same_closure}, we may assume without loss of generality that $\lambda(V) < 1$ is computable. (Note that if the original set $V$ happened to have full measure, then $x$ would not be Kurtz random. This would allow us to apply \cref{Kurtz_Poincare_characterization}, which would give us everything we want here except for totality of the transformation $T$.)

    Since $x$ is not Schnorr random, there is a nested sequence of uniformly $\Sigma^0_1$ sets $U_n$ with uniformly computable measures $\lambda\paren{U_n} \leq 2^{-n}$ such that $x \in \bigcap_n U_n$. Take uniformly clopen sets $U_{n, s}$ such that $\paren{U_{n, s}}_s$ is increasing and satisfies $U_n = \bigcup_s U_{n, s}$ for each $n$. By replacing $U_{n, s}$ with $\bigcap_{k \leq n} U_{k, s}$, which keeps $\bigcup_s U_{n, s}$ the same since $\paren{U_n}_n$ is nested, we may assume without loss of generality that $U_{k, s} \supseteq U_{n, s}$ whenever $k \leq n$.

    Below we will repeatedly use the fact that the sets $U_n \setminus U_{n, s}$ have uniformly computable measures, which holds by \cref{intersection_of_sigma0n_with_computable_measure}.

    \proofstep{Step 2: Sketching the construction}

    Our goal is to construct a computable measure-preserving transformation $T: 2^\N \to 2^\N$ such that $T^k(x) \in V$ for all $k > 0$. A naive approach would be to define $T$ only by mapping cylinders from the different levels of the Schnorr test into $V$ as we come across them. Then the domain of $T$ (which must have full measure) would be contained in $T^{-1}(V)$. This implies that
    \[\lambda(V) < 1 = \lambda(\dom T) \leq \lambda\paren{T^{-1}(V)},\]
    which means that $T$ cannot be measure-preserving. Thus, our approach must be more sophisticated than this naive one.
    
    We proceed as follows. Starting with some level $n_1$ such that $\lambda\paren{U_{n_1}} < \lambda(V)$, we would like to map all of $U_{n_1}$ into $V$. However, there will be no point at which we are guaranteed to see the entirety of this level. Thus, we can only map a finite portion of it into $V$ before being forced to decide where the rest of the space will go to satisfy $\lambda\paren{T^{-1}\cyl0} = \lambda\paren{T^{-1}\cyl1} = 1/2$.

    Because $V$ is dense, there are sure to be some cylinders $\cyl{\tau_0} \subseteq V \cap \cyl0$ and $\cyl{\tau_1} \subseteq V \cap \cyl 1$. These will constitute a ``reserved space,'' i.e., a portion of $V$ we choose not to use right now so that we can rely on it later.
    
    Then if we enumerate enough of $U_{n_1}$ so that the remaining portion could fit into either of these cylinders and then map this portion into $V \setminus \paren{\cyl{\tau_0} \cup \cyl{\tau_1}}$, we can freely map the rest of the space in a measure-preserving way to $\cyl 0$ and $\cyl 1$, knowing that each of these cylinders has enough unused portions of $V$ to accommodate the rest of $U_{n_1}$ that we have not yet seen. This specifies the first bit of the output for all input sequences.

    Choosing a suitable stage $s_1$ for this can be done effectively because $\lambda\paren{U_{n_1}}$ is computable. We can also make sure that we do not need to use either $\cyl{\tau_0}$ or $\cyl{\tau_1}$ when mapping $U_{n_1, s_1}$ by taking these cylinders to be sufficiently small.

    For the next stage, we follow a similar procedure:
    \begin{enumerate}[1.]
        \item Take $n_2$ sufficiently large so that $U_{n_1} \setminus U_{n_1, s_1}$ and $U_{n_2}$ could be mapped into the unused portion of $V$ in either $\cyl0$ or $\cyl 1$.
        \item Find suitable reserves $\cyl{\tau_\theta} \subseteq \cyl\theta$ for each cylinder $\cyl\theta$ with $|\theta| = 2$ that we have not exhausted in mapping $U_{n_1, s_1}$.
        \item Choose $s_2$ large enough so that the remaining portions of both $U_{n_1}$ and $U_{n_2}$ beyond stage $s_2$ could fit into any one of these cylinders $\cyl{\tau_\theta}$.
        \item Define $T$ so that both $T\paren{U_{n_1, s_2}}$ and $T^2\paren{U_{n_2, s_2}}$ are inside $V$ and do not use up the reserve cylinders $\cyl{\tau_\theta}$.
        \item Divide up the rest of the space into clopen sets $\cyl{H_\theta}$ of suitable measure and map them to the cylinders $\cyl\theta$ with $|\theta| = 2$ so that the second bit of the output is now specified for all input sequences.
    \end{enumerate}

    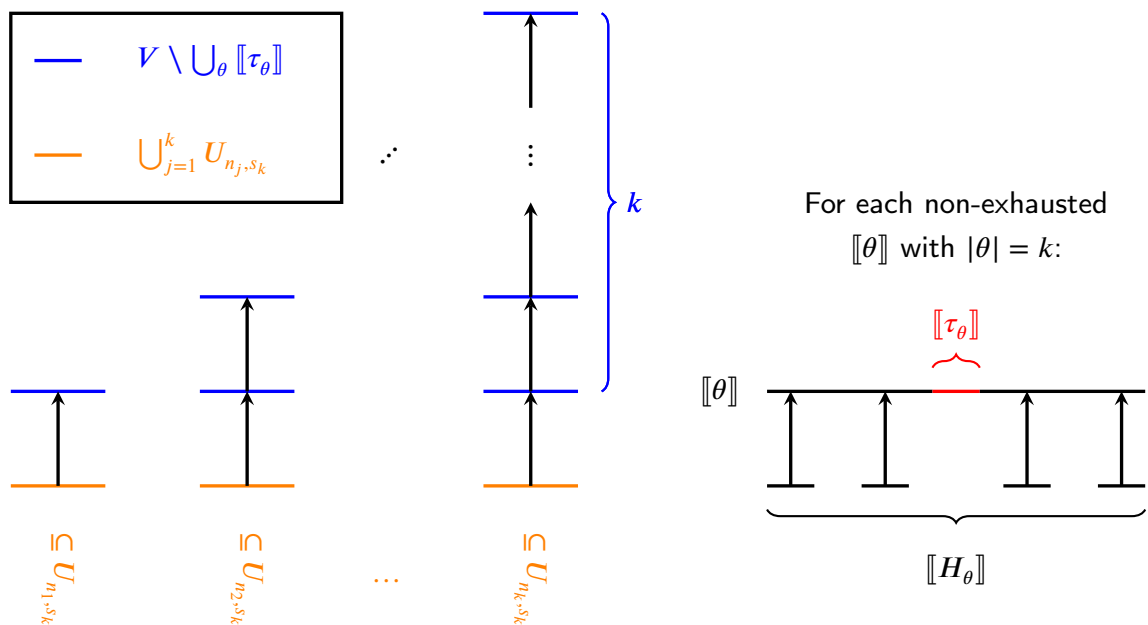
\begin{figure}
        \centering
        \begin{tikzpicture}[line width=1.25pt, scale=1.25, transform shape]
        
            % U_{n_1, s_k}
        
            \node[orange, rotate=-90] at (0.5, 1) {$\subseteq U_{n_1, s_k}$};
        
            \draw[orange] (0,2) -- (1,2);
            \draw[blue] (0,3) -- (1,3);
            
            \draw[-stealth] (0.5, 2) -- (0.5, 3);
        
            % U_{n_2, s_k}
        
            \node[orange, rotate=-90] at (2.5, 1) {$\subseteq U_{n_2, s_k}$};
        
            \draw[orange] (2,2) -- (3,2);
            \draw[blue] (2,3) -- (3,3);
            \draw[blue] (2,4) -- (3,4);
            
            \draw[-stealth] (2.5, 2) -- (2.5, 3);
            \draw[-stealth] (2.5, 3) -- (2.5, 4);
        
            \draw [thick, blue, decorate, decoration={brace, amplitude=5pt}] (6.25,7) -- (6.25,3) node [midway, right=4pt] {$k$};
        
            \node[orange] at (4,1) {$\dots$};
        
            \node at (4,5.5) {\reflectbox{$\ddots$}};
        
            % U_{n_k, s_k}
        
            \node[orange, rotate=-90] at (5.5, 1) {$\subseteq U_{n_k, s_k}$};
        
            \draw[orange] (5,2) -- (6,2);
            \draw[blue] (5,3) -- (6,3);
            \draw[blue] (5,4) -- (6,4);
            \draw[blue] (5,7) -- (6,7);
            
            \draw[-stealth] (5.5, 2) -- (5.5, 3);
            \draw[-stealth] (5.5, 3) -- (5.5, 4);
            \draw[-stealth] (5.5, 4) -- (5.5, 5);
            \draw[-stealth] (5.5, 6) -- (5.5, 7);
        
            \node at (5.5,5.5) {$\vdots$};
        
            \draw [thick, blue, decorate, decoration={brace, amplitude=5pt}] (6.25,7) -- (6.25,3) node [midway, right=4pt] {$k$};
        
            % H_\theta
        
            \draw [thick, decorate, decoration={brace, amplitude=5pt}] (12,1.75) -- (8,1.75) node [midway, below=10pt] {$\cyl{H_\theta}$};
        
            \draw (8,2) -- (8.5,2);
            \draw (9,2) -- (9.5,2);
            \draw (10.5,2) -- (11,2);
            \draw (11.5,2) -- (12,2);
        
            \draw[-stealth] (8.25, 2) -- (8.25, 3);
            \draw[-stealth] (9.25, 2) -- (9.25, 3);
            \draw[-stealth] (10.75, 2) -- (10.75, 3);
            \draw[-stealth] (11.75, 2) -- (11.75, 3);
        
            % theta
        
            \node at (7.5,3) {$\cyl\theta$};
            \draw (8,3) -- (9.75,3);
            \draw[red] (9.75,3) -- (10.25,3);
            \draw (10.25,3) -- (12,3);
        
            \draw[red, thick, decorate, decoration={brace, amplitude=5pt}] (9.75,3.25) -- (10.25,3.25) node [midway, above=5pt] {$\cyl{\tau_\theta}$};
        
            \node at (10,5) {For each non-exhausted};
            \node at (10,4.5) {$\cyl\theta$ with $|\theta| = k$:};
        
            % Legend
        
            \draw[draw=black] (0,5) rectangle ++(3.5,2);
            \draw[blue] (0.25, 6.5) -- (0.75, 6.5) node [midway, right=20pt] {$V \setminus \bigcup_\theta \cyl{\tau_\theta}$};
        
            \draw[orange] (0.25, 5.5) -- (0.75, 5.5) node [midway, right=20pt] {$\bigcup_{j = 1}^k U_{n_j, s_k}$};
        \end{tikzpicture}
        \caption{The constructed transformation by stage $k$}
        \label{Schnorr_picture}
    \end{figure}

    The pattern continues with future stages, as illustrated in \cref[S]{Schnorr_picture}. Observe that by the end of stage $k$, we have guaranteed that $T^k\paren{U_{n_j, s_k}} \subseteq V$ for all $1 \leq j \leq k$ and have specified the first $k$ bits of the output for all input sequences. Thus, we will obtain a total computable measure-preserving transformation $T$ such that $T^k(x) \in V$ for all $k > 0$.

    With this picture in mind, we will now proceed to flesh out the details of the construction with full rigor.
    
    \proofstep{Step 3: Setting up the induction}

    Our transformation $T: 2^\N \to 2^\N$ will be constructed as the limit of a computable monotone function $t: \bigcup_k 2^{\l_k} \to 2^{<\N}$. The $k$-fold iterate of $t$ (wherever it is defined) is denoted by $t^k$, and $t^0$ is the identity.

    We will now recursively define sequences $\paren{n_k}_k$, $\paren{s_k}_k$, and $\paren{\l_k}_k$ in $\N$, sequences $\paren{D_k}_k$, $\paren{R_k}_k$, and $\paren{S_k}_k$ in $\PP_{<\N}\paren{2^{<\N}}$, and a map $t: \bigcup_k 2^{\l_k} \to 2^{<\N}$ as follows, starting with
    \[n_0 = s_0 = \l_0 = 0, \quad D_0 = R_0 = \zero,\quad S_0 = \{\zero\}, \quad\text{and}\quad t(\zero) = \zero.\]
    Assume for some $k$ that $n_j$, $s_j$, $\l_j$, $D_j$, $R_j$, $S_j$, and $t(\sigma)$ have been defined for all $j \leq k$ and $\sigma \in \bigcup_{j \leq k} 2^{\l_j}$ subject to the following conditions (which are all trivial when $k = 0$):
    \begin{enumerate}[(A)]
        \item\label{increasing_sequences} (Strictly increasing sequences) $n_j < n_k$, $s_j < s_k$, and $\l_j < \l_k$ whenever $j < k$.
        \item\label{refined} (Sufficiently refined) $U_{n_j, s_k}$ can be written as a union of cylinders $\cyl\sigma$ with $|\sigma| = \l_k$ whenever $1 \leq j \leq k$.
        \item\label{determined_domain} (Determined domain) $D_k = \set{t^i(\sigma): i < j \leq k, \quad |\sigma| = \l_k, \quad \text{and}\quad \cyl\sigma \subseteq U_{n_j, s_k}}$.
        \item\label{determined_range} (Determined range) $R_k \subseteq 2^{\l_k}$ and $\cyl{R_k} \subseteq V$
        \item\label{mapping_onto_Rk} (Mapping to the determined range) $t$ injectively maps $D_k$ onto $R_k$.
        \item\label{undetermined_range} (Undetermined range) $S_k = \set{\rho \in 2^k: \cyl\rho \not\subseteq \cyl{R_k}}$.
        \item\label{small_measure_remaining} (Small measure remaining) $\sum_{j = 1}^k \lambda\paren{U_{n_j} \setminus U_{n_j, s_k}} < \min_{\rho \in S_k} \lambda\paren{V \cap \paren{\cyl\rho \setminus \cyl{R_k}}}$.
        \item\label{mapping_onto_Sk} (Mapping to the undetermined range) $t$ maps $2^{\l_k} \setminus D_k$ onto $S_k$.
        \item\label{measure_preserving} (Measure preservation) $\lambda\cyl{\set{\sigma \in 2^{\l_k}: t(\sigma) \supseteq \rho}} = \lambda\cyl\rho$ whenever $|\rho| = k$.
        \item\label{monotonicity} (Monotonicity) $t\paren{\sigma \prefix \l_j} \subseteq t(\sigma)$ whenever $|\sigma| = \l_k$ and $j < k$.
    \end{enumerate}
    Note that $S_k$ is always nonempty since $\cyl{R_k} \subseteq V$, where $V \not\ni x$ is not all of $2^\N$.

    \proofstep{Step 4: Finding suitable $n_{k + 1}$, $s_{k + 1}$, and $\l_{k + 1}$ to get conditions \ref{increasing_sequences} and \ref{refined} for $k + 1$}
    
    By condition \ref{small_measure_remaining}, we can computably find $n_{k + 1} > n_k$ such that
    \[\lambda\paren{U_{n_{k + 1}}} < \min_{\rho \in S_k} \lambda\paren{V \cap \paren{\cyl\rho \setminus \cyl{R_k}}} - \sum_{j = 1}^k \lambda\paren{U_{n_j} \setminus U_{n_j, s_k}} =: \delta.\]
    Note that the definition of $\delta$ implies that for any $\rho \in S_k$,
    \begin{equation}\label{bound_sums_of_remaining_measures}
        \sum_{j = 1}^k \lambda\paren{U_{n_j} \setminus U_{n_j, s_k}} \leq \lambda\paren{V \cap \paren{\cyl\rho \setminus \cyl{R_k}}} - \delta.
    \end{equation}
    Now define
    \[S_{k + 1} = \set{\theta \in 2^{k + 1}: \cyl\theta \not\subseteq \cyl{R_k}}.\]
    Then for each string $\theta \in S_{k + 1}$, by density of $V$, we can computably find a string $\tau_\theta$ such that
    \[\cyl{\tau_\theta} \subseteq V \cap \paren{\cyl\theta \setminus \cyl{R_k}}\quad\text{and}\quad \lambda\cyl{\tau_\theta} < \frac{\delta - \lambda\paren{U_{n_{k + 1}}}}{2}.\]
    Observe that each string $\rho \in S_k$ has at most two extensions in $S_{k + 1}$, which implies
    \[\lambda\paren{\bigcup_{\substack{\theta \in S_{k + 1} \\ \rho \subset \theta}} \cyl{\tau_\theta}} < \delta - \lambda\paren{U_{n_{k + 1}}}\]
    and hence, combined with inequality (\ref{bound_sums_of_remaining_measures}),
    \begin{equation}\label{bound_sum_of_remaining_measures_and_k+1}
        \sum_{j = 1}^k \lambda\paren{U_{n_j} \setminus U_{n_j, s_k}} + \lambda\paren{U_{n_{k + 1}}} < \lambda\paren{V \cap \paren{\cyl\rho \setminus \cyl{R_k}}} - \lambda\paren{\bigcup_{\substack{\theta \in S_{k + 1} \\ \rho \subset \theta}} \cyl{\tau_\theta}}.
    \end{equation}
    Since $\lambda\paren{U_n \setminus U_{n, s}} \to 0$ as $s \to \infty$ for each $n$, we can computably find $s_{k + 1} > s_k$ such that
    \begin{equation}\label{bound_sums_of_remaining_measures_after_sk+1}
        \sum_{j = 1}^{k + 1} \lambda\paren{U_{n_j} \setminus U_{n_j, s_{k + 1}}} < \min_{\theta \in S_{k + 1}} \lambda\cyl{\tau_\theta}.
    \end{equation}
    Then for each string $\rho \in S_k$, by inequality (\ref{bound_sum_of_remaining_measures_and_k+1}), we have
    \begin{align*}
        \sum_{j = 1}^k \lambda\paren{U_{n_j, s_{k + 1}} \setminus U_{n_j, s_k}} + \lambda\paren{U_{n_{k + 1}, s_{k + 1}}} &\leq \sum_{j = 1}^k \lambda\paren{U_{n_j} \setminus U_{n_j, s_k}} + \lambda\paren{U_{n_{k + 1}}} \\
        &< \lambda\paren{\paren{V \cap \cyl\rho} \setminus \paren{\cyl{R_k} \cup \bigcup_{\substack{\theta \in S_{k + 1} \\ \rho \subset \theta}} \cyl{\tau_\theta}}}.
    \end{align*}
    Hence we can computably find a clopen set $C_\rho$ such that
    \[C_\rho \subseteq \paren{V \cap \cyl\rho} \setminus \paren{\cyl{R_k} \cup \bigcup_{\substack{\theta \in S_{k + 1} \\ \rho \subset \theta}} \cyl{\tau_\theta}}\quad\text{and}\quad\lambda\paren{C_\rho} = \sum_{j = 1}^k \lambda\paren{U_{n_j, s_{k + 1}} \setminus U_{n_j, s_k}} + \lambda\paren{U_{n_{k + 1}, s_{k + 1}}}.\]
    Take $\l_{k + 1} > \l_k$ such that for all $1 \leq j \leq k$ and all $\rho \in S_k$, there exist $F_j, F_{k + 1}, G_\rho \subseteq 2^{\l_{k + 1}}$ satisfying
    \[U_{n_j, s_{k + 1}} \setminus U_{n_j, s_k} = \cyl{F_j}, \quad U_{n_{k + 1}, s_{k + 1}} = \cyl{F_{k + 1}}, \quad\text{and}\quad C_\rho = \cyl{G_\rho}.\]
    Conditions \ref{increasing_sequences} and \ref{refined} and now satisfied for $k + 1$. Write
    \[G_\rho = \bigcup_{j = 1}^{k + 1} G_{\rho, j}\]
    for some disjoint sets $G_{\rho, j} \subseteq 2^{\l_{k + 1}}$ such that $\lambda\cyl{G_{\rho, j}} = \lambda\cyl{F_j}$.

    \proofstep{Step 5: Defining $t$ on $D_{k + 1}$ to get condition \ref{determined_domain} for $k + 1$}

    Assume for some $j \leq k + 1$ that for all $1 \leq i \leq j$ and strings $\sigma$ of length $\l_{k + 1}$ with $\cyl\sigma \subseteq U_{n_i, s_{k + 1}}$, we have defined $t^i(\sigma)$ such that the following properties hold:
    \begin{enumerate}[(a)]
        \item\label{t_length} $|t^i(\sigma)| = \l_{k + 1}$.
        
        \item\label{t_in_V} $\cyl{t^i(\sigma)} \subseteq V \setminus \bigcup_{\theta \in S_{k + 1}} \cyl{\tau_\theta}$.

        \item\label{t_mapping_prefix} If $\eta := t^{i - 1}(\sigma) \prefix \l_k \in D_k$, then writing $t^{i - 1}(\sigma) = \eta\tau$, we have $t^i(\sigma) = t(\eta)\tau$.

        \item\label{t_if_not_in_Dk_then_maps_out_of_Rk} If $\eta := t^{i - 1}(\sigma) \prefix \l_k \notin D_k$, then $\sigma \in F_i$ and $\rho := t(\eta) \in S_k$. Moreover, if $\sigma$ is the $m\th$ string lexicographically in $F_i$, then $t^i(\sigma)$ is the $m\th$ string lexicographically in $G_{\rho, i}$.

        \item\label{t_mapping_prefix_corollary} If $i \leq k$ and $\cyl{\sigma \prefix \l_k} \subseteq U_{n_i, s_k}$, then writing $\sigma = \paren{\sigma \prefix \l_k}\tau$, we have $t^i(\sigma) = t^i\paren{\sigma \prefix \l_k}\tau$.
        
        \item\label{t_induction_injective} If $t^i(\sigma) = t^{i'}(\sigma')$ for some $i' \leq i$ and some string $\sigma'$ of length $\l_{k + 1}$ with $\cyl{\sigma'} \subseteq U_{n_{i'}, s_{k + 1}}$, then $t^{i - i'}(\sigma) = \sigma'$.
    \end{enumerate}
    (This holds vacuously when $j = 0$.)

    Now assume $j \leq k$ and consider any string $\sigma$ of length $\l_{k + 1}$ with $\cyl\sigma \subseteq U_{n_{j + 1}, s_{k + 1}}$. Since $U_{n_{j + 1}, s_{k + 1}} \subseteq U_{n_j, s_{k + 1}}$, we have already defined $t^j(\sigma)$ to be a string of length $\l_{k + 1}$.

    Let $\eta = t^j(\sigma) \prefix \l_k$. We will use the following lemma.
    \begin{lem}\label{eta_in_Dk}
        Assume $j < k$ and $\cyl{\sigma \prefix \l_k} \subseteq U_{n_{j + 1}, s_k}$. Then $\eta = t^j\paren{\sigma \prefix \l_k} \in D_k$.
    \end{lem}
    \begin{proof}
        The hypothesis implies $\cyl{\sigma \prefix \l_k} \subseteq U_{n_j, s_k}$, so
        \[\eta = t^j(\sigma) \prefix \l_k = t^j\paren{\sigma \prefix \l_k},\]
        where the latter equality holds trivially if $j = 0$ and by property \ref{t_mapping_prefix_corollary} if $j > 0$. Because $\cyl{\sigma \prefix \l_k} \subseteq U_{n_{j + 1}, s_k}$, where $j + 1 \leq k$, it follows that $\eta \in D_k$.
    \end{proof}
    
    To define $t^{j + 1}(\sigma)$, there are two cases:
    \begin{enumerate}[1.]
        \item $\eta \in D_k$. Then writing $t^j(\sigma) = \eta\tau$, we define $t^{j + 1}(\sigma) = t(\eta)\tau$.

        \item $\eta \notin D_k$. Then by condition \ref{mapping_onto_Sk}, $\rho := t(\eta) \in S_k$. We also have $\sigma \in F_{j + 1}$ according to the following two subcases:
        \begin{enumerate}[label=\roman*., ref=\theenumi\roman*]            
            \item\label{not_in_domain_j<k} $j < k$. Then by \cref{eta_in_Dk}, we must have $\cyl\sigma \subseteq \cyl{\sigma \prefix \l_k} \subseteq U_{n_{j + 1}, s_{k + 1}} \setminus U_{n_{j + 1}, s_k}$, which implies that $\sigma \in F_{j + 1}$.
            \item $j = k$. Then since $\cyl\sigma \subseteq U_{n_{k + 1}, s_{k + 1}}$, we clearly have $\sigma \in F_{k + 1}$.
        \end{enumerate}
        Taking $m$ such that $\sigma$ is the $m\th$ string lexicographically in $F_{j + 1}$, we define $t^{j + 1}(\sigma)$ to be the $m\th$ string lexicographically in $G_{\rho, j + 1}$.
    \end{enumerate}
    We now verify the induction properties for $j + 1$.
    \begin{enumerate}[(a)]
        \item We have two cases:
        \begin{itemize}
            \item $\eta \in D_k$. By condition \ref{mapping_onto_Rk}, we know that $|t(\eta)| = \l_k$ and hence $\abs{t^{j + 1}(\sigma)} = \l_{k + 1}$.
            \item $\eta \notin D_k$. Since $G_{\rho, j + 1} \subseteq 2^{\l_{k + 1}}$, we know that $\abs{t^{j + 1}(\sigma)} = \l_{k + 1}$.
        \end{itemize}
        \item We have two cases:
        \begin{itemize}
            \item $\eta \in D_k$. By condition \ref{determined_range} and the definition of the strings $\tau_\theta$, we have
            \[\cyl{t^{j + 1}(\sigma)} \subseteq \cyl{t(\eta)} \subseteq \cyl{R_k} \subseteq V \setminus \bigcup_{\theta \in S_k} \cyl{\tau_\theta}.\]
            \item $\eta \notin D_k$. Then
            \[\cyl{t^{j + 1}(\sigma)} \subseteq \cyl{G_{\rho, j + 1}} \subseteq \cyl{G_\rho} = C_\rho \subseteq V \setminus \bigcup_{\theta \in S_k} \cyl{\tau_\theta}.\]
        \end{itemize}
        \item This property with $j + 1$ in place of $j$ is clear from the definition of $t^{j + 1}(\sigma)$.

        \item This property with $j + 1$ in place of $j$ is clear from the definition of $t^{j + 1}(\sigma)$.

        \item Assume $j + 1 \leq k$ and $\cyl{\sigma \prefix \l_k} \subseteq U_{n_{j + 1}, s_k}$. Then by \cref{eta_in_Dk}, $\eta = t^j\paren{\sigma \prefix \l_k} \in D_k$. Writing $t^j(\sigma) = \eta\tau$, property \ref{t_mapping_prefix} with $j + 1$ implies
        \[t^{j + 1}(\sigma) = t(\eta)\tau = t^{j + 1}\paren{\sigma \prefix \l_k}\tau.\]
        \item Suppose $t^{j + 1}(\sigma) = t^i\paren{\sigma'}$ for some $i \leq j + 1$ and some string $\sigma'$ of length $\l_{k + 1}$ with $\cyl{\sigma'} \subseteq U_{n_i, s_{k + 1}}$. If $i = 0$, then trivially $t^{j + 1 - i}(\sigma) = \sigma'$, so assume $i \geq 1$ and let $\eta' = t^{i - 1}\paren{\sigma'} \prefix \l_k$. We have two cases:
        \begin{itemize}
            \item $\eta \in D_k$. Write $t^j(\sigma) = \eta\tau$ so that $t^{j + 1}(\sigma) = t(\eta)\tau$. Then since
            \[\cyl{t^i\paren{\sigma'}} = \cyl{t^{j + 1}(\sigma)} \subseteq \cyl{t(\eta)} \subseteq \cyl{R_k}\]
            and $\cyl{R_k}$ is disjoint from $C_\rho = \cyl{G_\rho} \supseteq G_{\rho, i}$, it follows by property \ref{t_if_not_in_Dk_then_maps_out_of_Rk} with $j + 1$ that $\eta' \in D_k$. Then writing $t^{i - 1}\paren{\sigma'} = \eta'\tau'$, property \ref{t_mapping_prefix} with $j + 1$ implies
            \[t(\eta)\tau = t^{j + 1}(\sigma) = t^i\paren{\sigma'} = t\paren{\eta'}\tau'.\]
            Thus, $t(\eta) = t\paren{\eta'}$ and $\tau = \tau'$. We know by condition \ref{mapping_onto_Rk} that $t$ is injective on $D_k$, so we must have $\eta = \eta'$. Then
            \[t^j(\sigma) = \eta\tau = \eta'\tau' = t^{i - 1}\paren{\sigma'},\]
            so property \ref{t_induction_injective} with $j$ implies that
            \[t^{(j + 1) - i}(\sigma) = t^{j - (i - 1)}(\sigma) = \sigma'.\]
            \item $\eta \notin D_k$. Let $\rho = t(\eta) \in S_k$. Then taking $m$ such that $\sigma$ is the $m\th$ string lexicographically in $F_{j + 1}$, we have defined $t^{j + 1}(\sigma)$ to be the $m\th$ string lexicographically in $G_{\rho, j + 1}$. Then
            \[\cyl{t^i\paren{\sigma'}} = \cyl{t^{j + 1}(\sigma)} \subseteq \cyl{G_{\rho, j + 1}} \subseteq \cyl{G_\rho} = C_\rho.\]
            If we had $\eta' \in D_k$, then property \ref{t_mapping_prefix} for $j + 1$ would imply
            \[\cyl{t^i\paren{\sigma'}} \subseteq \cyl{t\paren{\eta'}} \subseteq \cyl{R_k},\]
            yet $C_\rho$ is disjoint from $\cyl{R_k}$. Hence $\eta' \notin D_k$. Let $\rho' = t(\eta') \in S_k$.
            
            Now by property \ref{t_if_not_in_Dk_then_maps_out_of_Rk} with $j + 1$, if we take $m'$ such that $\sigma'$ is the $\paren{m'}\th$ string lexicographically in $F_i$, then $t^i(\sigma')$ is the $(m')\th$ string lexicographically in $G_{\rho', i}$. Then $\cyl{t^{j + 1}(\sigma)} = \cyl{t^i\paren{\sigma'}}$ lies inside both
            \[\cyl{G_{\rho, j + 1}} \subseteq \cyl{G_\rho} = C_\rho \subseteq \cyl\rho\quad\text{and}\quad\cyl{G_{\rho', i}} \subseteq \cyl{G_{\rho'}} = C_{\rho'} \subseteq \cyl{\rho'},\]
            so we must have $\rho = \rho'$. This implies $j + 1 = i$ since otherwise, $G_{\rho, j + 1}$ and $G_{\rho', i}$ would be disjoint. It follows that $m = m'$ and hence $t^{(j + 1) - i}(\sigma) = \sigma = \sigma'$.
        \end{itemize}
    \end{enumerate}
    This completes the inductive definition of $t$ on
    \[D_{k + 1} := \set{t^i(\sigma): i < j \leq k + 1, \quad |\sigma| = \l_{k + 1}, \quad\text{and}\quad \cyl\sigma \subseteq U_{n_j, s_{k + 1}}},\]
    which gives condition \ref{determined_domain} for $k + 1$. Let $R_{k + 1} = t(D_{k + 1})$.

    Later on, it will be important that $t$ is monotone on $D_{k + 1}$.

    \begin{lem}\label{t_monotone}
        For all $\sigma \in D_{k + 1}$, we have $t\paren{\sigma \prefix \l_k} \subseteq t(\sigma)$.
    \end{lem}
    \begin{proof}
        Consider any $i < j \leq k$ and any string $\sigma$ of length $\l_{k + 1}$ such that $\cyl\sigma \subseteq U_{n_j, s_{k + 1}}$. We must show that
        \[t\paren{t^i(\sigma) \prefix \l_k} \subseteq t^{i + 1}(\sigma).\]
        Let $\eta = t^i(\sigma) \prefix \l_k$. We have two cases:
        \begin{itemize}
            \item $\eta \in D_k$. Then writing $t^i(\sigma) = \eta\tau$, property \ref{t_mapping_prefix} implies
            \[t(\eta) \subseteq t(\eta)\tau = t^{i + 1}(\sigma).\]
            \item $\eta \notin D_k$. By property \ref{t_if_not_in_Dk_then_maps_out_of_Rk}, we have $\rho := t(\eta) \in S_k$ and $t^{i + 1}(\sigma) \in G_{\rho, i + 1}$. Since
            \[\cyl{G_{\rho, i + 1}} \subseteq \cyl{G_\rho} = C_\rho \subseteq \cyl\rho,\]
            it follows that
            \begin{align*}
                t(\eta) &= \rho \subseteq t^{i + 1}(\sigma). \qedhere
            \end{align*}
        \end{itemize}
    \end{proof}

    \proofstep{Step 6: Proving conditions \ref{determined_range}, \ref{mapping_onto_Rk}, \ref{undetermined_range}, and \ref{small_measure_remaining} for $k + 1$}
    
    The following lemma gives us condition \ref{determined_range} for $k + 1$.

    \begin{lem}\label{Rk+1_in_V}
        $R_{k + 1} \subseteq 2^{\l_{k + 1}}$ and $\cyl{R_{k + 1}} \subseteq V \setminus \bigcup_{\theta \in S_{k + 1}} \cyl{\tau_\theta}$.
    \end{lem}
    \begin{proof}
        This holds by properties \ref{t_length} and \ref{t_in_V}.
    \end{proof}

    The following lemma gives us condition \ref{mapping_onto_Rk} for $k + 1$.

    \begin{lem}\label{t_injective}
        $t$ is injective on $D_{k + 1}$.
    \end{lem}
    \begin{proof}
        Suppose that $t(\tau) = t\paren{\tau'}$ for some strings $\tau, \tau' \in D_{k + 1}$. Write $\tau = t^i(\sigma)$ for some $i < j \leq k + 1$ and some string $\sigma$ with length $\l_{k + 1}$ such that $\cyl\sigma \subseteq U_{n_j, s_{k + 1}}$. Likewise, write $\tau' = t^{i'}\paren{\sigma'}$ for some $i' < j' \leq k + 1$ and some string $\sigma'$ with length $\l_{k + 1}$ such that $\cyl{\sigma'} \subseteq U_{n_{j'}, s_{k + 1}}$. Say $i' \leq i$. Then
        \[t^{i + 1}(\sigma) = t(\tau) = t\paren{\tau'} = t^{i' + 1}\paren{\sigma'},\]
        where $\cyl\sigma \subseteq U_{n_{i + 1}, s_{k + 1}}$ and $\cyl{\sigma'} \subseteq U_{n_{i' + 1}, s_{k + 1}}$. By property \ref{t_induction_injective}, it follows that
        \[t^{i - i'}(\sigma) = t^{(i + 1) - (i' + 1)}(\sigma) = \sigma'\]
        and hence
        \begin{align*}
            \tau &= t^i(\sigma) = t^{i'}\paren{\sigma'} = \tau'. \qedhere
        \end{align*}
    \end{proof}

    It will be useful to know that the clopen sets generated by the determined domain and range are increasing from $k$ to $k + 1$.

    \begin{lem}\label{Dk_in_Dk+1}
        $\cyl{D_k} \subseteq \cyl{D_{k + 1}}$ and  $\cyl{R_k} \subseteq \cyl{R_{k + 1}}$.
    \end{lem}
    \begin{proof}
        Consider any string $\eta \in D_k$ and any string $\tau$ of length $\l_{k + 1} - \l_k$. Write $\eta = t^i(\sigma)$ for some $i < j \leq k$ and some string $\sigma$ of length $\l_k$ such that $\cyl\sigma \subseteq U_{n_j, s_k}$. Then also $\cyl\sigma \subseteq U_{n_i, s_k}$, so property \ref{t_mapping_prefix_corollary} tells us that
        \[t^i(\sigma\tau) = t^i(\sigma)\tau = \eta\tau.\]
        Since 
        \[\cyl{\sigma\tau} \subseteq \cyl\sigma \subseteq U_{n_j, s_{k + 1}},\]
        we have $t^i(\sigma\tau) \in D_{k + 1}$. Thus,
        \[\cyl{\eta\tau} = \cyl{t^i(\sigma\tau)} \subseteq \cyl{D_{k + 1}}.\]
        As $\tau \in 2^{\l_{k + 1} - \l_k}$ was arbitrary, it follows that $\cyl\eta \subseteq \cyl{D_{k + 1}}$. Therefore, $\cyl{D_k} \subseteq \cyl{D_{k + 1}}$, which implies $\cyl{R_k} \subseteq \cyl{R_{k + 1}}$.
    \end{proof}

    The following lemma gives us condition \ref{undetermined_range} for $k + 1$.

    \begin{lem}
        $S_{k + 1} = \{\theta \in 2^{k + 1}: \cyl\theta \not\subseteq \cyl{R_{k + 1}}\}$.
    \end{lem}
    \begin{proof}
        Recall that our original definition of $S_{k + 1}$ had $R_k$ in place of $R_{k + 1}$. The desired equality holds as follows:
        \begin{itemize}
            \item $\subseteq$. Assume $\theta \in S_{k + 1}$. Then by \cref{Rk+1_in_V}, $\cyl{\tau_\theta} \subseteq \cyl\theta$ is disjoint from $\cyl{R_{k + 1}}$, so $\cyl\theta \not\subseteq \cyl{R_{k + 1}}$.

            \item $\supseteq$. Assume $\theta \in 2^{k + 1} \setminus S_{k + 1}$. Then by \cref{Dk_in_Dk+1}, $\cyl\theta \subseteq \cyl{R_k} \subseteq \cyl{R_{k + 1}}$. \qedhere
        \end{itemize}
    \end{proof}

    The following lemma gives us condition \ref{small_measure_remaining} for $k + 1$.

    \begin{lem}
        We have
        \[\sum_{j = 1}^{k + 1} \lambda\paren{U_{n_j} \setminus U_{n_j, s_{k + 1}}} < \min_{\theta \in S_{k + 1}} \lambda\paren{V \cap \paren{\cyl\theta \setminus \cyl{R_{k + 1}}}}.\]
    \end{lem}
    \begin{proof}
        For each $\theta \in S_{k + 1}$, the definition of $\tau_\theta$ implies that $\cyl{\tau_\theta} \subseteq V \cap \cyl\theta$. We also know by \cref{Rk+1_in_V} that $\cyl{\tau_\theta}$ and $\cyl{R_{k + 1}}$ are disjoint. Hence
        \[\cyl{\tau_\theta} \subseteq V \cap \paren{\cyl\theta \setminus \cyl{R_{k + 1}}}.\]
        Hence by inequality \ref{bound_sums_of_remaining_measures_after_sk+1}, we have
        \begin{align*}
            \sum_{j = 1}^{k + 1} \lambda\paren{U_{n_j} \setminus U_{n_j, s_{k + 1}}} &< \min_{\theta \in S_{k + 1}} \lambda\cyl{\tau_\theta} \leq \min_{\theta \in S_{k + 1}} \lambda\paren{V \cap \paren{\cyl\theta \setminus \cyl{R_{k + 1}}}}. \qedhere
        \end{align*}
    \end{proof}

    \proofstep{Step 7: Proving facts about initial segments of $t(\sigma)$}

    Given any string $\tau$ with $|\tau| \leq \l_{k + 1}$, let
    \[D_{k + 1, \tau} = \set{\sigma \in D_{k + 1}: t(\sigma) \supseteq \tau}.\]
    Also, given any string $\rho$ of length $k$, let
    \[E_\rho = \set{\sigma \in 2^{\l_k}: t(\sigma) \supseteq \rho}.\]
    We now prove a few lemmas about these sets.
    \begin{lem}\label{how_much_was_mapped_into_theta}
        Consider any string $\theta$ of length $k + 1$. Then $\lambda\cyl{D_{k + 1, \theta}} \leq \lambda\cyl\theta$, and equality holds if and only if $\theta \notin S_{k + 1}$.
    \end{lem}
    \begin{proof}
        Since $t$ is injective on $D_{k + 1}$ by \cref{t_injective} and preserves the lengths of strings by \cref{Rk+1_in_V}, we have
        \[\lambda\cyl{D_{k + 1, \theta}} = \cyl{t\paren{D_{k + 1, \theta}}}.\]
        The result then follows according to each case:
        \begin{itemize}
            \item $\theta \in S_{k + 1}$. Then by \cref{Rk+1_in_V}, $\cyl{R_{k + 1}}$ is disjoint from $\cyl{\tau_\theta}$, so we have
            \[\cyl{t\paren{D_{k + 1, \theta}}} \subseteq \cyl\theta \setminus \cyl{\tau_\theta}\quad\text{and hence}\quad \lambda\cyl{t\paren{D_{k + 1, \theta}}} < \lambda\cyl\theta.\]
            \item $\theta \notin S_{k + 1}$. Then $\cyl\theta \subseteq \cyl{R_{k + 1}}$, which means that every length-$\l_{k + 1}$ extension of $\theta$ has the form $t(\sigma)$ for some $\sigma \in D_{k + 1}$. This gives us
            \begin{align*}
                \cyl{t\paren{D_{k + 1, \theta}}} &= \cyl\theta\quad\text{and hence}\quad \lambda\cyl{t\paren{D_{k + 1, \theta}}} = \lambda\cyl\theta. \qedhere
            \end{align*}
        \end{itemize}
    \end{proof}

    \begin{lem}\label{Dk+1_rho_in_E_rho}
        Let $\rho$ be a string of length $k$. Then $\cyl{D_{k + 1, \rho}} \subseteq \cyl{E_\rho}$.
    \end{lem}
    \begin{proof}
        Any string in $D_{k + 1, \rho}$ has the form $t^i(\sigma)$ for some $i < j \leq k + 1$ and some string $\sigma$ of length $\l_{k + 1}$ such that $\cyl\sigma \subseteq U_{n_j, s_{k + 1}}$ and $t^{i + 1}(\sigma) \supseteq \rho$. Let $\eta = t^i(\sigma) \prefix \l_k$.
        
        By \cref{t_monotone}, we know that $t(\eta) \subseteq t^{i + 1}(\sigma)$. By conditions \ref{mapping_onto_Rk} and \ref{mapping_onto_Sk}, the string $t(\eta)$ has length either $\l_k$ (if $\eta \in D_k$) or $k$ (if $\eta \notin D_k$). Either way, it must extend the length-$k$ initial segment $\rho$ of $t^{i + 1}(\sigma)$. This implies that $\eta \in E_\rho$ and hence
        \[\cyl{t^i(\sigma)} \subseteq \cyl\eta \subseteq \cyl{E_\rho}.\]
        Thus, $\cyl{D_{k + 1, \rho}} \subseteq \cyl{E_\rho}$.
    \end{proof}

    \begin{lem}\label{Dk+1_and_E_rho_minus_Dk+1_rho_partition}
        We have
        \[\cyl{D_{k + 1}} \cup \bigcup_{\rho \in S_k} \paren{\cyl{E_\rho} \setminus \cyl{D_{k + 1, \rho}}} = 2^\N.\]
    \end{lem}
    \begin{proof}
        Consider any string $\sigma$ of length $\l_{k + 1}$ such that $\sigma \notin D_{k + 1}$. Since $\cyl{D_k} \subseteq \cyl{D_{k + 1}}$ by \cref{Dk_in_Dk+1}, it follows that $\sigma \prefix \l_k \notin D_k$. Then condition \ref{mapping_onto_Sk} implies that $\rho := t\paren{\sigma \prefix \l_k} \in S_k$.
        
        Therefore, we have $\sigma \prefix \l_k \in E_\rho$ and $\sigma \notin D_{k + 1, \rho}$, so
        \begin{align*}
            \cyl\sigma &\subseteq \cyl{E_\rho} \setminus \cyl{D_{k + 1, \rho}}. \qedhere
        \end{align*}
    \end{proof}

    \begin{lem}\label{partition_of_E_rho}
        Let $\rho$ be a string of length $k$. Then
        \[\lambda\cyl{D_{k + 1, \rho}} + \sum_{\substack{\theta \in S_{k + 1} \\ \rho \subset \theta}} \paren{\lambda\cyl\theta - \lambda\cyl{D_{k + 1, \theta}}} = \lambda\cyl{E_\rho}.\]
    \end{lem}
    \begin{proof}
        Clearly,
        \[D_{k + 1, \rho} = D_{k + 1, \rho 0} \cup D_{k + 1, \rho 1},\]
        so
        \[\lambda\cyl{D_{k + 1, \rho}} = \lambda\cyl{D_{k + 1, \rho 0}} + \lambda\cyl{D_{k + 1, \rho 1}}.\]
        This can be rearranged to obtain
        \[\lambda\cyl{D_{k + 1, \rho}} + \sum_{i = 0}^1 \paren{\lambda\cyl{\rho i} - \lambda\cyl{D_{k + 1, \rho i}}} = \lambda\cyl\rho.\]
        By \cref{how_much_was_mapped_into_theta}, $\lambda\cyl{D_{k + 1, \rho i}} = \lambda\cyl{\rho i}$ if and only if $\rho i \notin S_{k + 1}$. Thus, the corresponding summand above is only nonzero when $\rho i \in S_{k + 1}$. We also know by condition \ref{measure_preserving} that $\lambda\cyl\rho = \lambda\cyl{E_\rho}$. This gives us the desired equality.
    \end{proof}

    \proofstep{Step 8: Defining $t$ everywhere else}

    Consider any string $\rho \in S_k$. By \cref{Dk+1_rho_in_E_rho,partition_of_E_rho}, we have
    \[\lambda\paren{\cyl{E_\rho} \setminus \cyl{D_{k + 1, \rho}}} = \sum_{\substack{\theta \in S_{k + 1} \\ \rho \subset \theta}} \paren{\lambda\cyl\theta - \lambda\cyl{D_{k + 1, \theta}}}.\]
    Thus, we can find sets $H_\theta \subseteq 2^{\l_{k + 1}}$ for each $\theta \in S_{k + 1}$ with $\rho \subset \theta$ such that
    \[\lambda\cyl{H_\theta} = \lambda\cyl\theta - \lambda\cyl{D_{k + 1, \theta}}\quad\text{and}\quad\cyl{E_\rho} \setminus \cyl{D_{k + 1, \rho}} = \bigcup_{\substack{\theta \in S_{k + 1} \\ \rho \subset \theta}} \cyl{H_\theta}.\]
    \begin{lem}\label{H_theta_partition_not_Dk+1}
        The sets $H_\theta$ for different $\theta \in S_{k + 1}$ form a partition of $2^{\l_{k + 1}} \setminus D_{k + 1}$.
    \end{lem}
    \begin{proof}
        There are three things to prove:
        \begin{itemize}
            \item $H_\theta$ is disjoint from $D_{k + 1}$ for each $\theta \in S_{k + 1}$.

            This holds because any string $\sigma \in H_\theta$ satisfies $t(\sigma) \supseteq \rho$ yet $\sigma \notin D_{k + 1, \rho}$, where $\rho = \theta \prefix k$.

            \item Every string of length $\l_{k + 1}$ is either in $D_{k + 1}$ or $H_\theta$ for some $\theta \in S_{k + 1}$.

            This holds by \cref{Dk+1_and_E_rho_minus_Dk+1_rho_partition} since
            \[\cyl{D_{k + 1}} \cup \bigcup_{\rho \in S_k}\bigcup_{\substack{\theta \in S_{k + 1} \\ \rho \subset \theta}} \cyl{H_\theta} = \cyl{D_{k + 1}} \cup \bigcup_{\rho \in S_k} \paren{\cyl{E_\rho} \setminus \cyl{D_{k + 1, \rho}}} = 2^\N.\]
            \item The sets $H_\theta$ for different strings $\theta \in S_{k + 1}$ are disjoint.

            Consider any $\theta, \theta' \in S_{k + 1}$ for which there exists $\sigma \in H_\theta \cap H_{\theta'}$. Let $\rho = \theta \prefix k$ and $\rho' = \theta' \prefix k$. Then $\cyl\sigma \subseteq \cyl{E_\rho} \cap \cyl{E_{\rho'}}$, so we must have $\rho = \rho'$.

            Now observe that if $\theta$ and $\theta'$ were different, then the clopen sets $\cyl{H_\theta}$ and $\cyl{H_{\theta'}}$ would be disjoint as they form a partition of $\cyl{E_\rho} \setminus \cyl{D_{k + 1, \rho}}$ due to their measures. Hence we must have $\theta = \theta'$.
        \end{itemize}
    \end{proof}
    By \cref{H_theta_partition_not_Dk+1}, we are now able to define $t$ on $2^{\l_{k + 1}} \setminus D_{k + 1}$ by $t(\sigma) = \theta$ for $\sigma \in H_\theta$.

    \proofstep{Step 9: Proving conditions \ref{mapping_onto_Sk}, \ref{measure_preserving}, and \ref{monotonicity} for $k + 1$}

    The following lemma gives us condition \ref{mapping_onto_Sk} for $k + 1$.

    \begin{lem}
        $t$ maps $2^{\l_{n + 1}} \setminus D_{k + 1}$ onto $S_{k + 1}$.
    \end{lem}
    \begin{proof}
        This is direct from the definition of $t$ on $2^{\l_{n + 1}} \setminus D_{k + 1}$, along with the fact that each $H_\theta$ for $\theta \in S_{k + 1}$ is nonempty by \cref{how_much_was_mapped_into_theta}.
    \end{proof}

    The following lemma gives us condition \ref{measure_preserving} for $k + 1$.

    \begin{lem}
        $\lambda\cyl{\set{\sigma \in 2^{\l_{k + 1}}: t(\sigma) \supseteq \theta}} = \lambda\cyl\theta$ whenever $|\theta| = k + 1$.
    \end{lem}
    \begin{proof}
        Consider any string $\theta$ of length $k + 1$. Then
        \[\set{\sigma \in 2^{\l_{k + 1}}: t(\sigma) \supseteq \theta} = \begin{cases}
            D_{k + 1, \theta} \cup H_\theta & \text{if } \theta \in S_{k + 1} \\
            D_{k + 1, \theta} & \text{if } \theta \notin S_{k + 1}
        \end{cases}.\]
        Note that whenever $\theta \in S_{k + 1}$, the definition of $H_\theta$ ensures that
        \[\lambda\cyl{D_{k + 1, \theta}} + \lambda\cyl{H_\theta} = \lambda\cyl\theta.\]
        Also, whenever $\theta \notin S_{k + 1}$, \cref{how_much_was_mapped_into_theta} tells us that
        \[\lambda\cyl{D_{k + 1, \theta}} = \lambda\cyl\theta.\]
        This gives the result.
    \end{proof}

    The following lemma gives us condition \ref{monotonicity} for $k + 1$.

    \begin{lem}
        For all strings $\sigma$ of length $\l_{k + 1}$, we have $t\paren{\sigma \prefix \l_k} \subseteq t(\sigma)$.
    \end{lem}
    \begin{proof}
        \cref{t_monotone} already showed this when $\sigma \in D_{k + 1}$, so assume $\sigma \in 2^{\l_{k + 1}} \setminus D_{k + 1}$. Then taking $\theta$ such that $\sigma \in H_\theta$ and letting $\rho = \theta \prefix k$, we have
        \[\cyl\sigma \subseteq \cyl{H_\theta} \subseteq \cyl{E_\rho}.\]
        This implies that $\sigma \prefix \l_k \in E_\rho$ and hence $t\paren{\sigma \prefix \l_k} \supseteq \rho$. In fact, given that $\sigma \notin D_{k + 1}$, since $\cyl{D_k} \subseteq \cyl{D_{k + 1}}$ by \cref{Dk_in_Dk+1}, it follows that $\sigma \prefix \l_k \notin D_k$. Thus, $t\paren{\sigma \prefix \l_k}$ has length exactly $k$ by condition \ref{mapping_onto_Sk}, so we get
        \begin{align*}
            t\paren{\sigma \prefix \l_k} &= \rho \subset \theta = t(\sigma). \qedhere
        \end{align*}
    \end{proof}
    With all conditions of the induction now verified for $k + 1$, this at last completes the definition of the sequences $\paren{n_k}_k$, $\paren{s_k}_k$, and $\paren{\l_k}_k$ in $\N$, the sequences $\paren{D_k}_k$, $\paren{R_k}_k$, and $\paren{S_k}_k$ in $\PP_{<\N}\paren{2^{<\N}}$, and the map $t: \bigcup_k 2^{\l_k} \to 2^{<\N}$. Note that all of these are computable.

    \proofstep{Step 10: Showing that the limit $T$ satisfies all the desired properties}

    Since the map $t$ is monotone by condition \ref{monotonicity}, it has a well-defined limit $T: \ \subseteq \! 2^\N \to 2^\N$. Let $A = \neg V$. We now check each of the properties we claimed in the statement of the theorem:
    \begin{itemize}
        \item $T$ is total.

        This holds by conditions \ref{mapping_onto_Rk} and \ref{mapping_onto_Sk}, which imply that $|t(\sigma)| \geq k$ whenever $|\sigma| = \l_k$.
        
        \item $T$ is computable.

        This holds by \cref{limit_of_computable_monotone_map} since $t$ is computable.

        \item $T$ is measure-preserving.

        Observe that for any string $\rho$ of length $k$ and any $y \in 2^\N$, we have
        \[y \in T^{-1}\cyl\rho \iff \rho \subset T(y) \iff \rho \subset t\paren{y \prefix \l_k}.\]
        Hence condition \ref{measure_preserving} implies
        \[\lambda\paren{T^{-1}\cyl\rho} = \lambda\cyl{\set{\sigma \in 2^{\l_k}: t(\sigma) \supseteq \rho}} = \lambda\cyl\rho.\]
        As the cylinders $\cyl\rho$ generate the Borel $\sigma$-algebra on $2^\N$, it follows that $T$ is measure-preserving.

        \item $A \ni x$ is a $\Pi^0_1$ set with computable $\lambda(A) > 0$.

        Since $V \not\ni x$ is $\Sigma^0_1$, clearly $A \ni x$ is $\Pi^0_1$. Also, we assumed at the beginning of the proof that $\lambda(V) < 1$ was computable, which implies that $\lambda(A) > 0$ is likewise computable.
        
        \item $T^j(x) \notin A$ for all $j > 0$.

        Consider any $j > 0$ and take $k$ large enough such that $x \in U_{n_j, s_k}$. Let $\sigma = x \prefix \l_k$. Then $\cyl\sigma \subseteq U_{n_j, s_k}$, so $\tau := t^{j - 1}(\sigma) \in D_k$ by condition \ref{determined_domain}. Hence conditions \ref{determined_range} and \ref{mapping_onto_Rk} imply
        \[T^j(x) \in \cyl{t^j(\sigma)} = \cyl{t(\tau)} \subseteq \cyl{R_k} \subseteq V = \neg A.\]
    \end{itemize}
\end{proof}

\subsection{Discussion}

With the proof of \cref{Schnorr_converse} finally complete, a couple of questions arise.

First, recall from \cref{CPS_weak_pi0n_generic_Poincare_ergodic} that weakly 1-generic points recur in any $\Pi^0_1$ set with positive computable measure containing them under a computable \emph{ergodic} measure-preserving transformation. Thus, it is natural to ask whether we can modify the construction above to obtain an ergodic transformation under the same hypotheses. This would characterize when a point is either Schnorr random or weakly 1-generic. 

\begin{qn}
    Given a point $x \in 2^\N$ which is neither Schnorr random nor weakly 1-generic, do there exist a $\Pi^0_1$ set $A \ni x$ with computable $\lambda(A) > 0$ and a computable \emph{ergodic} measure-preserving transformation $T: 2^\N \to 2^\N$ such that $T^k(x) \notin A$ for all $k > 0$?
\end{qn}

Similarly, \cref{CPS_pi0n_generic_Poincare} said that 1-generic points recur in any $\Pi^0_1$ set with positive computable measure containing them under a computable measure-preserving transformation. Thus, we can likewise ask whether a suitable transformation $T$ could still be constructed if we replaced ``weakly 1-generic'' with ``1-generic.'' This would mean assuming that $x$ lies on the boundary of some (not necessarily dense) $\Sigma^0_1$ set $V$. 

Without density of $V$, the strategy used in the proof above breaks down completely, and there is little hope of replacing it with anything else. Indeed, if we had a computable transformation $T$ that worked, then there would exist an algorithm that can list all the cylinders mapping into any given cylinder $\cyl\theta$ under $T$. However, if $x$ is weakly 1-generic but not 1-generic, then there is no algorithmic way to determine whether a given cylinder $\cyl\theta$ will intersect $V$ until we see something from $V$ enumerated in it.

The only safe way to handle this seems to be choosing portions of $\cyl\theta$ itself to map into $\cyl\theta$, for if $x$ happens to be in $\cyl\theta$, then because $x$ lies on the boundary of $V$, we will eventually see some of $V$ enumerated in $\cyl\theta$. Depending on the timing of when we find cylinders of the Schnorr test, though, we may not be able to map the entirety of them into $V$ due to the commitments we have to make eventually about where everything is mapped.

Thus, an affirmative answer to the following question is sure to be vastly more difficult than the proof above, assuming there even exists one.

\begin{qn}
    Given a point $x \in 2^\N$ which is neither Schnorr random nor \emph{1-generic}, do there exist a $\Pi^0_1$ set $A \ni x$ with computable $\lambda(A) > 0$ and a computable measure-preserving transformation $T: 2^\N \to 2^\N$ such that $T^k(x) \notin A$ for all $k > 0$?
\end{qn}

\section[Effective recurrence in Π0n sets]{Effective recurrence in $\Pi^0_n$ sets}\label{section_other}

\subsection{Recurrence for $n$-random points}

Recall from \cref{CPS_weak_n_Poincare_ergodic} that weakly $(n + 1)$-random points recur in any $\Sigma^0_{n + 1}$ set (or equivalently, any $\Pi^0_n$ set) with positive measure under a computable ergodic measure-preserving transformation. The following result due to \cite{Bienvenu2012}, as well as \cite{Franklin2012} independently, shows that the same is in fact true of $n$-random points.

\begin{thm}\label{CPS_n_Poincare_ergodic}
    Let $(X, \mu)$ be a computable probability space and $n \geq 1$. If $x \in X$ is $n$-random, then for any $\Pi^0_n$ set $A$ with $\mu(A) > 0$ and any computable ergodic measure-preserving transformation $T: \ \subseteq \! X \to X$ with $x \in \dom T$, we have $T^k(x) \in A$ for some $k > 0$.
\end{thm}
\begin{proof}
    Theorem 12 in \cite{Bienvenu2012} gives a slightly more general version of this for the case $n = 1$. For $n > 1$, note that if $A$ is $\Pi^0_n$ with $\mu(A) > 0$, then by \cref{approximate_sigma0n_with_sigma01(zero(n-1))}, there is a $\Pi^0_1\paren{\zero^{(n - 1)}}$ set $B \subseteq A$ with $\mu(B) > 0$. Relativizing the result for $n = 1$ then shows that for any $n$-random point $x \in X$ and any computable ergodic measure-preserving transformation $T: \ \subseteq \! X \to X$ with $x \in \dom T$, we will have $T^k(x) \in B \subseteq A$ for some $k > 0$.
\end{proof}

A similar result is expected to be true for non-ergodic transformations, but this remains open.

\begin{conj}
    Let $(X, \mu)$ be a computable probability space and $n \geq 1$. If $x \in X$ is $n$-random, then for any $\Pi^0_n$ set $A \ni x$ with $\mu(A) > 0$ and any computable measure-preserving transformation $T: \ \subseteq \! X \to X$ with $x \in \dom T$, we have $T^k(x) \in A$ for some $k > 0$.
\end{conj}

Note that if $n$-randomness is replaced with weak $(n + 1)$-randomness here, then the result is true by \cref{CPS_weak_n_Poincare}. (In fact, weakly $(n + 1)$-random points possess a stronger recurrence property that will be the content of \cref{CPS_weak_recurrence_frequency}.) In the non-ergodic setting, there are no weaker notions of randomness known to guarantee recurrence in $\Pi^0_n$ sets.

By the following result, $n$-randomness is indeed optimal in \cref{CPS_n_Poincare_ergodic}.

\begin{thm}\label{n_Poincare_ergodic_characterization}
    Consider Cantor space with the Lebesgue measure $\lambda$, and let $n \geq 1$. Then $x \in 2^\N$ is $n$-random if and only if for any $\Pi^0_n$ set $A$ with $\mu(A) > 0$ and any computable ergodic measure-preserving transformation $T: \ \subseteq \! 2^\N \to 2^\N$ with $x \in \dom T$, we have $T^k(x) \in A$ for some $k > 0$.
\end{thm}
\begin{proof}
    The $\RA$ direction is simply \cref{CPS_n_Poincare_ergodic}. For the $\LA$ direction, assume that $x \in 2^\N$ is not $n$-random. Let $\paren{U_j}_j$ be a universal $n$-test and $T: 2^\N \to 2^\N$ be the shift map.

    By \cref{shift_preserves_non_randomness}, $T^k(x)$ is not $n$-random for any $k$, which implies that $T^k(x) \in \bigcap_j U_k$ for all $k$. If we take $A = \neg U_1$, which has measure $\lambda(A) = 1 - \lambda(U_1) \geq 1/2$, then we get $T^k(x) \notin A$ for all $k$.
\end{proof}

Similar to when we characterized Schnorr $n$-randomness in terms of recurrence in $\Pi^0_n$ sets with $\zero^{(n - 1)}$-computable measure, the set $A$ constructed in the proof above did \emph{not} contain the point $x$. We have already seen in \cref{CPS_pi0n_generic_Poincare} that there exist points that recur in any $\Pi^0_n$ set containing them and are not Schnorr $n$-random, namely the $\Pi^0_n$-generic points. In the following subsection, we will further expand this to a broader class of points, which we will call quasi-$\Pi^0_n$-generic.

\subsection[Recurrence for quasi-Π0n-generic points]{Recurrence for quasi-$\Pi^0_n$-generic points}

In \cref[S]{section_Schnorr}, we saw that no matter the computable probability space $(X, \mu)$, for each $n \geq 2$, there always exist points $x \in X$ such that for every $\Pi^0_n$ set $P \ni x$, there is a $\Pi^0_1\paren{\zero^{(n - 2)}}$ set $Q \subseteq P$ with $\mu(Q) > 0$ and $x \in Q$. We called these points $\Pi^0_n$-generic, and because they are always weakly $n$-random, their defining property guarantees that they recur in any $\Pi^0_n$ set containing them with positive measure under a computable measure-preserving transformation. This expanded the class of points which recur in any $\Pi^0_n$ set containing them with positive $\zero^{(n - 1)}$-computable measure beyond the Schnorr $n$-random points.

We now take this a step further by introducing a broader notion of genericity that will allow us to obtain a similar recurrence result in greater generality.

\subsubsection[Defining quasi-Π0n-genericity]{Defining (weak) quasi-$\Pi^0_n$-genericity}

We define this broader notion of genericity as follows.

\begin{defn}\label{quasi_pi0n_generic_definition}
    Let $(X, \mu)$ be a computable probability space, $n \geq 1$, and $A \subseteq \N$. (In each of the following definitions, we drop ``relative to $A$'' if $A = \zero$.)
    \begin{itemize}
        \item A point $x \in X$ is \textbf{quasi-$\Pi^0_n$-generic relative to $A$} if for any $\Pi^0_n(A)$ set $P \ni x$, there is a $\Pi^0_1\paren{A^{(n - 1)}}$ set $Q \subseteq P$ with $A^{(n - 1)}$-computable $\mu(Q) > 0$ and $x \in Q$.

        \item A point $x \in X$ is \textbf{weakly quasi-$\Pi^0_n$-generic relative to $A$} if for any $\Pi^0_n(A)$ set $P \ni x$, there is a $\Pi^0_1\paren{A^{(n - 1)}}$ set $Q \subseteq P$ with $A^{(n - 1)}$-computable $\mu(Q) > 0$.
    \end{itemize}
\end{defn}

When considering Cantor space with the Lebesgue measure, we could also say ``(weakly) quasi-1-generic'' in place of ``(weakly) quasi-$\Pi^0_1$-generic,'' which is more in line with the standard terminology and is slightly more concise.

\begin{lem}
    Let $(X, \mu)$ be a computable probability space, $n \geq 1$, and $A \subseteq \N$. Then relativizing all terms to the oracle $A$,
    \begin{center}
		$\Pi^0_n$-generic $\myimplies{(a)}$ Quasi-$\Pi^0_n$-generic \qquad\qquad\qquad\qquad\qquad\qquad \ \\
		\qquad $\downimplies$ \qquad\qquad\qquad\qquad $\downimplies$ \qquad\qquad\qquad\qquad\qquad\qquad\qquad\qquad \ \\
		Weakly $\Pi^0_n$-generic $\myimplies{(b)}$ Weakly quasi-$\Pi^0_n$-generic $\myimplies{(c)}$ Weakly $n$-random.
	\end{center}
\end{lem}
\begin{proof}
    The vertical implications are obvious from the definition since we merely remove the condition ``$x \in Q$'' when moving from (quasi-)$\Pi^0_n$-genericity to weak (quasi-)$\Pi^0_n$-genericity.
    
    For implications (a) and (b), there are two cases:
    \begin{itemize}
        \item $n = 1$. Assume that $x$ is $\Pi^0_1$-generic relative to $A$ and consider any $\Pi^0_1(A)$ set $P \ni x$. Then there is a $\Sigma^0_1$ set $U \subseteq P$ with $\mu(U) > 0$ and $x \in U$. By \cref{approximate_sigma01_or_pi01}, we can write $U = \bigcup_s P_s$ for some increasing sequence $\paren{P_s}_s$ of uniformly $\Pi^0_1$ sets with uniformly computable measures. Then there must be some $P_s$ (which is inside $P$) such that $\mu\paren{P_s} > 0$ and $x \in P_s$. Thus, $x$ is quasi-$\Pi^0_n$-generic relative to $A$, giving implication (a).

        A nearly identical proof shows that any weakly $\Pi^0_1$-generic point relative to $A$ is also weakly quasi-$\Pi^0_1$-generic relative to $A$, giving implication (b).

        \item $n \geq 2$. Then implications (a) and (b) hold because any $\Pi^0_1\paren{A^{(n - 2)}}$ set is also $\Pi^0_1\paren{A^{(n - 1)}}$ with $A^{(n - 2)}$-upper-semicomputable measure and hence $A^{(n - 1)}$-computable measure.
    \end{itemize}

    Implication (c) holds because no weakly quasi-$\Pi^0_n$ generic point relative to $A$ can lie in a null $\Pi^0_n(A)$ set.
\end{proof}

\subsubsection[Existence and randomness of quasi-Π0n-genericity]{Existence and randomness of quasi-$\Pi^0_n$-genericity}

Since any (weakly) $\Pi^0_n$-generic point is (weakly) quasi-$\Pi^0_n$-generic, we have already established conditions under which these points exist in \cref{pi01_generic_comeager,pi0n_generic_exists}. In fact, we will be able to say something even stronger through the following two lemmas.

\begin{lem}\label{quasi_pi0n_generic_dense}
    Let $(X, \mu)$ be a computable probability space, $n \geq 1$, and $A \subseteq \N$. Let $(\P, \subseteq)$ be the partial order consisting of $\Pi^0_1\paren{A^{(n - 1)}}$ sets with positive $A^{(n - 1)}$-computable measure, ordered by the subset relation. For each $\Pi^0_n(A)$ set $P$, let
    \[D_P = \{Q \in \P: Q \subseteq P \text{ or } P \cap Q = \zero\}.\]
    Then $D_P$ is dense in $\P$ for each $\Pi^0_n(A)$ set $P$.
\end{lem}
\begin{proof}
    Consider any $\Pi^0_n(A)$ set $P$ and any $Q \in \P$. We seek to show that there is some $R \in D_P$ such that $R \subseteq Q$. There are two cases:
    \begin{itemize}
        \item $\mu(Q \setminus P) > 0$. We have two subcases:
        \begin{itemize}
            \item $n = 1$. Since $\neg P$ is $\Sigma^0_1(A)$, by \cref{approximate_sigma01_or_pi01}, we can write $\neg P = \bigcup_j P_j$, where the sets $P_j$ are uniformly $\Pi^0_1(A)$ with uniformly $A$-computable measures. Then
            \[\mu\paren{\bigcup_j \paren{Q \cap P_j}} = \mu(Q \cap \neg P) > 0,\]
            so there exists some $j$ such that $R := Q \cap P_j$ has positive measure. Since $Q$ and $P_j$ are both $\Pi^0_1(A)$ with $A$-computable measure, so is $R$. Then $R \in \P$ and $R \subseteq P_j \subseteq \neg P$, so $R \in D_P$. Also, $R \subseteq Q$.

            \item $n \geq 2$. Since $\neg P$ is $\Sigma^0_n(A)$, we can write $\neg P = \bigcup_j P_j$, where the sets $P_j$ are uniformly $\Pi^0_{n - 1}(A)$. Then
            \[\mu\paren{\bigcup_j \paren{Q \cap P_j}} = \mu(Q \cap \neg P) > 0,\]
            so there exists some $j$ such that $\mu\paren{Q \cap P_j} > 0$. Fix $q \in \Q$ such that $0 < q < \mu\paren{Q \cap P_j}$. Since $P_j$ is $\Pi^0_{n - 1}(A)$, by \cref{approximate_sigma0n_with_sigma01(zero(n-1))} relativized to $A$, there exists a $\Pi^0_1\paren{A^{(n - 2)}}$ set $Q_j \subseteq P_j$ such that $\mu\paren{Q_j} \geq \mu\paren{P_j} - q$. Then given that
            \[q < \mu\paren{Q \cap P_j} \leq \mu\paren{Q \cap Q_j} + \mu\paren{P_j \setminus Q_j} \leq \mu\paren{Q \cap Q_j} + q,\]
            we must have $\mu\paren{Q \cap Q_j} > 0$. Since $Q$ and $Q_j$ are both $\Pi^0_1\paren{A^{(n - 1)}}$ with $A^{(n - 1)}$-computable measures, so is $R := Q \cap Q_j$. Then $R \in \P$ and $R \subseteq Q_j \subseteq P_j \subseteq \neg P$, so $R \in D_P$. Also, $R \subseteq Q$.
        \end{itemize}

        \item $\mu(Q \setminus P) = 0$. We have two subcases:
        \begin{itemize}
            \item $n = 1$. Then $R := Q \cap P$ is $\Pi^0_1(A)$, and $\mu(R) = \mu(Q)$ is positive and computable. Thus, $R \in \P$ and $R \subseteq P$, so $R \in D_P$. Also, $R \subseteq Q$.

            \item $n \geq 2$. We have $\mu\paren{Q \cap P_j} = 0$ for all $j$. Fix $q \in \Q$ with $0 < q < \mu(Q)$.
            
            Since $Q$ is $\Pi^0_1\paren{A^{(n - 1)}}$, by \cref{approximate_sigma01_or_pi01}, we can write $Q = \bigcap_k U_k$, where the sets $U_k$ are uniformly $\Sigma^0_1\paren{A^{(n - 1)}}$ with uniformly $A^{(n - 1)}$-computable measures. Likewise, since the sets $P_j$ are uniformly $\Pi^0_{n - 1}(A)$, we can write $P_j = \bigcap_k V_{j, k}$, where the sets $V_{j, k}$ are uniformly $\Sigma^0_{n - 2}(A)$ (with uniformly $A^{(n - 2)}$-computable measures). Then for each $j$,
            \[\lim_{k \to \infty} \mu\paren{U_k \cap V_{j, k}} = \mu\paren{Q \cap P_j} = 0,\]
            so there is an $A^{(n - 2)}$-computable function $s$ such that $\mu\paren{U_{s(j)} \cap V_{j, s(j)}} \leq 2^{-(j + 2)}q$ for all $j$.
            
            By \cref{approximate_sigma0n_with_sigma01(zero(n-1))}, take uniformly $\Sigma^0_1\paren{A^{(n - 3)}}$ sets $U_{j, k} \supseteq V_{j, k}$ such that $\mu\paren{U_{j, k}} \leq \mu\paren{V_{j, k}} + 2^{-(j + 2)}q$ for all $j$ and $k$. Then since
            \[U_k \cap U_{j, k} \subseteq \paren{U_k \cap V_{j, k}} \cup \paren{U_{j, k} \setminus V_{j, k}},\]
            we have
            \begin{align*}
                \mu\paren{U_{s(j)} \cap U_{j, s(j)}} &\leq \mu\paren{U_{s(j)} \cap V_{j, s(j)}} + \mu\paren{U_{j, s(j)} \setminus V_{j, s(j)}} \\
                &\leq 2^{-(j + 2)}q + 2^{-(j + 2)}q = 2^{-(j + 1)}q.
            \end{align*}
            Now observe that, since $Q \subseteq U_k$ for all $k$,
            \[\mu\paren{Q \cap \bigcup_j U_{j, s(j)}} \leq \sum_j \mu\paren{Q \cap U_{j, s(j)}} \leq \sum_j \mu\paren{U_{s(j)} \cap U_{j, s(j)}} \leq \sum_j 2^{-(j + 1)}q = q.\]
            Then $R := Q \setminus \bigcup_j U_{j, s(j)}$ is $\Pi^0_1\paren{A^{(n - 1)}}$ and satisfies
            \[\mu(R) = \mu(Q) - \mu\paren{Q \cap \bigcup_j U_{j, s(j)}} \geq \mu(Q) - q > 0.\]
            Thus, $R \in \P$. Moreover,
            \[\neg P = \bigcup_j P_j \subseteq \bigcup_j V_{j, s(j)} \subseteq \bigcup_j U_{j, s(j)},\]
            so
            \[R \subseteq \neg \bigcup_j U_{j, s(j)} \subseteq P.\]
            This implies that $R \in D_P$. Also, $R \subseteq Q$.
        \end{itemize}
    \end{itemize}
    Thus, $D_P$ is dense in $\P$.
\end{proof}

\begin{lem}\label{Schnorr_dense}
    Let $(X, \mu)$ be a computable probability space, $n \geq 1$, and $A \subseteq \N$. Let $(\P, \subseteq)$ be the partial order consisting of $\Pi^0_1\paren{A^{(n - 1)}}$ sets with positive $A^{(n - 1)}$-computable measure, ordered by the subset relation. For each Schnorr test $\U = \paren{U_k}_k$ relative to $A^{(n - 1)}$, let
    \[E_\U = \set{Q \in \P: Q \cap \bigcap_k U_k = \zero}.\]
    Then $E_\U$ is dense in $\P$ for each Schnorr test $\U$ relative to $A^{(n - 1)}$.
\end{lem}
\begin{proof}
    Consider any Schnorr test $\U = \paren{U_k}_k$ relative to $A^{(n - 1)}$ and any $Q \in \P$. That is, the sets $U_k$ are uniformly $\Sigma^0_1\paren{A^{(n - 1)}}$ with uniformly $A^{(n - 1)}$-computable measures $\mu(U_k) \leq 2^{-k}$. We seek to show that there is some $R \in E_\U$ such that $R \subseteq Q$. Since
    \[\mu\paren{\bigcup_k \paren{Q \setminus U_k}} = \mu\paren{Q \setminus \bigcap_k U_k} = \mu(Q) > 0,\]
    there must be some $j$ such that $\mu\paren{Q \setminus U_j} > 0$. This measure is also $A^{(n - 1)}$-computable by \cref{intersection_of_sigma0n_with_computable_measure} since $Q$ and $\neg U_j$ are $\Pi^0_1\paren{A^{(n - 1)}}$ sets with $A^{(n - 1)}$-computable measures, so $R := Q \setminus U_j \in \P$. Clearly, $R$ is disjoint from $\bigcap_k U_k$, so $R \in E_\U$, and $R \subseteq Q$.
\end{proof}

Now it easily follows that, not only do quasi-$\Pi^0_n$-generic points exist, but they can also be Schnorr $n$-random.

\begin{cor}\label{quasi_pi0n_generic_and_Schnorr_n_random_exists}
    Let $(X, \mu)$ be a computable probability space, $n \geq 1$, and $A \subseteq \N$. Then there is a quasi-$\Pi^0_n$-generic point relative to $A$ which is also Schnorr $n$-random relative to $A$.
\end{cor}
\begin{proof}
    Define the partial order $(\P, \subseteq)$ and the sets $D_P$ and $E_\U$ as in \cref{quasi_pi0n_generic_dense,Schnorr_dense}. Let $P_e$ denote the $e\th$ $\Pi^0_n(A)$ set and $\U_e$ denote the $e\th$ Schnorr test relative to $A^{(n - 1)}$ (according to some non-effective enumeration of them). Also, by \cref{approximate_sigma01_inside_with_compact}, we may take a compact $\Pi^0_1$ set $K$ with positive measure.
    
    Since $K \in \P$ and the sets $D_P$ and $E_\U$ are dense in $\P$, we can recursively define a nested sequence $\paren{Q_j}_j$ such that $Q_0 \subseteq K$, $Q_{2e} \in D_{P_e}$ for all $e$, and $Q_{2e + 1} \in E_{\U_e}$ for all $e$. Since $K$ is compact and each set $Q_j \subseteq K$ is $\Pi^0_1\paren{A^{(n - 1)}}$ and hence closed, there exists a point $x \in \bigcap_j Q_j$.
    
    For each $\Pi^0_n(A)$ set $P_e \ni x$, we have $x \in Q_{2e}$, which means that $P_e$ and $Q_{2e}$ are not disjoint. Since $Q_{2e} \in D_{P_e}$, it follows that $Q_{2e} \subseteq P_e$. This set is $\Pi^0_1\paren{A^{(n - 1)}}$ with $A^{(n - 1)}$-computable $\mu(Q_e) > 0$ and $x \in Q_{2e}$, so $x$ is quasi-$\Pi^0_n$-generic point relative to $A$.

    Also, for each Schnorr test $\U_e = \paren{U_{e, k}}_k$ relative to $A^{(n - 1)}$, we have $Q_{2e + 1} \in E_{\U_e}$, which means that $Q_{2e + 1} \ni x$ is disjoint from $\bigcap_k U_{e, k}$. That is, $x \notin \bigcap_k U_{e, k}$, so $x$ is Schnorr $n$-random relative to $A$.
\end{proof}

When considering Cantor space with the Lebesgue measure, recall by \cref{pi01_generic_not_Schnorr_random,pi0n_generic_not_Schnorr_random} that no Schnorr $n$-random point relative to $A$ is ever weakly $\Pi^0_n$-generic relative to $A$. Thus, we can conclude from \cref{quasi_pi0n_generic_and_Schnorr_n_random_exists} that quasi-$\Pi^0_n$-genericity is strictly more general a notion than $\Pi^0_n$-genericity.

\begin{cor}
    Consider Cantor space with the Lebesgue measure. Then for all $n \geq 1$ and $A \subseteq \N$, there is a point which is quasi-$\Pi^0_n$-generic relative to $A$ but not weakly $\Pi^0_n$-generic relative to $A$.
\end{cor}

The level of non-randomness that can be guaranteed of quasi-$\Pi^0_n$-genericity is likewise slightly different from that of $\Pi^0_n$-genericity.

\begin{thm}\label{quasi_pi0n_generic_not_n_random}
    Consider Cantor space with the Lebesgue measure. Then for all $n \geq 1$ and $A \subseteq \N$, no weakly quasi-$\Pi^0_n$-generic point relative to $A$ is $n$-random relative to $A$.
\end{thm}
\begin{proof}
    Let $x \in 2^\N$ be $n$-random relative to $A$, i.e., Martin-L\"of random relative to $A^{(n - 1)}$. Taking $\paren{U_k}_k$ to be a universal Martin-L\"of test relative to $A^{(n - 1)}$, we must have $x \in \neg U_k =: P$ for some $k$, which is $\Pi^0_1\paren{A^{(n - 1)}}$. Any nonempty $\Pi^0_1(A^{(n - 1)})$ subset of $P$ contains only Martin-L\"of random points relative to $A^{(n - 1)}$ and consequently must have Martin-L\"of random measure relative to $A^{(n - 1)}$ by \cref{nonempty_pi01(zero(n-1))_subset_of_nR_has_nR_measure}. In particular, $P$ has no $\Pi^0_1\paren{A^{(n - 1)}}$ subset with positive $A^{(n - 1)}$-computable measure, so $x$ cannot be weakly quasi-$\Pi^0_n$-generic relative to $A$.
\end{proof}

As with $\Pi^0_n$-genericity when $n \geq 2$, there does not appear to be a simple relationship between (weakly) quasi-$\Pi^0_n$-generic points relative to different oracles. This is demonstrated by the following result, which mirrors \cref{pi0n_genericity_does_not_get_stronger_with_relativization}.

\begin{lem}\label{quasi-pi0n_genericity_does_not_get_stronger_with_relativization}
    Consider Cantor space with the Lebesgue measure and $n \geq 1$. Then whenever $A, B \subseteq \N$ satisfy $A \geq_T B'$, there is a quasi-$\Pi^0_n$-generic point relative to $A$ which is not weakly quasi-$\Pi^0_n$-generic relative to $B$.
\end{lem}
\begin{proof}
    By \cref{quasi_pi0n_generic_and_Schnorr_n_random_exists}, there exists a $\Pi^0_n$-generic point $x$ relative to $A$ which is also Schnorr $n$-random relative to $A$. Since $A \geq_T B'$, $x$ is Schnorr $(n + 1)$-random relative to $B$ and hence $n$-random relative to $B$. Then by \cref{quasi_pi0n_generic_not_n_random}, $x$ cannot be weakly quasi-$\Pi^0_n$-generic relative to $B$.
\end{proof}

\begin{figure}
    \centering
    \begin{tikzpicture}[line width=1.25pt, transform shape]
        % Weakly $n$-random
        \draw[black, fill=olive, fill opacity=0.1] (0,0) rectangle ++(16,11);
    
        % Schnorr $n$-random
        \draw[black, fill=cyan, fill opacity=0.2] (1,1) rectangle ++(7,5);
    
        % $n$-random
        \draw[black, fill=blue, fill opacity=0.2] (2,2) rectangle ++(3,2);
    
        % Weakly quasi-$\Pi^0_n$-generic
        \draw[black, fill=pink, fill opacity=0.4] (6,2) rectangle ++(9,8);
    
        % Quasi-$\Pi^0_n$-generic
        \draw[black, fill=violet, fill opacity=0.2] (7,3) rectangle ++(6,5);
    
        % Weakly $\Pi^0_n$-generic
        \draw[black, fill=yellow, fill opacity=0.3] (9,4) rectangle ++(5,5);
    
        % $\Pi^0_n$-generic
        \draw[black, fill=orange, fill opacity=0.4] (10,5) rectangle ++(2,2);
    
        \node at (3,8.5) {Weakly $n$-random};
    
        \node at (3.5,5) {Schnorr $n$-random};
    
        \node at (3.5,3) {$n$-random};
    
        \node at (11,6) {$\Pi^0_n$-generic};
    
        \node at (11.5,8.5) {Weakly $\Pi^0_n$-generic};
    
        \node at (10,3.5) {Quasi-$\Pi^0_n$-generic};
    
        \node at (10.5,2.5) {Weakly quasi-$\Pi^0_n$-generic};
    \end{tikzpicture}
    \caption{Relationships among genericity and randomness in $\paren{2^\N, \lambda}$}
    \label{genericity_picture}
\end{figure}
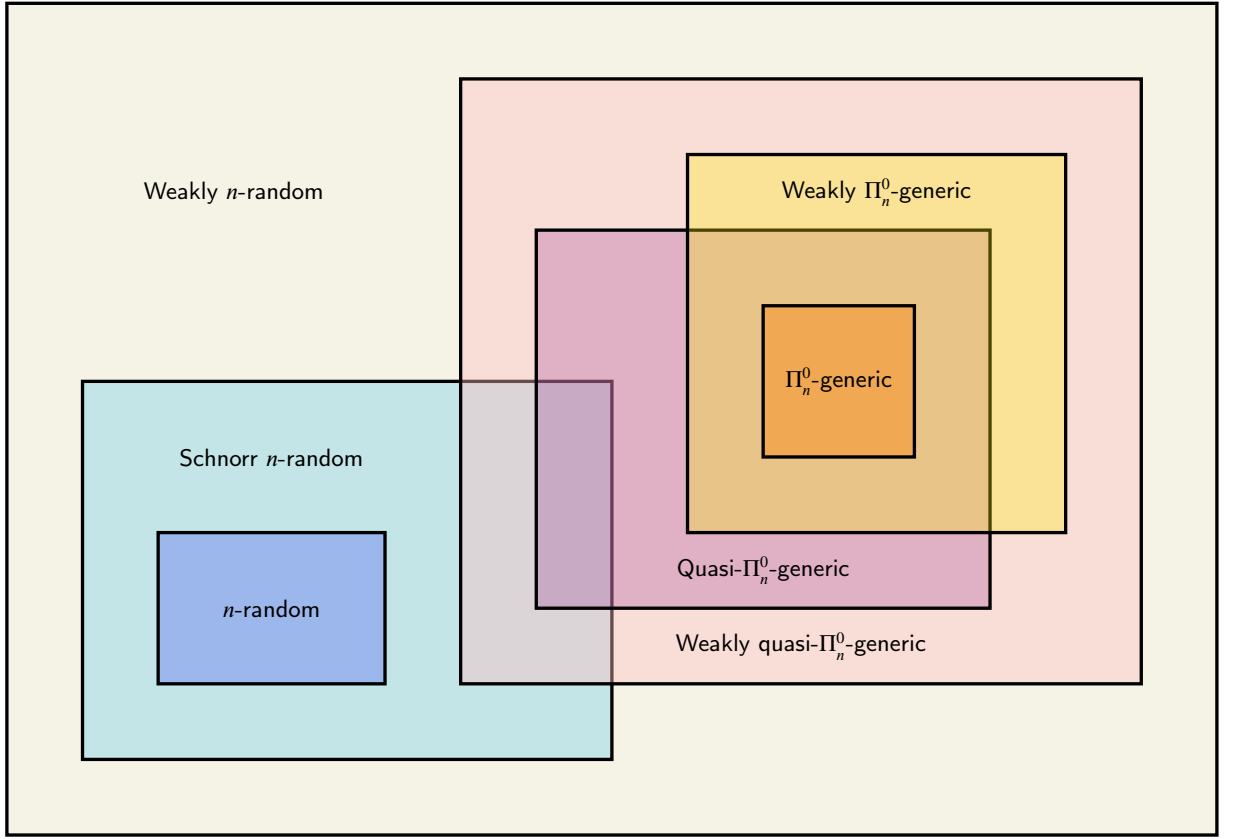

With all these new notions of genericity, we have only discussed a small portion of the possible relationships among them. Altogether, the picture will look something like \cref[S]{genericity_picture} when considering Cantor space with the Lebesgue measure, but we have not shown that every intersection it depicts is nonempty. Thus, we have the following question, which could of course be expanded in scope even further by relativizing to various oracles.

\begin{qn}
    Considering Cantor space with the Lebesgue measure, which intersections of (weak) $\Pi^0_n$-genericity, (weak) quasi-$\Pi^0_n$-genericity, weak $n$-randomness, Schnorr $n$-randomness, $n$-randomness, and their complements are nonempty? Here are some of the more interesting intersections to consider:
    \begin{itemize}
        \item Is there a point which is weakly $\Pi^0_n$-generic but not $\Pi^0_n$-generic when $n \geq 2$?
        \item Is there a point which is weakly quasi-$\Pi^0_n$-generic but not quasi-$\Pi^0_n$-generic?
        \item Is there a point which is weakly $\Pi^0_n$-generic but not quasi-$\Pi^0_n$-generic?
        \item Is there a point which is quasi-$\Pi^0_n$-generic but neither Schnorr $n$-random nor weakly $\Pi^0_n$-generic?
        \item Is there a point which is Schnorr $n$-random but neither $n$-random nor weakly quasi-$\Pi^0_n$-generic?
    \end{itemize}
\end{qn}

\subsubsection[Effective recurrence and preservation of quasi-Π0n-genericity]{Effective recurrence and preservation of quasi-$\Pi^0_n$-genericity}

We at last provide effective recurrence theorems for (weakly) quasi-$\Pi^0_n$-generic points which are also Schnorr $n$-random. Since these points always exist in $\paren{2^\N, \lambda}$ by \cref{quasi_pi0n_generic_and_Schnorr_n_random_exists} and are not $n$-random by \cref{quasi_pi0n_generic_not_n_random}, the following results properly expand the class of recurrent points for $\Pi^0_n$ sets established so far.

\begin{thm}\label{CPS_Schnorr_n_quasi_pi0n_generic_Poincare}
    Let $(X, \mu)$ be a computable probability space, $n \geq 1$, and $A \subseteq \N$. If $x \in X$ is both Schnorr $n$-random relative to $A$ and quasi-$\Pi^0_n$-generic relative to $A$, then for any $\Pi^0_n(A)$ set $P \ni x$ with $\mu(P) > 0$ and any computable measure-preserving transformation $T: \ \subseteq \! X \to X$ with $x \in \dom T$, we have $T^k(x) \in P$ for some $k > 0$.
\end{thm}
\begin{proof}
    Assume $x \in X$ is both Schnorr $n$-random relative to $A$ and $\Pi^0_n$-generic relative to $A$, and let $P \ni x$ be $\Pi^0_n(A)$ with $\mu(P) > 0$. Then there is a $\Pi^0_1(A^{(n - 1)})$ set $Q \subseteq P$ with $A^{(n - 1)}$-computable $\mu(Q) > 0$ and $x \in Q$. Since $x$ is Schnorr random relative to $A^{(n - 1)}$ and $Q \ni x$ is $\Pi^0_1\paren{A^{(n - 1)}}$ with $A^{(n - 1)}$-computable $\mu(Q) > 0$, \cref{CPS_Schnorr_n_Poincare} relativized to $A^{(n - 1)}$ implies that $T^k(x) \in Q \subseteq P$ for some $k > 0$.
\end{proof}

Using \cref{CPS_Schnorr_n_Poincare_ergodic}, a nearly identical proof gives us the same result for weakly quasi-$\Pi^0_n$-generic and Schnorr $n$-random points and ergodic transformations.

\begin{thm}\label{CPS_Schnorr_n_weak_quasi_pi0n_generic_Poincare_ergodic}
    Let $(X, \mu)$ be a computable probability space, $n \geq 1$, and $A \subseteq \N$. If $x \in X$ is both Schnorr $n$-random relative to $A$ and weakly quasi-$\Pi^0_n$-generic relative to $A$, then for any $\Pi^0_n(A)$ set $P \ni x$ with $\mu(P) > 0$ and any computable ergodic measure-preserving transformation $T: \ \subseteq \! X \to X$ with $x \in \dom T$, we have $T^k(x) \in P$ for some $k > 0$.
\end{thm}

As with $\Pi^0_n$-genericity, these two effective recurrence theorems only guarantee that \emph{some} iterate of the point $x$ returns to the set $P$, not that infinitely many iterates return. In order to achieve infinitely many, we would need to know that (weak) quasi-$\Pi^0_n$-genericity is preserved by the transformation $T$. The following lemma provides a sufficient condition for this, which is proven the same as \cref{preservation_of_pi0n_genericity}.

\begin{lem}\label{preservation_of_quasi-pi0n_genericity}
    Let $(X, \mu, T)$ be a computable measure-preserving system, $n \geq 1$, and $A \subseteq \N$. If $T$ preserves $\Pi^0_1\paren{A^{(n - 1)}}$ sets with $A^{(n - 1)}$-computable measure, then $T$ preserves (weak) quasi-$\Pi^0_n$-genericity relative to $A$.
\end{lem}

We can use this result to show that the shift on $\paren{2^\N, \lambda}$, for example, preserves (weak) quasi-$\Pi^0_n$-genericity.

\begin{cor}
    Consider Cantor space with the Lebesgue measure the shift map $T: 2^\N \to 2^\N$. Let $n \geq 1$ and $A \subseteq \N$. Then $T$ preserves (weak) $\Pi^0_n$-genericity relative to $A$.
\end{cor}
\begin{proof}
    By \cref{shift_of_pi01}, $T$ preserves $\Pi^0_1(A^{(n - 1)})$ sets with $A^{(n - 1)}$-computable Lebesgue measure. Then by \cref{preservation_of_quasi-pi0n_genericity}, $T$ preserves (weak) quasi-$\Pi^0_n$-genericity relative to $A$.
\end{proof}

In considering a potential converse to these theorems involving recurrence in $\Pi^0_n$ sets, a natural approach mirroring \cref{Schnorr_converse} would be the following.

\begin{qn}
    Considering Cantor space with the Lebesgue measure, given a point $x \in 2^\N$ which is neither Martin-L\"of random nor weakly quasi-1-generic, do there exist a $\Pi^0_1$ set $A \ni x$ with $\lambda(A) > 0$ and a computable measure-preserving transformation $T: \ \subseteq \! 2^\N \to 2^\N$ such that $T^k(x) \notin A$ for all $k > 0$?
\end{qn}

One could also try to strengthen this by making the transformation ergodic or by replacing ``weakly quasi-1-generic'' with ``quasi-1-generic.''

Of course, before going to all this trouble, one should first prove that there exists a Schnorr random point which is neither Martin-L\"of random nor weakly quasi-1-generic. (Otherwise, \cref{Schnorr_converse} would suffice in resolving this.)

An affirmative answer to the question above would likely provide some clarity regarding the relationship between $n$-randomness and recurrence in $\Pi^0_n$ sets for non-ergodic transformations, potentially leading to a way of improving \cref{CPS_weak_n_Poincare} towards $n$-randomness.

\subsection{Effective ergodic theorems for higher randomness}

Our last effective recurrence theorems will involve the frequency of recurrence, i.e., the long-term behavior of the \emph{ergodic averages}
\[A_Nf(x) := \frac{1}{N} \sum_{k < N} \paren{f \circ T^k}(x)\]
when $f$ is the characteristic function of a set $P$ and $x \in P$. The Birkhoff Ergodic Theorem (\cref{Birkhoff_ergodic_theorem}) tells us that these converge almost everywhere, a property which has been studied in the effective setting for various kinds of functions $f$ and transformations $T$, as discussed in the introduction. However, because we want to consider sets $A$ more complex than $\Pi^0_1$, we will first generalize some of these other results to more complex functions in any computable probability space.

The first of these requires the following upcrossing inequality due to \cite{Bishop}.

\begin{thm}\label{upcrossing_inequality}
    Let $(X, \mu, T)$ be a measure-preserving system and $f: X \to \cl\R$ be integrable. For $\alpha < \beta$, let $\tau_{\alpha, \beta}f(x)$ be the number of times the ergodic averages of $f$ at $x$ go from below $\alpha$ to above $\beta$. Then we have
    \[\int \tau_{\alpha, \beta}f \dmu \leq \frac{1}{\beta - \alpha}\int(f - \alpha)^+ \dmu,\]
    where $g^+ := \max\{g, 0\}$ denotes the positive part of $g$. In particular, $\tau_{\alpha, \beta}f$ is integrable.
\end{thm}

The following effective version of the Birkhoff Ergodic Theorem for $n = 1$ was proven by \cite{Vyugin1998}. In generalizing it to arbitrary $n$, we follow the streamlined argument given by Theorem 6.2 of \cite{FranklinTowsner2012}. (See \cref{n_computability_definition} for the definition of $n$-computability and $n$-lower/upper-computability.)

\begin{thm}\label{nR_Birkhoff}
    Let $(X, \mu, T)$ be a computable measure-preserving system. Let $f: \ \subseteq \! X \to \cl\R$ be $n$-computable, defined almost everywhere, and integrable. Then for all $n$-random $x \in \bigcap_k T^{-k}(\dom f)$, $\lim_{N \to \infty} A_Nf(x)$ exists.
\end{thm}
\begin{proof}
    Fix rational numbers $\alpha < \beta$. For any positive integers $M_1 < N_1 < \cdots < M_t < N_t$, define sets
    \[U_{M_1, \dots, M_t}^{N_1, \dots, N_t} = \bigcap_{j = 1}^t \set{x \in X: A_{M_j}f(x) < \alpha < \beta < A_{N_j}f(x)},\]
    which are uniformly $\Sigma^0_n$ in $D := \bigcap_k T^{-k}(\dom f)$ since the ergodic averages are uniformly $n$-computable by \cref{ergodic_averages_computable}. Then we see that the sets
    \[\set{x \in X: \tau_{\alpha, \beta}f(x) > q} = \begin{dcases}
        X & \text{if } q < 0 \\
        \bigcup_{M_1 < N_1 < \cdots < M_t < N_t} U_{M_1, \dots, M_t}^{N_1, \dots, N_t} \quad \text{(where } t = \floor q + 1) & \text{if } q \geq 0
    \end{dcases}\]
    are also uniformly $\Sigma^0_n$ in $D$, so $\tau_{\alpha, \beta}f$ is $n$-lower-semicomputable.
    
    Since $f$ is integrable, $\tau_{\alpha, \beta}f$ is also integrable by \cref{upcrossing_inequality}, so we may take $q \in \Q$ such that $\int \tau_{\alpha, \beta}f \dmu \leq q$. Then we can write
    \[\set{x \in X: \tau_{\alpha, \beta}f(x) > 2^jq} = U_j \cap D\]
    for some sets $U_j$ which are uniformly $\Sigma^0_n$ in $D$, which has full measure since $T$ is measure-preserving and $f$ is defined almost everywhere. Then by Chebyshev's Inequality and the definition of $q$, these have measure
    \[\mu\paren{U_j} = \mu\paren{U_j \cap D} \leq \frac{1}{2^jq}\int \tau_{\alpha, \beta}f \dmu \leq 2^{-j}.\]
    Hence $\paren{U_j}_j$ is an $n$-test which captures every point $x \in D$ with
    \[\liminf_{N \to \infty} A_Nf(x) < \alpha < \beta < \limsup_{N \to \infty} A_Nf(x)\]
    since then $\tau_{\alpha, \beta}f(x) = \infty$. Such rational numbers $\alpha < \beta$ exist whenever the ergodic averages do not converge, which implies that the ergodic averages do converge at all $n$-random points in $D$.
\end{proof}

The previous result helps us prove another effective version of the Birkhoff Ergodic Theorem for $n$-lower-semicomputable functions, which was shown in the case $n = 1$ for Cantor space by \cite{Miyabe2016}. We follow the same argument, applying the machinery of computability probability spaces developed in \cref[S]{section_background}.

\begin{thm}\label{OWnR_Birkhoff}
    Let $(X, \mu, T)$ be a computable measure-preserving system and $n \geq 1$. Let $f: \ \subseteq \! X \to [0, \infty]$ be $n$-lower-semicomputable, defined almost everywhere, and integrable. Then for all Oberwolfach $n$-random $x \in \bigcap_k T^{-k}(\dom f)$, $\lim_{N \to \infty} A_Nf(x)$ exists.
\end{thm}
\begin{proof}
    According to the following cases, there exists an increasing sequence of uniformly $n$-computable functions $\paren{f_k}_k$ which converges to $f$ on $\dom f$ and has uniformly $\zero^{(n - 1)}$-computable integrals:
    \begin{itemize}
        \item $n = 1$. By \cref{approximate_lower_semicomputable_from_below}, there are an increasing computable sequence $\paren{M_k}_k$ in $\Q_{\geq 0}$ and a sequence of uniformly computable functions $f_k: X \to [0, M_k]$ such that $f_k \upto f$ on $\dom f$ as $k \to \infty$. By \cref{integral_of_computable}, their integrals $\int f_k \dmu$ are uniformly computable.

        \item $n \geq 2$. By \cref{approximate_(n+1)_lower_semicomputable_from_below}, there is a sequence of uniformly $(n - 1)$-upper-semicomputable functions $f_k: X \to [0, \infty)$ such that $f_k \upto f$ on $\dom f$ as $k \to \infty$. By \cref{integral_of_n_lower_semicomputable}, their integrals $\int f_k \dmu$ are uniformly $\zero^{(n - 2)}$-upper-semicomputable and hence uniformly $\zero^{(n - 1)}$-computable. The functions $f_k$ themselves are also uniformly $n$-computable by \cref{n_lower_semicomputable_implies_(n+1)_computable}.
    \end{itemize}

    Fix rational numbers $\alpha < \beta$ and define sets
    \[U_k = \bigcup_{N > 0}\set{x \in X: A_N\paren{f - f_k}(x) > \beta - \alpha}\]
    which are uniformly $\Sigma^0_n$ in $D := \bigcap_k T^{-k}(\dom f)$ by \cref{ergodic_averages_computable} because the functions $f - f_k$ are uniformly $n$-lower-semicomputable. Let
    \[\beta_k = \frac{1}{b - a}\int f_k \dmu\quad\text{and}\quad\beta = \frac{1}{b - a}\int f \dmu.\]
    Then the numbers $\beta_k$ are uniformly $\zero^{(n - 1)}$-computable, and $\beta_k \upto \beta$ as $k \to \infty$ by the Monotone Convergence Theorem. Using the Maximal Ergodic Theorem (\cref{maximal_ergodic_theorem}), we obtain
    \[\mu\paren{U_k} \leq \frac{1}{b - a}\int \paren{f - f_k} \dmu = \beta - \beta_k,\]
    which makes $\paren{U_k}_k$ an Oberwolfach $n$-test.

    Now consider any $x \in X$ and assume that $\lim_{N \to \infty} A_Nf(x)$ does not exist, in which case there exist rational numbers $\alpha < \beta$ such that
    \[\liminf_{N \to \infty} A_Nf(x) < \alpha < \beta < \limsup_{N \to \infty} A_Nf(x).\]
    We seek to show that $x$ is not Oberwolfach $n$-random. This is trivially true if $x$ is not $n$-random, so assume $x$ is $n$-random and consider any $k$. Since $f_k$ is $n$-computable, we know by \cref{nR_Birkhoff} that $\lim_{N \to \infty} A_Nf_k(x)$ exists. As $f_k \leq f$, it follows that
    \[\lim_{N \to \infty} A_Nf_k(x) \leq \liminf_{N \to \infty} A_Nf(x) < \alpha,\]
    so $A_Nf_k(x) < \alpha$ for all sufficiently large $N$, say all $N \geq N_0$. Also, as $\limsup_{N \to \infty} A_Nf(x) > \beta$, we must have $A_Nf(x) > \beta$ for infinitely many $N$. Taking such an $N \geq N_0$, we see that $A_N\paren{f - f_k}(x) > \beta - \alpha$, i.e., $x$ lies in the set $U_k$ defined above for this choice of $\alpha < \beta$.
    
    Thus, $x \in \bigcap_k U_k$, which means that $x$ is not Oberwolfach $n$-random.
\end{proof}

\subsection{Frequency of recurrence}

In order to describe which points recur at a positive frequency, we will make use of the Multiple Recurrence Theorem due to \cite{Furstenberg1978}.

\begin{thm}\label{multiple_recurrence}
    Let $(X, \B, \mu)$ be a probability space and $T_1, \dots, T_j: \ \subseteq \! X \to X$ be commuting measure-preserving transformations. Take $A \in \B$ with $\mu(A) > 0$. Then
    \[\liminf_{N \to \infty}\frac{1}{N}\sum_{k < N}\mu\paren{\bigcap_{i = 1}^j T_i^{-k}A} > 0.\]
\end{thm}

From this we can derive the following pointwise result which tells us that almost every point in $A$ returns to $A$ simultaneously under all transformations $T_i$ at a positive upper frequency. Having been unable to find this fact anywhere in the literature, we provide a proof, which will also be useful in effectivizing it.

\begin{thm}\label{multiple_recurrence_frequency}
    Let $(X, \B, \mu)$ be a probability space and $T_1, \dots, T_j: \ \subseteq \! X \to X$ be commuting measure-preserving transformations. Take $A \in \B$ with $\mu(A) > 0$. Then for $\mu$-a.e.\ $x \in A$, we have
    \[\limsup_{N \to \infty} \frac{1}{N}\sum_{k < N} \prod_{i = 1}^j \paren{\one_A \circ T_i^k}(x) > 0.\]
\end{thm}
\begin{proof}
    Let $T_0$ be the identity map, which is measure-preserving and computes with every other transformation $T_i$. For any set $C \subseteq X$, define
    \[f_C = \limsup_{N \to \infty} \frac{1}{N}\sum_{k < N} \prod_{i = 0}^j \paren{\one_C \circ T_i^k}.\]
    Let
    \[B = \set{x \in A: f_A(x) = 0}.\]
    Consider any $x \in X$. If $x \notin \bigcap_{i = 0}^j T_i^{-k}B$ for all $n$, then clearly $f_B(x) = 0$. Otherwise, if $x \in \bigcap_{i = 0}^j T_i^{-k}B$ for some $n$, then in particular $x \in B$, so $f_B(x) \leq f_A(x) = 0$. Thus, $f_B = 0$ everywhere. Now observe by Fatou's Lemma that
    \begin{align*}
        \liminf_{N \to \infty} \frac{1}{N}\sum_{k < N} \mu\paren{\bigcap_{i = 0}^j T_i^{-k}B} &\leq \limsup_{N \to \infty} \frac{1}{N}\sum_{k < N} \int \prod_{i = 0}^j \paren{\one_B \circ T_i^k} \dmu \\
        &\leq \int \limsup_{N \to \infty} \frac{1}{N}\sum_{k < N} \prod_{i = 0}^j \paren{\one_B \circ T_i^k} \dmu = \int f_B \dmu = 0.
    \end{align*}
    By \cref{multiple_recurrence}, we then must have $\mu(B) = 0$. Since $\bigcap_{i = 1}^j T_i^{-k}A \supseteq \bigcap_{i = 0}^j T_i^{-k}A$, this implies that for $\mu$-a.e.\ $x \in A$,
    \begin{align*}
        \limsup_{N \to \infty} \frac{1}{N}\sum_{k < N} \prod_{i = 1}^j \paren{\one_A \circ T_i^k} &\geq f_A(x) > 0. \qedhere
    \end{align*}
\end{proof}

Note that it was the use of Fatou's Lemma in this proof that forced us to consider the limit superior rather than the limit inferior. This problem disappears when there is just a single transformation since the limit exists almost everywhere by the Birkhoff Ergodic Theorem.

\begin{cor}\label{recurrence_frequency}
    Let $(X, \B, \mu, T)$ be a measure-preserving system. Take $A \in \B$ with $\mu(A) > 0$. Then for $\mu$-a.e.\ $x \in A$, we have
    \[\lim_{N \to \infty} \frac{1}{N}\sum_{k < N} \one_A \circ T^k(x) > 0.\]
\end{cor}

We now state effective versions of these results. The first, dealing with multiple recurrence, is very weak due to the fact that multiple ergodic averages are not known to converge pointwise almost everywhere without more restrictive conditions on the transformations.
\begin{thm}\label{CPS_weak_multiple_recurrence_frequency}
    Let $(X, \mu)$ be a computable probability space, $n \geq 1$, and $x \in X$ be weakly $(n + 2)$-random. Then for any $\Sigma^0_n$ set $A \ni x$ with $\mu(A) > 0$ and any commuting computable measure-preserving transformations $T_1, \dots, T_j: \ \subseteq \! X \to X$ with $x \in \bigcap_{i = 1}^j \dom T_i$, we have
    \[\limsup_{N \to \infty} \frac{1}{N}\sum_{k < N} \prod_{i = 1}^j \paren{\one_A \circ T_i^k}(x) > 0.\]
\end{thm}
\begin{proof}
    Since $A$ is $\Sigma^0_n$ and each transformation $T_i$ is computable, by \cref{iterated_preimage_of_sigma0n}, the sets $\bigcap_{i = 1}^j T_i^{-k}A$ are likewise $\Sigma^0_n$, uniformly in $k$. Hence the functions
    \[f_N := \frac{1}{N}\sum_{k < N} \prod_{i = 1}^j \paren{\one_A \circ T_i^k}\]
    are $n$-lower-semicomputable, uniformly in $N > 0$. This implies that the sets
    \[P_{m, N} := \set{x \in X: f_N(x) \leq 2^{-m}},\]
    are uniformly $\Pi^0_n$. Let
    \[B = \set{x \in A: \limsup_{N \to \infty} f_N(x) = 0},\]
    which is null by \cref{multiple_recurrence_frequency}. Then we can write
    \[B = A \cap \bigcap_m\bigcup_{N_0}\bigcap_{N > N_0} P_{m, N},\]
    so $B$ is $\Pi^0_{n + 2}$. Therefore, any weakly $(n + 2)$-random point will not lie in $B$.
\end{proof}

With just a single transformation, however, we are able to obtain a much stronger result.

\begin{thm}\label{CPS_weak_recurrence_frequency}
    Let $(X, \mu)$ be a computable probability space, $n \geq 1$, and $x \in X$ be weakly $(n + 1)$-random. Then for any $\Pi^0_n$ set $A \ni x$ with $\mu(A) > 0$ and any computable measure-preserving transformation $T: \ \subseteq \! X \to X$ with $x \in \dom T$, we have
    \[\lim_{N \to \infty} \frac{1}{N}\sum_{k < N} \one_A \circ T^k(x) > 0.\]
\end{thm}
\begin{proof}
    Since $A$ is $\Pi^0_n$, its characteristic function is $n$-upper-semicomputable. By \cref{ergodic_averages_computable}, so are the functions
    \[f_N := \frac{1}{N}\sum_{k < N} \one_A \circ T^k,\]
    uniformly in $N > 0$. This implies that the sets
    \[U_{m, N} := \set{x \in X: f_N(x) < 2^{-m}},\]
    are uniformly $\Sigma^0_n$. Let
    \[B = \set{x \in A: \liminf_{N \to \infty} f_N(x) = 0},\]
    which is null by \cref{recurrence_frequency}. Then we can write
    \[B = A \cap \bigcap_m\bigcup_{N > 0} U_{m, N},\]
    so $B$ is $\Pi^0_{n + 1}$. Therefore, any weakly $(n + 1)$-random point will not lie in $B$. Moreover, we know by \cref{OWnR_Birkhoff} that $\lim_{N \to \infty} f_N(x)$ exists since any weakly $(n + 1)$-random point is also Oberwolfach $n$-random.
\end{proof}

\section*{Acknowledgments}

The author would like to extend his sincerest thanks to Jan Reimann, who could always be relied on for his guidance, encouragement, and clarity throughout this project.

\section*{Declaration of AI use}

During the preparation of this work, ChatGPT and Gemini were used in order to search for related literature and check whether various facts are already known. All proofs included here were written entirely by the author.

\phantomsection
\addcontentsline{toc}{section}{References}

\bibliographystyle{cas-model2-names}

\bibliography{references}

\appendix
\section{Appendix}\label{section_background}

This appendix provides many auxiliary lemmas about computable probability spaces, as well as the notions of algorithmic randomness that are relevant for our study. We assume familiarity with the basics of computability theory, a helpful reference for which is given by \cite{Nies}.

To avoid repetition throughout this appendix, if we have a definition of what it means for a particular object $X$ to have a property P which involves the existence of a natural number, we will say that a sequence $\paren{X_n}_n$ of objects has the property P \emph{uniformly} if there is a computable function $f: \N \to \N$ such that $f(n)$ gives a suitable natural number for object $X_n$.

\subsection{Computable metric spaces}

In order to consider effective ergodic theorems for spaces other than Cantor space, we first must introduce computable metric spaces.

\subsubsection{Definitions}

The following definition is taken from \cite{Gacs2011}.

\begin{defn}
    A \textbf{computable metric space} is a triple $(X, d, S)$, sometimes abbreviated simply to $X$ if we do not need to refer to $d$ or $S$, such that:
    \begin{itemize}
        \item $(X, d)$ is a complete separable metric space,
        \item $S = \paren{s_i}_i$ is a dense sequence in $(X, d)$, and
        \item $d\paren{s_i, s_j}$ is a computable real number, uniformly in $i$ and $j$.
    \end{itemize}
\end{defn}

The topology in a computable metric space is generated by ideal balls.

\begin{defn}
    Let $(X, d, S)$ be a computable metric space, where $S = \paren{s_i}_i$. We call the points $s_i$ \textbf{ideal points} and the sets $B_{\cs{i, j}} := \set{x \in X: d\paren{x, s_i} < q_j}$ \textbf{ideal balls}, where $\paren{q_j}_j$ is a fixed enumeration of $\Q_{> 0}$ and $\cs{\cdot, \cdot}: \N^2 \to \N$ is the standard pairing function.
\end{defn}

The computable metric space we care most about is Cantor space: $\paren{2^\N, d, S}$, where $d(x, y) = 2^{-\min\set{i: \ x_i \neq y_i}}$ and $S$ is an effective enumeration of the sequences of the form $\sigma 0^\infty$ with $\sigma \in 2^{<\N}$. In this case, the ideal balls are simply the cylinder sets $\cyl\sigma = \set{x \in 2^\N: \sigma \subset x}$ with $\sigma \in 2^{<\N}$, which are clopen.

Using these ideal balls, we can consider an arithmetical hierarchy relative to any oracle. 

\begin{defn}
    Let $X$ be a computable metric space, $A \subseteq \N$, and $W_e$ denote the $e\th$ $A$-c.e.\ subset of $\N$.
    \begin{itemize}
        \item A set $U \subseteq X$ is \textbf{$\Sigma^0_1(A)$} (or \textbf{$A$-effectively open}) if there is an $A$-c.e.\ subset $W \subseteq \N$ such that $U = \bigcup_{i \in W} B_i$. If $W = W_e$, this is called the \textbf{$e\th$ $\Sigma^0_1(A)$ subset of $X$}. When $A = \zero$, we just write \textbf{$\Sigma^0_1$} or \textbf{effectively open}.
        \item For $n \geq 1$, the following are defined recursively:
        \begin{itemize}
            \item A set $P \subseteq X$ is \textbf{$\Pi^0_n(A)$} (or \textbf{$A$-effectively closed} if $n = 1$) if $\neg P$ is $\Sigma^0_n(A)$. When $A = \zero$, we just write \textbf{$\Pi^0_n$} (or \textbf{effectively closed} if $n = 1$).
            \item Denoting the $e\th$ $\Sigma^0_n(A)$ subset of $X$ by $U_e^n$, a set $U \subseteq X$ is \textbf{$\Sigma^0_{n + 1}(A)$} if there is an $A$-c.e.\ subset $W \subseteq \N$ such that $U = \bigcup_{i \in W} \neg U_i^n$. If $W = W_e$, this is called the \textbf{$e\th$ $\Sigma^0_{n + 1}(A)$ subset of $X$}. When $A = \zero$, we just write \textbf{$\Sigma^0_{n + 1}$}.
        \end{itemize}
        \item For any complexity class $\Gamma$ and any set $D \subseteq X$, a subset $U \subseteq D$ is \textbf{$\Gamma$ in $D$} if there is a $\Gamma$ set $V \subseteq X$ such that $U = V \cap D$.
    \end{itemize}
\end{defn}

\subsubsection[Hierarchical structure of Σ0n sets]{Hierarchical structure of $\Sigma^0_n$ sets}

It is clear that $\Sigma^0_n$ sets in a computable metric space are always $\Pi^0_{n + 1}$. Showing that they are also $\Sigma^0_{n + 1}$ requires a few elementary lemmas, of which we often make implicit use.

\begin{lem}\label{union_of_sigma0n}
    Let $X$ be a computable metric space, $n \geq 1$, and $A \subseteq \N$. Then the union of $A$-uniformly $\Sigma^0_n(A)$ sets $U_k$ is $\Sigma^0_n(A)$, uniformly in $\paren{U_k}_k$. Likewise, the intersection of $A$-uniformly $\Pi^0_n(A)$ sets $P_k$ is $\Pi^0_n(A)$, uniformly in $\paren{P_k}_k$.
\end{lem}

\begin{lem}\label{intersection_of_sigma0n}
    Let $X$ be a computable metric space, $n \geq 1$, and $A \subseteq \N$. Then a finite intersection of $\Sigma^0_n(A)$ sets is $\Sigma^0_n(A)$, uniformly in the finite set of indices. Likewise, a finite union of $\Pi^0_n(A)$ sets is $\Pi^0_n(A)$, uniformly in the finite set of indices.
\end{lem}

The next lemma comes down to the fact that $d\paren{x, s_i} > r_j$ if and only if $d\paren{x, s_k} < q_\l$ for some $k$ and $\l$ such that $d\paren{s_i, s_k} > r_j + q_\l$.

\begin{lem}\label{closed_balls_closed}
    Let $(X, d, S)$ be a computable metric space and $\paren{r_j}_j$ be a uniformly upper-semicomputable sequence in $\R_{\geq 0}$. Then the closed balls
    \[\set{x \in X: d\paren{x, s_i} \leq r_j}\quad\text{are uniformly $\Pi^0_1$.}\]
\end{lem}

By writing ideal balls as unions of closed balls, \cref{closed_balls_closed} can be used to obtain the following result.

\begin{lem}\label{approximate_sigma01_or_pi01_cms}
    Let $(X, d, S)$ be a computable metric space. Then:
    \begin{itemize}
        \item For any $\Sigma^0_1$ set $U$, there is an increasing uniformly $\Pi^0_1$ sequence $\paren{P_s}_s$ computable from $U$ such that $U = \bigcup_s P_s$.
        \item For any $\Pi^0_1$ set $P$, there is a nested uniformly $\Sigma^0_1$ sequence $\paren{U_s}_s$ computable from $P$ such that $P = \bigcap_s U_s$.
    \end{itemize}
\end{lem}

Applying induction then immediately yields the sought hierarchy.

\begin{cor}\label{sigma0n_is_sigma0(n+1)}
    Let $X$ be a computable metric space and $n \geq 1$. Then $\Sigma^0_n$ sets are uniformly $\Sigma^0_{n + 1}$, and $\Pi^0_n$ sets are uniformly $\Pi^0_{n + 1}$.
\end{cor}

\subsubsection[Relationship between Σ0n subsets of X and Σ0n subsets of N]{Relationship between $\Sigma^0_n$ subsets of $X$ and $\Sigma^0_n$ subsets of $\N$}

The following lemma allows us to use $\Sigma^0_n$ formulas in definitions of uniformly $\Sigma^0_n$ sets.

\begin{lem}\label{sigma0n(A)_for_N_is_sigma0m(A(n-m))_for_X}
    Let $X$ be a computable metric space and $A \subseteq \N$. Then for all $1 \leq m \leq n$ and all sets $S \subseteq \N$ which are $\Sigma^0_n(A)$ or $\Pi^0_n(A)$, the sets
    \[V_k = \begin{cases}
        X & \text{if } k \in S \\
        \zero & \text{if } k \notin S
    \end{cases}\]
    are $\Sigma^0_m\paren{A^{(n - m)}}$ or $\Pi^0_m\paren{A^{(n - m)}}$, respectively, uniformly in $k$ and a $\Sigma^0_n(A)$ or $\Pi^0_n(A)$ index for $S$. When $m = n$, this in particular implies that the sets $V_k$ are $\Sigma^0_n(A)$ or $\Pi^0_n(A)$, respectively.
\end{lem}
\begin{proof}
    By considering the sets
    \[\neg V_k = \begin{cases}
        X & \text{if } k \in \neg S \\
    \zero & \text{if } k \notin \neg S
    \end{cases},\]
    we can see that the statements for when $S$ is $\Sigma^0_n(A)$ and when $S$ is $\Pi^0_n(A)$ are equivalent.

    We first prove the claim when $m = 1$. Consider any $n \geq 1$ and any $\Sigma^0_n(A)$ set $S \subseteq \N$. By Post's Theorem, $S$ is $A^{(n - 1)}$-c.e., which implies that the sets $W_k = \{i \in \N: k \in A\}$ are uniformly $A^{(n - 1)}$-c.e.\ given $k$ and $S$. (This just means that $W_k = \N$ if $k \in S$ and $W_k = \zero$ otherwise.) Letting $B_i$ denote the $i\th$ ideal ball in $X$, we see that $V_k = \bigcup_{i \in W_k} B_i$ is uniformly $\Sigma^0_1\paren{A^{(n - 1)}}$ in $k$ and $S$.

    Now assume for some $m \geq 1$ that the claim holds for all $n \geq m$ and all $\Sigma^0_n(A)$ sets $S \subseteq \N$. Consider any $n \geq m + 1$ and any $\Sigma^0_n(A)$ set $S \subseteq \N$. Write $S = \bigcup_j S_j$ for some uniformly $\Pi^0_{n - 1}(A)$ sets $S_j \subseteq \N$ and define
    \[V_{j, k} = \begin{cases}
        X & \text{if } k \in S_j \\
    \zero & \text{if } k \notin S_j
    \end{cases}.\]
    Since $m \leq n - 1$, these sets are uniformly $\Pi^0_m\paren{A^{((n - 1) - m)}}$ in $j$, $k$, and $S$ by the inductive hypothesis. Then $V_k = \bigcup_j V_{j, k}$ is uniformly $\Sigma^0_{m + 1}\paren{A^{(n - (m + 1))}}$ in $k$ and $S$, so the claim holds for $m + 1$ as well. Thus, the claim holds for all $m \geq 1$.
\end{proof}

We have two important corollaries of this result.

\begin{cor}\label{zero(n-1)_uniformly_sigma0n_is_uniformly_sigma0n}
    Let $X$ be a computable metric space, $A \subseteq \N$, and $n \geq 1$. Then a sequence $\paren{U_k}_k$ of $A^{(n - 1)}$-uniformly $\Sigma^0_n(A)$ sets is uniformly $\Sigma^0_n(A)$, uniformly given $\paren{U_k}_k$.
\end{cor}
\begin{proof}
    Let $U_e$ be the $e\th$ $\Sigma^0_n(A)$ set and $f: \N \to \N$ be an $A^{(n - 1)}$-computable function. Then the graph $\Gamma := \set{\cs{j, f(j)}: j \in \N}$ of $f$ is $A^{(n - 1)}$-c.e.\ and hence $\Sigma^0_n(A)$ by Post's Theorem. Write
    \[U_{f(j)} = \bigcup_e \paren{U_e \cap V_{j, e}},\quad\text{where}\quad V_{j, e} = \begin{cases}
        X & \text{if } \cs{j, e} \in \Gamma \\
        \zero & \text{if } \cs{j, e} \notin \Gamma
    \end{cases}.\]
    By \cref{sigma0n(A)_for_N_is_sigma0m(A(n-m))_for_X}, the sets $V_{j, e}$ are uniformly $\Sigma^0_n(A)$, as are the sets $U_k$, along with their intersections $U_e \cap V_{j, e}$. Thus, the sets $U_{f(j)}$ are uniformly $\Sigma^0_n(A)$, and this is uniform in $f$.
\end{proof}

\begin{cor}\label{sigma0n(A(m))_is_sigma0(m+n)(A)}
    Let $X$ be a computable metric space, $A \subseteq \N$, and $n \geq 1$. Then any $\Sigma^0_n\paren{A'}$ set $U$ is also $\Sigma^0_{n + 1}(A)$, uniformly in a $\Sigma^0_n\paren{A'}$ index for $U$. Thus, any $\Sigma^0_n\paren{A^{(m)}}$ set $U$ is also $\Sigma^0_{m + n}(A)$, uniformly in a $\Sigma^0_n\paren{A^{(m)}}$ index for $U$.
\end{cor}
\begin{proof}
    First consider the case $n = 1$. If $U$ is $\Sigma^0_1\paren{A'}$, then $U = \bigcup_{i \in W} B_i$ for some set $W \subseteq \N$ which is $A'$-c.e., i.e., $\Sigma^0_2(A)$ by Post's Theorem. Thus, we can write
    \[U = \bigcup_i \ \paren{B_i \cap V_i},\quad\text{where}\quad V_i = \begin{cases}
        X & \text{if } i \in W \\
        \zero & \text{if } i \notin W
    \end{cases}.\]
    By \cref{sigma0n(A)_for_N_is_sigma0m(A(n-m))_for_X}, the sets $V_i$ are uniformly $\Sigma^0_2(A)$, as are the sets $B_i$ by \cref{sigma0n_is_sigma0(n+1)}, along with their intersections $B_i \cap V_i$. Thus, $U$ is $\Sigma^0_2(A)$, uniformly in a $\Sigma^0_1\paren{A'}$ index for $U$.

    Now assume the claim holds for some $n \geq 1$, and consider any $\Sigma^0_{n + 1}\paren{A'}$ set $U$. Write $U = \bigcup_k U_k$, where the sets $U_k$ are $A'$-uniformly $\Pi^0_n\paren{A'}$. Then by the inductive hypothesis, the sets $U_k$ are $A'$-uniformly $\Pi^0_{n + 1}(A)$ and hence uniformly $\Pi^0_{n + 1}(A)$ by \cref{zero(n-1)_uniformly_sigma0n_is_uniformly_sigma0n}. Thus, $U$ is $\Sigma^0_{n + 2}(A)$, uniformly in a $\Sigma^0_{n + 1}\paren{A'}$ index for $U$. 

    This implies that the claim holds for all $n \geq 1$. By induction, it follows that any $\Sigma^0_n\paren{A^{(m)}}$ set $U$ is also $\Sigma^0_{m + n}(A)$, uniformly in a $\Sigma^0_n\paren{A^{(m)}}$ index for $U$.
\end{proof}

\subsection{Computable probability spaces}

We now impose more structure on a computable metric space by adding a computable probability measure.

\subsubsection{Definition}

\begin{defn}
    Let $(X, d, S)$ be a computable metric space and $B_i$ denote the $i\th$ ideal ball of $X$. A Borel probability measure $\mu$ on $X$ is \textbf{computable} if $\mu\paren{\bigcup_{i \in F} B_i}$ is uniformly lower-semicomputable in $F \in \PP_{<\N}(\N)$. We then call $(X, d, S, \mu)$ a \textbf{computable probability space}, or just $(X, \mu)$ for brevity.
\end{defn}

For example, Cantor space can be considered a computable probability space when paired with the Lebesgue measure $\lambda$, which is defined by $\lambda\cyl\sigma = 2^{-|\sigma|}$ for each $\sigma \in 2^{<\N}$.

Since the supremum of a sequence $\paren{r_k}_k$ of uniformly lower-semicomputable real numbers is also lower-semicomputable (uniformly in the sequence), the definition above immediately yields the following result.

\begin{lem}\label{measure_of_sigma01_lower_semicomputable}
    Let $(X, \mu)$ be a computable probability space. Then the measures of $\Sigma^0_1$ sets are uniformly lower-semicomputable.
\end{lem}

\subsubsection{Almost decidable sets}

Cantor space is particularly nice because of its clopen basis of cylinder sets. Not every computable probability space has such a basis, but if we consider a slightly broader class of sets called ``almost decidable,'' it turns out that we can get a similar property. These were studied by \cite{Hoyrup2009}.

\begin{defn}
    Let $(X, \mu)$ be a computable probability space. Then a Borel set $A$ is \textbf{almost decidable} if there exist $\Sigma^0_1$ sets $U \subseteq A$ and $V \subseteq \neg A$ such that $U \cup V$ is dense and has full measure. (An index for $A$ is then a pair of $\Sigma^0_1$ indices for $U$ and $V$.)
\end{defn}

It is simple to show that almost decidable sets satisfy the following properties.

\begin{lem}\label{finite_union_of_almost_decidable_sets}
    Let $(X, \mu)$ be a computable probability space. Then a finite union of almost decidable sets is almost decidable, uniformly given indices for the almost decidable sets.
\end{lem}

\begin{lem}\label{almost_decidable_has_computable_measure}
    Let $(X, \mu)$ be a computable probability space and $A$ be almost decidable with $\Sigma^0_1$ sets $U \subseteq A$ and $V \subseteq \neg A$ such that $U \cup V$ is dense with full measure. Then
    \[\mu(U) = \mu(A) = \mu(\neg V).\]
    Thus, almost decidable sets have computable measure, uniformly given an index for the almost decidable set.
\end{lem}

The ideal balls in a computable probability space are not necessarily almost decidable. The following result, however, due to \cite{Hoyrup2009}, says that if we instead consider balls with radii coming from a particular computable sequence of real numbers, we do get an almost decidable basis.

This relies on the Computable Baire Category Theorem, which allows one to find uniformly computable radii $r_n > 0$ that are distinct from all distances between ideal points such that the spheres $\set{x \in X: d\paren{x, s_i} = r_n}$ have measure zero. These two properties ensure that the union of the $\Sigma^0_1$ sets $B\paren{s_i, r_n}$ and $\neg\cl B\paren{s_i, r_n}$ is dense with full measure, implying that $B\paren{s_i, r_n}$---along with $\cl B\paren{s_i, r_n}$---is almost decidable.

\begin{thm}[Lemma 5.1.1 in \cite{Hoyrup2009}]\label{almost_decidable_basis}
    Let $(X, \mu)$ be a computable probability space. Write
    \[B(s, r) = \{x \in X: d(x, s) < r\}\quad\text{and}\quad\cl B(s, r) = \{x \in X: d(x, s) \leq r\}.\]
    Then there is a dense uniformly computable sequence $\paren{r_n}_n$ in $\R_{> 0}$ such that the open balls $B\paren{s_i, r_n}$ and the closed balls $\cl B\paren{s_i, r_n}$ are uniformly almost decidable.
\end{thm}

Using this almost decidable basis, we can strengthen \cref{approximate_sigma01_or_pi01_cms} as follows when working inside a computable probability space.

\begin{lem}\label{approximate_sigma01_or_pi01}
    Let $(X, \mu)$ be a computable probability space. Then:
    \begin{itemize}
        \item For any $\Sigma^0_1$ set $U$, there is an increasing uniformly $\Pi^0_1$ (and uniformly almost decidable) sequence $\paren{P_s}_s$ such that $U = \bigcup_s P_s$ and $\mu\paren{P_s}$ is uniformly computable in $s$. Moreover, indices for $P_s$ and $\mu\paren{P_s}$ can be computed from $U$ and $s$.
        \item For any $\Pi^0_1$ set $P$, there is a nested uniformly $\Sigma^0_1$ (and uniformly almost decidable) sequence $\paren{U_s}_s$ such that $P = \bigcap_s U_s$ and $\mu\paren{U_s}$ is uniformly computable in $s$. Moreover, indices for $U_s$ and $\mu\paren{U_s}$ can be computed from $P$ and $s$.
    \end{itemize}
\end{lem}

\subsection{Approximation of functions and sets}

We will now consider a hierarchy of computable real-valued functions on computable metric and probability spaces, along with how they can be used for approximation.

\subsubsection{Computable real-valued functions}

The following definitions appear to be new but are closely related to the notion of $\SIGMA^0_n$-computability defined by \cite{Brattka}. Here we write $\cl\R$ for the extended real numbers, i.e., $[-\infty, \infty]$.

\begin{defn}\label{n_computability_definition}
    Let $X$ be a computable metric space and $n \geq 1$.
    \begin{itemize}
        \item We say $f: \ \subseteq \! X \to \cl\R$ is \textbf{$n$-lower-semicomputable} (or \textbf{lower-semicomputable} if $n = 1$) if $\{x \in X: f(x) > q\}$ is $\Sigma^0_n$ in $\dom f$, uniformly in $q \in \Q$.
        \item We say $f: \ \subseteq \! X \to \cl\R$ is \textbf{$n$-upper-semicomputable} (or \textbf{upper-semicomputable} if $n = 1$) if $\{x \in X: f(x) < q\}$ is $\Sigma^0_n$ in $\dom f$, uniformly in $q \in \Q$.
        \item We say $f: \ \subseteq \! X \to \cl\R$ is \textbf{$n$-computable} (or \textbf{computable} if $n = 1$) if $f$ is both $n$-lower-semicomputable and $n$-upper-semicomputable.
    \end{itemize}
\end{defn}

For example, the characteristic function of a $\Sigma^0_n$ set is always $n$-lower-semicomputable. The following lemmas are both simple to show.

\begin{lem}\label{n_lower_semicomputable_implies_(n+1)_computable}
    Let $X$ be a computable metric space and $f: \ \subseteq \! X \to \cl\R$ be either $n$-lower-semicomputable or $n$-upper-semicomputable. Then $f$ is also $(n + 1)$-computable, uniformly given an index for $f$.
\end{lem}

\begin{lem}\label{max_of_computable_functions}
    Let $(X, d, S)$ be a computable metric space and $f, g: \ \subseteq \! X \to \cl\R$ be both $n$-lower-semicomputable or both $n$-upper-semicomputable. Then $\max\{f, g\}$ is also $n$-lower-semicomputable or $n$-upper-semicomputable, respectively, uniformly in $f$ and $g$.
\end{lem}

An important example of a computable real-valued function is the distance from any ideal point.

\begin{lem}\label{distance_is_computable}
    Let $(X, d, S)$ be a computable metric space. Then the functions $d(\cdot, s_i)$ are uniformly computable.
\end{lem}

\subsubsection{Approximating functions from below}

Often it is useful to approximate more complicated real-valued functions using simpler ones. We achieve this through the following lemmas, where we use the notation $f_k \upto f$ as $k \to \infty$ to mean that $\paren{f_k}_k$ is an increasing sequence of functions that converges to $f$.

\begin{lem}\label{approximate_characteristic_function_of_ideal_ball_from_below}
    Let $(X, d, S)$ be a computable metric space with $i\th$ ideal ball $B_i$. Then there is a sequence $\paren{f_{i, j, k}}_{i, j, k}$ of uniformly computable functions $f_{i, j, k}: X \to [0, 1]$ such that $f_{i, j, k} \upto \one_{B_{\cs{i, j}}}$ as $k \to \infty$ for each $i$ and $j$.
\end{lem}
\begin{proof}
    Let $f = \one_{B_{\cs{i, j}}}$, i.e.,
    \[f(x) = \begin{cases}
        1 & \text{if } d(x, s_i) < q_j \\
        0 & \text{if } d(x, s_i) \geq q_j
    \end{cases}.\]
    Then it is straightforward to verify that the functions
    \[f_{i, j, k}(x) = \begin{cases}
        1 & \text{if } d\paren{x, s_i} < q_j - 2^{-k} \\
        2^k\paren{q_j - d\paren{x, s_i}} & \text{if } q_j - 2^{-k} \leq d\paren{x, s_i} \leq q_j \\
        0 & \text{if } d\paren{x, s_i} > q_j        
    \end{cases}\]
    satisfy the desired property.
\end{proof}

\begin{lem}\label{approximate_lower_semicomputable_from_below}
    Let $X$ be a computable metric space. If $f: \ \subseteq \! X \to [0, \infty]$ is lower-semicomputable, then there are an increasing computable sequence $\paren{M_k}_k$ in $\Q_{\geq 0}$ and a sequence of uniformly computable functions $f_k: X \to [0, M_k]$ such that $f_k \upto f$ on $\dom f$ as $k \to \infty$. Moreover, indices for $\paren{f_k}_k$ and $\paren{M_k}_k$ can be computed from an index for $f$.
\end{lem}
\begin{proof}
    Since $f$ is lower-semicomputable, for all $q \in \Q$, we have
    \[\{x \in X: f(x) > q\} = U_q \cap \dom f\]
    for some uniformly $\Sigma^0_1$ sets $U_q$. Write $U_q = \bigcup_{i \in W_q} B_i$ for some uniformly c.e.\ sets $W_q \subseteq \N$. By \cref{approximate_characteristic_function_of_ideal_ball_from_below}, there are uniformly computable functions $f_{i, j}: X \to [0, 1]$ such that $f_{i, j} \upto \one_{B_i}$ as $j \to \infty$ for all $i$. Let
    \[g(x) = \sup_{\substack{q \in \Q_{> 0} \\ i \in W_q \\ j \in \N}} \paren{q \cdot f_{i, j}(x)},\]
    where we say that $\sup\zero = 0$. Then $f = g$ on $\dom f$. Now consider the nonempty c.e.\ set
    \[W = \set{\cs{q, i, j}: q \in \Q_{> 0}, i \in W_q, j \in \N} \cup \{\cs{0, 0, 0}\},\]
    which is the range of some computable function $h$. Let $g_{\cs{q, i, j}}(x) = q \cdot f_{i, j}(x)$. Then the functions $f_k = \max_{s \leq k} g_{h(s)}$ are uniformly computable by \cref{max_of_computable_functions} and satisfy $f_k \upto f$ on $\dom f$ as $k \to \infty$. Moreover, if we let $\pi: \N \to \Q_{\geq 0}$ be the computable function satisfying $\pi\paren{\cs{q, i, j}} = q$, then $f_k \leq \max_{s \leq k} \pi(h(s)) =: M_k$.
\end{proof}

\begin{lem}\label{approximate_(n+1)_lower_semicomputable_from_below}
    Let $X$ be a computable metric space and $n \geq 1$. If $f: \ \subseteq \! X \to [0, \infty]$ is $(n + 1)$-lower-semicomputable, then there are an increasing computable sequence $\paren{M_k}_k$ in $\Q_{> 0}$ and a sequence of uniformly $n$-upper-semicomputable functions $f_k: X \to [0, M_k]$ such that $f_k \upto f$ on $\dom f$ as $k \to \infty$. Moreover, indices for $\paren{f_k}_k$ and $\paren{M_k}_k$ can be computed from an index for $f$.
\end{lem}
\begin{proof}
    Since $f$ is $(n + 1)$-lower-semicomputable, for all $q \in \Q$, we have
    \[\{x \in X: f(x) > q\} = U_q \cap \dom f\]
    for some uniformly $\Sigma^0_{n + 1}$ sets $U_q$. Let
    \[g(x) = \sup\{q \in \Q_{> 0}: x \in U_q\},\]
    where we say that $\sup\zero = 0$. Then $f = g$ on $\dom f$.
    
    Since $U_q$ is uniformly $\Sigma^0_{n + 1}$ in $q$, we can write $U_q = \bigcup_j P_{q, j}$ for some uniformly $\Pi^0_n$ sets $P_{q, j}$. Fix a computable bijection $\cs{\cdot, \cdot}: \Q_{> 0} \times \N \to \N$. Note that the functions $g_{\cs{q, j}} = q\one_{P_{q, j}}$ for $q \in \Q_{> 0}$ are uniformly $n$-upper-semicomputable since $\set{x \in X: g_{\cs{q, j}}(x) < r}$ is $\neg P_{q, j}$ if $r \leq q$ and $X$ if $r > q$.
    
    Then by \cref{max_of_computable_functions}, the functions $f_k = \max_{s \leq k} g_s$ are likewise uniformly $n$-upper-semicomputable and satisfy $f_k \upto f$ on $\dom f$ as $k \to \infty$. Moreover, if we let $\pi: \N \to \Q_{> 0}$ be the computable function satisfying $\pi\paren{\cs{q, i}} = q$, then $f_k \leq \max_{s \leq k} \pi(s) =: M_k$.
\end{proof}

\subsubsection{Integrals of functions}

The following two results due to \cite{Hoyrup2009} demonstrate that integrals of (lower-semi-)computable functions are (lower-semi-)computable.

\begin{lem}[Corollary 4.3.2 in \cite{Hoyrup2009}]\label{integral_of_computable}
    Let $(X, \mu)$ be a computable probability space and $M \geq 0$ be computable. If $f: \ \subseteq \! X \to [-M, M]$ is computable, defined almost everywhere, and integrable, then $\int f \dmu$ is also computable, uniformly given an index for $f$ and its bound $M$.
\end{lem}

\begin{lem}[Proposition 4.3.1 in \cite{Hoyrup2009}]\label{integral_of_lower_semicomputable}
    Let $(X, \mu)$ be a computable probability space. If $f: \ \subseteq \! X \to [0, \infty]$ is lower-semicomputable, defined almost everywhere, and integrable, then $\int f \dmu$ is also lower-semicomputable, uniformly in an index for $f$.
\end{lem}

There is a similar result for $n$-lower-semicomputable functions.

\begin{lem}\label{integral_of_n_lower_semicomputable}
    Let $(X, \mu)$ be a computable probability space and $n \geq 1$. If $f: \ \subseteq \! X \to [0, \infty]$ is $n$-lower-semicomputable, defined almost everywhere, and integrable, then $\int f \dmu$ is $\zero^{(n - 1)}$-lower-semicomputable, uniformly in an index for $f$.
\end{lem}
\begin{proof} 
    The case $n = 1$ is given by \cref{integral_of_lower_semicomputable}, so assume this holds for some $n \geq 1$ and consider any $(n + 1)$-lower-semicomputable function $f: \ \subseteq \! X \to [0, \infty]$. By \cref{approximate_(n+1)_lower_semicomputable_from_below}, there are an increasing computable sequence $\paren{M_k}_k$ in $\Q_{\geq 0}$ and a sequence of uniformly $n$-upper-semicomputable functions $f_k: X \to [0, M_k]$ such that $f_k \upto f$ on $\dom f$ as $k \to \infty$.
    
    Since $M_k - f_k$ is nonnegative, uniformly $n$-lower-semicomputable, and integrable, $\int\paren{M_k - f_k} \dmu$ is uniformly $\zero^{(n - 1)}$-lower-semicomputable by the inductive hypothesis. This implies that $\int f_k \dmu$ is uniformly $\zero^{(n - 1)}$-upper-semicomputable and hence uniformly $\zero^{(n)}$-computable. It also converges to $\int f \dmu$ by the Monotone Convergence Theorem since $\dom f$ has full measure, so $\int f \dmu$ is $\zero^{(n)}$-lower-semicomputable. Thus, the claim holds for all $n \geq 1$.
\end{proof}

By considering characteristic functions of $\Sigma^0_n$ sets, the following corollary is immediate, although it could also be proven directly.

\begin{cor}\label{measure_of_sigma0n}
    Let $(X, \mu)$ be a computable probability space and $n \geq 1$. Then the measures of $\Sigma^0_n$ sets are uniformly $\zero^{(n - 1)}$-lower-semicomputable.
\end{cor}

\subsubsection[Approximating Σ0n sets]{Approximating $\Sigma^0_n$ sets}

We have just seen that the measures of $\Sigma^0_n$ sets are always $\zero^{(n - 1)}$-lower-semicomputable. Sometimes they are also $\zero^{(n - 1)}$-upper-semicomputable, and fortunately, this property is preserved by intersections and unions.

\begin{lem}\label{intersection_of_sigma0n_with_computable_measure}
    Let $(X, \mu)$ be a computable probability space and $n \geq 1$. Let $U$ and $V$ be $\Sigma^0_n$ subsets of $X$ such that $\mu(U)$ and $\mu(V)$ are $\zero^{(n - 1)}$-computable. Then $\mu(U \cap V)$ and $\mu(U \cup V)$ are also $\zero^{(n - 1)}$-computable, uniformly given $U$, $V$, $\mu(U)$, and $\mu(V)$.
\end{lem}
\begin{proof}
    Since $U \cap V$ and $U \cup V$ are $\Sigma^0_n$, by \cref{measure_of_sigma0n} their measures are $\zero^{(n - 1)}$-lower-semicomputable, uniformly given $U$ and $V$. On the other hand, we can write
    \[\mu(U \cap V) = \mu(U) + \mu(V) - \mu(U \cup V),\]
    where $\mu(U)$ and $\mu(V)$ are $\zero^{(n - 1)}$-computable and $\mu(U \cup V)$ is $\zero^{(n - 1)}$-lower-semicomputable. Hence $\mu(U \cap V)$ is $\zero^{(n - 1)}$-upper-semicomputable, uniformly given $U$, $V$, $\mu(U)$, and $\mu(V)$. Rearranging the equation above shows that $\mu(U \cup V)$ is likewise $\zero^{(n - 1)}$-upper-semicomputable, uniformly given $U$, $V$, $\mu(U)$, and $\mu(V)$.
\end{proof}

Recall from \cref{sigma0n(A(m))_is_sigma0(m+n)(A)} that every $\Sigma^0_1\paren{\zero^{(n - 1)}}$ set is also $\Sigma^0_n$. The converse is not true in general, but we can always approximate a $\Sigma^0_n$ set from the outside arbitrarily well with a $\Sigma^0_1\paren{\zero^{(n - 1)}}$ set. The following result extends Theorem 6.8.3(i) from \cite{DowneyHirschfeldt} to any computable probability space while also preserving $\zero^{(n - 1)}$-computability of the measure.

\begin{lem}\label{approximate_sigma0n_with_sigma01(zero(n-1))}
    Let $(X, \mu)$ be a computable probability space and $n \geq 1$. Then for any $\Sigma^0_n$ set $U$ and rational $\eps > 0$, there exists a $\Sigma^0_1\paren{\zero^{(n - 1)}}$ set $V \supseteq U$ such that $\mu(V) < \mu(U) + \eps$. Moreover, an index for $V$ can be computed from $U$ and $\eps$.

    If in addition $\mu(U)$ is $\zero^{(n - 1)}$-computable, then $\mu(V)$ will also be $\zero^{(n - 1)}$-computable, uniformly given $U$, $\mu(U)$, and $\eps$.
\end{lem}
\begin{proof}
    Assume the result holds some $n \geq 1$, which it trivially does for $n = 1$ by taking $V = U$.
    
    Let $U$ be $\Sigma^0_{n + 1}$ and $\eps > 0$ be rational. Then we can write $U = \bigcup_i P_i$ for some uniformly $\Pi^0_n$ sets $P_i$. We now write $P_i = \bigcap_j U_{i, j}$ according to one of the two cases below:
    \begin{itemize}
        \item $n = 1$. In this case, use \cref{approximate_sigma01_or_pi01} to write $P_i = \bigcap_j U_{i, j}$ for some uniformly $\Sigma^0_1$ sets $U_{i, j}$ with uniformly computable measures.
        \item $n \geq 2$. In this case, write $P_i = \bigcap_j U_{i, j}$ for some uniformly $\Sigma^0_{n - 1}$ sets $U_{i, j}$, which have uniformly $\zero^{(n - 2)}$-lower-semicomputable measures.
    \end{itemize}
    Since $\mu(P_i)$ and $\mu\paren{U_{i, j}}$ are both uniformly $\zero^{(n)}$-computable and $\mu\paren{U_{i, j}} \to \mu\paren{P_i}$ as $j \to \infty$ for each $i$, there is a $\zero^{(n)}$-computable function $i \mapsto j_i$ such that $\mu\paren{U_{i, j_i}} < \mu\paren{P_i} + 2^{-(i + 2)}\eps$.
    
    Now define sets $V_{i, j} \supseteq U_{i, j}$ according to one of the two cases below:
    \begin{itemize}
        \item $n = 1$. In this case, just take $V_{i, j} = U_{i, j}$.
        \item $n \geq 2$. In this case, since the sets $U_{i, j}$ are uniformly $\Sigma^0_{n - 1}$, we can use the inductive hypothesis to get uniformly $\Sigma^0_1\paren{\zero^{(n - 2)}}$ sets $V_{i, j} \supseteq U_{i, j}$ such that $\mu\paren{V_{i, j}} < \mu\paren{U_{i, j}} + 2^{-(i + 2)}\eps$.
    \end{itemize}
    Then the sets $V_{i, j_i}$ are uniformly $\Sigma^0_1\paren{\zero^{(n)}}$, along with their union $V := \bigcup_i V_{i, j_i}$, uniformly in $U$ and $\eps$.

    Since $V_{i, j_i} \supseteq U_{i, j_i} \supseteq P_i$ for all $i$, we clearly have $V \supseteq U$. Moreover, since
    \[V = \bigcup_i V_{i, j_i} = U \cup \bigcup_i \paren{V_{i, j_i} \setminus P_i},\]
    where
    \[\mu\paren{V_{i, j_i} \setminus P_i} = \mu\paren{V_{i, j_i} \setminus U_{i, j_i}} + \mu\paren{U_{i, j_i} \setminus P_i} < 2^{-(i + 2)}\eps + 2^{-(i + 2)}\eps = 2^{-(i + 1)}\eps,\]
    we have
    \[\mu(V) \leq \mu(U) + \sum_i \mu\paren{V_{i, j_i} \setminus P_i} < \mu(U) + \sum_i 2^{-(i + 1)}\eps = \mu(U) + \eps.\]
    Now assume moreover that $\mu(U)$ is $\zero^{(n)}$-computable. Let $V_k = \bigcup_{i \leq k} V_{i, j_i}$, which is uniformly $\Sigma^0_1\paren{\zero^{(n)}}$ with uniformly $\zero^{(n)}$-computable measure in $k$ by \cref{intersection_of_sigma0n_with_computable_measure} because it is a finite union of uniformly $\Sigma^0_1\paren{\zero^{(n)}}$ sets with uniformly $\zero^{(n)}$-computable measures. Then since $P_k \subseteq V_k$, we have
    \[V \setminus V_k = \paren{U \setminus V_k} \cup \paren{\paren{\bigcup_i \paren{V_{i, j_i} \setminus P_i}} \setminus V_k} \subseteq \paren{U \setminus P_k} \cup \bigcup_{i > k} \paren{V_{i, j_i} \setminus P_i}\]
    and hence
    \[\mu\paren{V \setminus V_k} \leq \mu\paren{U \setminus P_k} + \sum_{i > k} \mu\paren{V_{i, j_i} \setminus P_i} < \mu\paren{U \setminus P_k} + \sum_{i > k} 2^{-(i + 1)}\eps = \mu(U) - \mu\paren{P_k} + 2^{-k}\eps.\]
    The quantity on the right-hand side is uniformly $\zero^{(n)}$-computable in $k$ and approaches 0 as $k \to \infty$. Thus, $\mu(V)$ is $\zero^{(n)}$-computable, uniformly in $U$, $\mu(U)$, and $\eps$.

    By induction, the claim holds for all $n \geq 1$.
\end{proof}

Later on we will also want the ability to approximate $\Sigma^0_1$ sets from the inside with compact $\Pi^0_1$ sets, which the following result shows is always possible. (In fact, we could upgrade this to \emph{computably compact} $\Pi^0_1$ sets, as studied by \cite{Downey2023}, but this will not be necessary for our purposes.)

\begin{lem}\label{approximate_sigma01_inside_with_compact}
    Let $(X, \mu)$ be a computable probability space, $U$ be $\Sigma^0_1$ and $\eps > 0$ be rational. Then there exists a compact $\Pi^0_1$ set $K \subseteq U$ such that $\mu(K) > \mu(U) - \eps$. If in addition $\mu(U)$ is computable, then an index for $K$ can be computed from $U$, $\mu(U)$, and $\eps$.
\end{lem}
\begin{proof}
    We assume without loss of generality that $\mu(U)$ is computable. (If this is not the case, then one can use \cref{almost_decidable_basis} to write $U$ as a c.e.\ union of almost decidable open balls. Some finite subcollection of these will have union $V$ with computable measure $\mu(V) > \mu(U) - \eps$. Then one can apply the result with $V$ and any rational $\delta$ such that $0 < \delta < \mu(V \setminus U) + \eps$.)
    
    For $s \in X$ and $r > 0$, let
    \[B(s, r) = \{x \in X: d(x, s) < r\}\quad\text{and}\quad\cl B(s, r) = \{x \in X: d(x, s) \leq r\}.\]
    Let $\paren{s_i}_i$ be the sequence of ideal points in $X$. Then for each $n$, we have
    \[X = \bigcup_i B\paren{s_i, 2^{-n}\eps}.\]
    Define sets
    \[U_{n, j} = U \cap \bigcup_{i < j} B(s_i, 2^{-n}\eps),\]
    which are uniformly $\Sigma^0_1$ and hence have uniformly lower-semicomputable measures. Also, for each $n$, the sequence $\paren{U_{n, j}}_j$ is increasing and satisfies $\bigcup_j U_{n, j} = U$. Consequently, there is a computable function $n \mapsto j_n$ such that
    \[\mu\paren{U_{n, j_n}} > \mu(U) - 2^{-(n + 2)}\eps.\]
    Also, by \cref{approximate_sigma01_or_pi01}, there is a $\Pi^0_1$ set $P \subseteq U$ such that
    \[\mu(P) > \mu(U) - \frac{\eps}{2}.\]
    Then define
    \[K = P \cap \bigcap_n \bigcup_{i < j_n} \cl B\paren{s_i, 2^{-n}\eps},\]
    which is a $\Pi^0_1$ set satisfying
    \[P \cap \bigcap_n U_{n, j_n} \subseteq K \subseteq U.\]
    Note that $K$ is totally bounded since for each $n$, it is contained within finitely many open balls of radius $2^{-n + 1}\eps$. Since $K$ is also closed and $X$ is a complete metric space, it follows that $K$ is compact.
    
    Finally, observe that
    \[U \setminus K \subseteq (U \setminus P) \cup \bigcup_n \paren{U \setminus U_{n, j_n}},\]
    which implies
    \[\mu(U \setminus K) \leq \mu(U \setminus P) + \sum_n \mu\paren{U \setminus U_{n, j_n}} < \frac{\eps}{2} + \sum_n 2^{-(n + 2)}\eps = \frac{\eps}{2} + \frac{\eps}{2} = \eps,\]
    as desired. Moreover, an index for $K$ can be computed from $U$, $\mu(U)$, and $\eps$.
\end{proof}

\subsection{Notions of Randomness}

We make use of four kinds of randomness when discussing effective ergodic theorems.

\subsubsection{Definitions}

\begin{defn}
    Let $(X, \mu)$ be a computable probability space, $n \geq 1$ and $A \subseteq \N$.
    \begin{itemize}
        \item A point $x \in X$ is \textbf{weakly $n$-random relative to $A$} (or \textbf{Kurtz random relative to $A$} if $n = 1$) if $x$ does not lie in any null $\Pi^0_n(A)$ set.
        \item If $A = \zero$, we just say \textbf{weakly $n$-random} (or \textbf{Kurtz random} if $n = 1$).
    \end{itemize}
\end{defn}

\begin{defn}
    Let $(X, \mu)$ be a computable probability space, $n \geq 1$, and $A \subseteq \N$.
    \begin{itemize}
        \item A \textbf{Schnorr $n$-test relative to $A$} is a sequence $\paren{U_k}_k$ of uniformly $\Sigma^0_n(A)$ sets such that $\mu\paren{U_k} \leq 2^{-k}$ for all $k$ and $\mu\paren{U_k}$ is uniformly $A^{(n - 1)}$-computable in $k$.
    
        \item A point $x \in X$ is \textbf{Schnorr $n$-random relative to $A$} if $x \notin \bigcap_k U_k$ for all Schnorr $n$-tests $\paren{U_k}_k$ relative to $A$.
        
        \item If $A = \zero$, we just say \textbf{Schnorr $n$-test} and \textbf{Schnorr $n$-random}. We also omit ``$n$-'' if $n = 1$.
    \end{itemize}
\end{defn}

\begin{defn}
    Let $(X, \mu)$ be a computable probability space, $n \geq 1$, and $A \subseteq \N$.
    \begin{itemize}
        \item An \textbf{$n$-test relative to $A$} is a sequence $\paren{U_k}_k$ of uniformly $\Sigma^0_n(A)$ sets such that $\mu\paren{U_k} \leq 2^{-k}$ for all $k$.
    
        \item A point $x \in X$ is \textbf{$n$-random relative to $A$} if $x \notin \bigcap_k U_k$ for all $n$-tests $\paren{U_k}_k$ relative to $A$.

        \item If $A = \zero$, we just say \textbf{$n$-test} and \textbf{$n$-random}. We also say \textbf{Martin-L\"of test} and \textbf{Martin-L\"of random} if $n = 1$.
    \end{itemize}
\end{defn}

\begin{defn}
    Let $(X, \mu)$ be a computable probability space, $n \geq 1$, and $A \subseteq \N$.
    \begin{itemize}
        \item An \textbf{Oberwolfach $n$-test relative to $A$} is a sequence $\paren{U_k}_k$ of uniformly $\Sigma^0_n(A)$ sets such that for some sequence $\paren{\beta_k}_k$ of uniformly $A^{(n - 1)}$-computable real numbers with $\beta := \sup_k \beta_k < \infty$, we have $\mu\paren{U_k} \leq \beta - \beta_k$ for all $k$.
    
        \item A point $x \in X$ is \textbf{Oberwolfach $n$-random relative to $A$} if $x \notin \bigcap_k U_k$ for all Oberwolfach $n$-tests $\paren{U_k}_k$ relative to $A$. 

        \item If $A = \zero$, we just say \textbf{Oberwolfach $n$-test} and \textbf{Oberwolfach $n$-random}. We also omit ``$n$-'' if $n = 1$.
    \end{itemize}
\end{defn}

\subsubsection{Basic relationships and properties}

The following lemma is straightforward to show. (Note, however, that when $n = 1$, implication (d) requires \cref{approximate_sigma01_or_pi01}.)

\begin{lem}
    In any computable probability space, for any $n \geq 1$,
    \begin{center}
        weakly $(n + 1)$-random $\myimplies{(a)}$ Oberwolfach $n$-random $\myimplies{(b)}$ $n$-random \\
        $\myimplies{(c)}$ Schnorr $n$-random $\myimplies{(d)}$ weakly $n$-random.
    \end{center}
\end{lem}

The next lemma tells us that having an explicit bound on the measures of the $\Sigma^0_n$ sets in a Schnorr $n$-test is not needed, only that the intersection is null.

\begin{lem}\label{loose_Schnorr_test}
    Let $(X, \mu)$ be a computable probability space and $n \geq 1$. Then a point $x \in X$ is Schnorr $n$-random if and only if $x \notin \bigcap_k U_k$ whenever $\paren{U_k}_k$ is a sequence of uniformly $\Sigma^0_n$ sets such that $\mu\paren{U_k}$ is uniformly $\zero^{(n - 1)}$-computable in $k$ and $\mu\paren{\bigcap_k U_k} = 0$.
\end{lem}
\begin{proof}
    $\LA$ (contrapositive). Assume $x$ is not Schnorr $n$-random. Then there is a sequence $\paren{U_k}_k$ is of uniformly $\Sigma^0_n$ sets such that $\mu\paren{U_k} \leq 2^{-k}$ for all $k$, $\mu\paren{U_k}$ is uniformly $\zero^{(n - 1)}$-computable in $k$, and $x \in \bigcap_k U_k$. Consequently, $\mu\paren{\bigcap_k U_k} = 0$.

    $\RA$ (contrapositive). Assume there is a sequence $\paren{U_k}_k$ of uniformly $\Sigma^0_n$ sets such that $\mu\paren{U_k}$ is uniformly $\zero^{(n - 1)}$-computable in $k$, $\mu\paren{\bigcap_k U_k} = 0$, and $x \in \bigcap_k U_k$. By \cref{intersection_of_sigma0n_with_computable_measure}, we may assume without loss of generality that $\paren{U_k}_k$ is nested, in which case $\mu\paren{U_k} \to 0$ as $k \to \infty$.
    
    Define $k_j = \min\{k: \mu(U_k) < 2^{-j}\}$ and $V_j = U_{k_j}$. Then $x \in \bigcap_k U_k \subseteq \bigcap_j V_j$, where $\mu\paren{V_j} < 2^{-j}$ for all $j$. Since $\mu\paren{U_k}$ is uniformly $\zero^{(n - 1)}$-computable, the function $j \mapsto k_j$ is also $\zero^{(n - 1)}$-computable. Hence $\paren{V_j}_j$ is uniformly $\Sigma^0_n$ by \cref{zero(n-1)_uniformly_sigma0n_is_uniformly_sigma0n}, so $x$ is not Schnorr $n$-random.
\end{proof}

We can also show that replacing $\Sigma^0_n$ sets in the definitions above with $\Sigma^0_1\paren{\zero^{(n - 1)}}$ sets does not change these notions of randomness. (One direction uses \cref{sigma0n(A(m))_is_sigma0(m+n)(A)} to say that uniformly $\Sigma^0_1\paren{\zero^{(n - 1)}}$ sets are uniformly $\Sigma^0_n$, and the other uses \cref{approximate_sigma0n_with_sigma01(zero(n-1))} to approximate $\Sigma^0_n$ sets from the outside using $\Sigma^0_1\paren{\zero^{(n - 1)}}$ sets.)

\begin{lem}\label{relativizing_randomness}
    Let $X$ be a computable probability space, $n \geq 1$, and $x \in X$. Then:
    \begin{enumerate}[(a)]
        \item $x$ is Schnorr $n$-random if and only if $x$ is Schnorr random relative to $\zero^{(n - 1)}$.
        \item $x$ is $n$-random if and only if $x$ is Martin-L\"of random relative to $\zero^{(n - 1)}$.
        \item $x$ is Oberwolfach $n$-random if and only if $x$ is Oberwolfach random relative to $\zero^{(n - 1)}$.
    \end{enumerate}
\end{lem}

Arguably the most robust notion of randomness is $n$-randomness because there is always a universal $n$-test.

\begin{thm}[Theorem 2.0.4 in \cite{Hoyrup2009a} relative to $\zero^{(n - 1)}$]\label{universal_n_test}
    Let $(X, \mu)$ be a computable probability space. Then there is an $n$-test $\paren{U_k}_k$ such that $x \in X$ is $n$-random if and only if $x \notin \bigcap_k U_k$.
\end{thm}

When considering Cantor space with the Lebesgue measure, $n$-random points have the following interesting property. (In order to speak of the measure of a set being $n$-random here, one must consider the interval $[0, 1]$ a computable probability space in its own right with the Lebesgue measure.)

\begin{thm}[Theorem 3.2.35 in \cite{Nies} relative to $\zero^{(n - 1)}$]\label{nonempty_pi01(zero(n-1))_subset_of_nR_has_nR_measure}
    Consider Cantor space with the Lebesgue measure. Then a nonempty $\Pi^0_1\paren{\zero^{(n - 1)}}$ set consisting only of $n$-random points has $n$-random measure.
\end{thm}

\subsection{Computable measure-preserving systems}

Ergodic theory studies statistical properties of measure-preserving systems. We now describe what these are and what it means to effectivize them.

\subsubsection{Measure-preserving systems and ergodicity}

\begin{defn}
    Let $(X, \B, \mu)$ be a probability space. A $\B$-measurable map $T: \ \subseteq \! X \to X$ is called a \textbf{measure-preserving transformation (m.p.t.)} if $\mu\paren{T^{-1}(A)} = \mu(A)$ for every $A \in \B$ and the domain of $T$ is invariant, i.e., $x \in \dom T$ implies $T(x) \in \dom T$. Then we call $(X, \B, \mu, T)$ a \textbf{measure-preserving system (m.p.s.)}.
\end{defn}

Because classical ergodic theory ignores measure-zero sets, the domain of an otherwise measure-preserving transformation can always be made invariant by restricting to the full-measure set $\bigcap_k T^{-k}(\dom T)$. However, we choose to make this a part of the definition above so that all iterates of $T$ will be defined without taking away any null sets that may be relevant when dealing with algorithmic randomness.

An important example of a measure-preserving system is Cantor space with the Lebesgue measure $\lambda$ and the shift map.

\begin{defn}
    The \textbf{shift map} is the transformation $T: 2^\N \to 2^\N$ which simply deletes the first bit of the input sequence:
    \[T\paren{x_0x_1x_2\dots} = x_1x_2\dots.\]
\end{defn}

The following property is instrumental in many results of ergodic theory. It essentially says that there are no non-trivial $T$-invariant sets, so the space cannot be meaningfully decomposed with respect to the action of $T$.

\begin{defn}
    A measure-preserving system $(X, \B, \mu, T)$ is \textbf{ergodic} if whenever a set $A \in \B$ satisfies $T^{-1}(A) = A$, it must have measure 0 or 1. We also may say that $T$ is ergodic with respect to $\mu$ or vice versa.
\end{defn}

The shift map, for example, is ergodic with respect to the Lebesgue measure on Cantor space. In fact, the shift satisfies the following stronger property, which essentially means that the history of each set under $T$ is asymptotically independent of any other set.

\begin{defn}
    A measure-preserving system $(X, \B, \mu, T)$ is \textbf{mixing} if for all $A, B \in \B$,
    \[\lim_{k \to \infty} \mu\paren{T^{-k}(A) \cap B} = \mu(A)\mu(B).\]
\end{defn}

In this paper, we consider effective versions of the following two foundational results in ergodic theory.

\begin{thm}[Poincar\'e Recurrence Theorem, Theorems 1.4 and 1.5 in \cite{Walters}]\label{poincare_recurrence_theorem}
    Let $(X, \B, \mu, T)$ be a measure-preserving system and take $A \in \B$ with $\mu(A) > 0$. Then for $\mu$-a.e.\ $x \in A$, we have
    \[T^k(x) \in A\quad\text{for some}\quad k > 0. \tag{$*$}\label{poincare}\]
    If $T$ is moreover ergodic, then (\ref{poincare}) holds for $\mu$-a.e.\ $x \in X$ (not just in $A$).
\end{thm}

By induction, it easily follows that infinitely many iterates must lie in $A$. When the system is moreover ergodic, the following theorem with $f = \one_A$ says that almost every point visits $A$ with a positive frequency.

\begin{thm}[Birkhoff Ergodic Theorem, Theorem 1.14 in \cite{Walters}]\label{Birkhoff_ergodic_theorem}
    Let $(X, \B, \mu, T)$ be a measure-preserving system and $f: X \to \cl\R$ be integrable. Then for $\mu$-a.e.\ $x \in X$,
    \[\lim_{n \to \infty}\frac{1}{n}\sum_{k < n} \paren{f \circ T^k}(x) \quad \text{exists}.\]
    If $T$ is moreover ergodic, then this limit is equal to the constant $\int f \dmu$ for $\mu$-a.e.\ $x \in X$.
\end{thm}

Closely related is the Maximal Ergodic Theorem, which is useful in proving the Birkhoff Ergodic Theorem.

\begin{thm}[Maximal Ergodic Theorem, Corollary 1.16.1 in \cite{Walters}]\label{maximal_ergodic_theorem}
    Let $(X, \B, \mu, T)$ be a measure-preserving system and $f: X \to \cl\R$ be integrable. Then for all $\alpha > 0$,
    \[\mu\set{x \in X: \sup_{n > 0} \frac{1}{n}\sum_{k < n} \paren{f \circ T^k}(x) > \alpha} \leq \frac{1}{\alpha}\int |f| \dmu.\]
\end{thm}

\subsubsection{Computable transformations}

The transformations relevant to our study will all be computable, which is a straightforward restriction of continuity for computable metric spaces.

\begin{defn}
    Let $X$ and $Y$ be computable metric spaces and $B_i$ be the $i\th$ ideal ball of $Y$. A partial transformation $T: \ \subseteq \! X \to Y$ is \textbf{computable} if the sets $T^{-1}\paren{B_i}$ are uniformly $\Sigma^0_1$ in $\dom T$.
\end{defn}

Then we can say what it means for a measure-preserving system to be computable.

\begin{defn}
    $(X, \mu, T)$ is a \textbf{computable measure-preserving system (c.m.p.s.)} if $(X, \mu)$ is a computable probability space and $T: \ \subseteq \! X \to X$ is a computable measure-preserving transformation.
\end{defn}

The shift map, for example, is a computable measure-preserving transformation on Cantor space with the Lebesgue measure.

It is simple to show from the definition that preimages under computable transformations behave well for all $\Sigma^0_n$ sets, not just ideal balls.

\begin{lem}\label{preimage_of_sigma0n}
    Let $X$ and $Y$ be computable metric spaces and $T: \ \subseteq \! X \to Y$ be computable. If $U$ is a $\Sigma^0_n$ subset of $Y$, then $T^{-1}(U)$ is $\Sigma^0_n$ in $\dom T$, uniformly given an index for $U$.
\end{lem}

\begin{cor}\label{iterated_preimage_of_sigma0n}
    Let $X$ be a computable metric space, $T: \ \subseteq \! X \to X$ be computable with $T$-invariant domain, and $n \geq 1$. If $U$ is a $\Sigma^0_n$ subset of $X$, then $T^{-k}(U)$ is $\Sigma^0_n$ in $\dom T$, uniformly given $k$ and an index for $U$.
\end{cor}

We can easily use \cref{preimage_of_sigma0n} to infer that all of the randomness notions we have considered are preserved by computable measure-preserving transformations. This allows us to go from saying that $T^k(x) \in A$ for some $k > 0$ to saying that $T^k(x) \in A$ for infinitely many $k$ when stating effective versions of the Poincar\'e Recurrence Theorem.

\begin{lem}\label{preservation_of_randomness}
    Let $(X, \mu, T)$ be a computable measure-preserving system and $n \geq 1$. Then respectively, if $x \in \dom T$ is
    \begin{center}
    		weakly $n$-random/Schnorr $n$-random/$n$-random/Oberwolfach $n$-random,
    	\end{center}
    	so is $T(x)$.
\end{lem}

Because $n$-lower-semicomputable functions are closed uniformly under $\Q_{> 0}$-linear combinations, we can also use \cref{preimage_of_sigma0n} to conclude the following result, which is helpful when studying effective versions of the Birkhoff Ergodic Theorem.

\begin{lem}\label{ergodic_averages_computable}
    Let $(X, \mu)$ be computable probability space, $T: \ \subseteq \! X \to X$ be a computable transformation, and $f: \ \subseteq \! X \to \cl\R$ be $n$-lower-semicomputable. Then the ergodic averages
    \[A_Nf := \frac{1}{N}\sum_{k < N} f \circ T^k\]
    of $f$ are likewise $n$-lower-semicomputable on $\bigcap_k T^{-k}(\dom f)$, uniformly in $N > 0$.
\end{lem}

\subsubsection{Images under the shift map}

Recall that when considering Cantor space with the Lebesgue measure, the shift map is a computable ergodic measure-preserving transformation. Thus, the preimage of any $\Sigma^0_1$ or $\Pi^0_1$ set is $\Sigma^0_1$ or $\Pi^0_1$, respectively. Sometimes, however, we need this to be true for images as well.

First we consider images of $\Sigma^0_1$ sets under the shift.

\begin{lem}\label{shift_of_sigma01}
    Consider the shift map $T: 2^\N \to 2^\N$ on Cantor space. Let $A \subseteq \N$ and $U$ be a $\Sigma^0_1(A)$ set. Then $T(U)$ is also $\Sigma^0_1(A)$, uniformly in $U$. If $U$ moreover has $A$-computable Lebesgue measure $\lambda(U)$, then so does $T(U)$, uniformly in $U$ and $\lambda(U)$.
\end{lem}
\begin{proof}
    Write $U = \bigcup_{\sigma \in W} \cyl\sigma$, where $W \subseteq 2^{<\N}$ is $A$-c.e. We may assume without loss of generality that $W$ does not contain the empty string and is \emph{prefix-free}, i.e., that whenever $\sigma, \tau \in W$ with $\sigma \subseteq \tau$, necessarily $\sigma = \tau$.
    
    For each $i \in \{0, 1\}$, let
    \[W_i = \set{\sigma \in 2^{<\N}: i\sigma \in W}\quad\text{and}\quad U_i = \bigcup_{\sigma \in W_i} \cyl{i\sigma}.\]
    Then $U_i = U \cap \cyl i$ for each $i$, so $U = U_0 \cup U_1$. Observe that
    \[T\paren{U_i} = \bigcup_{\sigma \in W_i}\cyl\sigma\]
    is $\Sigma^0_1(A)$, uniformly in $U$, because $W_i$ is $A$-c.e., uniformly in $W$. Thus,
    \[T(U) = T\paren{U_0} \cup T\paren{U_1}\]
    is also $\Sigma^0_1(A)$, uniformly in $U$.

    Now assume also that $U$ has $A$-computable measure. Since $W_i$ is prefix-free,
    \[\lambda\paren{T\paren{U_i}} = \sum_{\sigma \in W_i} 2^{-|\sigma|} = 2\sum_{\sigma \in W_i} 2^{-|i\sigma|} = 2\lambda(U_i) = 2\lambda\paren{U \cap \cyl i}\]
    is $A$-computable by \cref{intersection_of_sigma0n_with_computable_measure} relativized to $A$, and this is uniform in $U$ and $\lambda(U)$. Thus,
    \[\lambda(T(U)) = \lambda\paren{T\paren{U_0} \cup T\paren{U_1}}\]
    is also $A$-computable by \cref{intersection_of_sigma0n_with_computable_measure} relativized to $A$, and this is uniform in $U$ and $\lambda(U)$.
\end{proof}

Similarly, we consider images of $\Pi^0_1$ sets under the shift.

\begin{lem}\label{shift_of_pi01}
    Consider the shift map $T: 2^\N \to 2^\N$ on Cantor space. Let $A \subseteq \N$ and $P$ be a $\Pi^0_1(A)$ set. Then $T(P)$ is also $\Pi^0_1(A)$, uniformly in $P$. If $P$ moreover has $A$-computable Lebesgue measure $\lambda(P)$, then so does $T(P)$, uniformly in $P$ and $\lambda(P)$.
\end{lem}
\begin{proof}
    Consider any $\Pi^0_1(A)$ set $P = \bigcap_{F \in W} \cyl F$, where $W \subseteq \PP_{<\N}\paren{2^{<\N}}$ is $A$-c.e. Assume without loss of generality that no set in $W$ contains the empty string, that $W$ is closed under intersections, and that each set in $W$ is prefix-free. Given $F \subseteq 2^{<\N}$ and $i \in \{0, 1\}$, define
    \[F_i = \set{\sigma \in 2^{<\N}: i\sigma \in F}.\]
    For each $i \in \{0, 1\}$, let
    \[P_i = \bigcap_{F \in W} \bigcup_{\sigma \in F_i} \cyl{i\sigma}.\]
    Then $P_i = P \cap \cyl i$ for each $i$, so $P = P_0 \cup P_1$. Observe that
    \[T\paren{P_i} = \bigcap_{F \in W} \cyl{F_i}\]
    is $\Pi^0_1(A)$, uniformly in $P$, because $W_i := \set{F_i: F \in W}$ is $A$-c.e., uniformly in $W$. Thus,
    \[T(P) = T\paren{P_0} \cup T\paren{P_1}\]
    is also $\Pi^0_1(A)$, uniformly in $P$.

    Now assume also that $P$ has $A$-computable measure. Since $W$ is closed under intersections, so is $W_i$. Together with the fact that each set in $W$ is prefix-free, this implies that
    \[\lambda\paren{T\paren{P_i}} = \inf_{F \in W} \lambda\cyl{F_i} = \inf_{F \in W} \sum_{i\sigma \in F} 2^{-|\sigma|} = 2\inf_{F \in W} \sum_{i\sigma \in F} 2^{-|i\sigma|} = 2\lambda\paren{P_i} = 2\lambda\paren{P \cap \cyl i}\]
    is $A$-computable by \cref{intersection_of_sigma0n_with_computable_measure} relativized to $A$, and this is uniform in $P$ and $\lambda(P)$. Thus,
    \[\lambda(T(P)) = \lambda\paren{T\paren{P_0} \cup T\paren{P_1}}\]
    is also $A$-computable by \cref{intersection_of_sigma0n_with_computable_measure} relativized to $A$, and this is uniform in $P$ and $\lambda(P)$.
\end{proof}

\cref{shift_of_sigma01,relativizing_randomness} can be used to show that the shift map preserves non-randomness of points.

\begin{cor}\label{shift_preserves_non_randomness}
    Consider Cantor space with the Lebesgue measure $\lambda$ and the shift map $T: 2^\N \to 2^\N$. Then respectively, if $x \in 2^\N$ is not
    \begin{center}
    		Schnorr $n$-random/$n$-random/Oberwolfach $n$-random,
    	\end{center}
    	neither is $T(x)$.
\end{cor}

\end{document}